\documentclass[reqno, 8pt]{amsart}

\usepackage{amsmath,amssymb,amsthm,amsfonts,mathrsfs}
\usepackage{fullpage}
\usepackage{xcolor}
\usepackage{graphicx}
\usepackage{makecell}
\usepackage{listings}
\usepackage{bbm}

\definecolor{codegreen}{rgb}{0,0.6,0}
\definecolor{codegray}{rgb}{0.5,0.5,0.5}
\definecolor{codepurple}{rgb}{0.58,0,0.82}
\definecolor{backcolour}{rgb}{0.95,0.95,0.92}

\lstdefinestyle{mystyle}{
  backgroundcolor=\color{backcolour},   commentstyle=\color{codegreen},
  keywordstyle=\color{magenta},
  numberstyle=\tiny\color{codegray},
  stringstyle=\color{codepurple},
  basicstyle=\ttfamily\footnotesize,
  breakatwhitespace=false,         
  breaklines=true,                 
  captionpos=b,                    
  keepspaces=true,                 
  numbers=left,                    
  numbersep=5pt,                  
  showspaces=false,                
  showstringspaces=false,
  showtabs=false,                  
  tabsize=2
}

\usepackage{listings}
\usepackage[colorlinks=true,
linkcolor=blue,
anchorcolor=blue,
citecolor=red
]{hyperref}
\allowdisplaybreaks 
\newtheorem{theorem1number}{Theorem}
\newtheorem{theorem}{Theorem}[section]
\newtheorem{lemma}[theorem]{Lemma}

\newtheorem*{conjecture*}{Conjecture}
\newtheorem*{theorem*}{Theorem}
\theoremstyle{definition}

\theoremstyle{remark}
\newtheorem{remark}[theorem]{Remark}
\newtheorem*{remark*}{remark}

\usepackage{colonequals}

\author{\textbf{Runbo Li}}
\address{International Curriculum Center, The High School Affiliated to Renmin University of China, Beijing, China}
\email{runbo.li.carey@gmail.com}

\makeatletter
\@namedef{subjclassname@2020}{\textup{}2020 Mathematics Subject Classification}
\makeatother

\title[]{The number of primes in short intervals and numerical calculations for Harman's sieve}

\begin{document}

\begin{abstract}
The author gives nontrivial upper and lower bounds for the number of primes in the interval $[x - x^{\theta}, x]$ for some $0.52 \leqslant \theta \leqslant 0.525$, showing that the interval $[x - x^{0.52}, x]$ contains prime numbers for all sufficiently large $x$. This refines a result of Baker, Harman and Pintz (2001) and gives an affirmative answer to Harman and Pintz's argument. New arithmetic information, a delicate sieve decomposition, various techniques in Harman's sieve and accurate estimates for integrals are used to good effect.
\end{abstract}

\maketitle

\tableofcontents

\section{Introduction}
One of the famous topics in number theory is to find prime numbers in short intervals. In 1937, Cramér \cite{Cramer1937} conjectured that every interval $[x-f(x) (\log x)^2, x]$ contains prime numbers for some $f(x) \to 1$ as $x \to \infty$. The Riemann Hypothesis implies that for all sufficiently large $x$, the interval $[x-x^{\theta}, x]$ contains $\sim x^{\theta} (\log x)^{-1}$ prime numbers for every $\frac{1}{2} < \theta \leqslant 1$. The first unconditional result of this asymptotic formula with some $\theta < 1$ was proved by Hoheisel \cite{Hoheisel} in 1930 with $\theta \geqslant 1-\frac{1}{33000}$. After the works of Hoheisel \cite{Hoheisel}, Heilbronn \cite{Heilbronn}, Chudakov \cite{Chudakov}, Ingham \cite{Ingham} and Montgomery \cite{Montgomery}, Huxley \cite{Huxley} proved in 1972 that the above asymptotic formula holds when $\theta>\frac{7}{12}$ by his zero density estimate. In 2024, Guth and Maynard \cite{GuthMaynard} improved this to $\theta>\frac{17}{30}$ by a new zero density estimate.

In 1979, Iwaniec and Jutila \cite{IwaniecJutila} first introduced a sieve method into this problem. They established a lower bound with correct order of magnitude (instead of an asymptotic formula) with $\theta=\frac{13}{23}$. After that breakthrough, many improvements were made and the value of $\theta$ was reduced successively to
$$
\frac{5}{9} \approx 0.5556,\ \frac{11}{20}=0.5500,\ \frac{17}{31} \approx 0.5484,\ \frac{23}{42} \approx 0.5476, 
$$
$$
\frac{1051}{1920} \approx 0.5474,\ \frac{35}{64} \approx 0.5469,\ \frac{6}{11} \approx 0.5455\ \text{ and }\ \frac{7}{13} \approx 0.5385
$$
by Iwaniec and Jutila \cite{IwaniecJutila}, Heath-Brown and Iwaniec \cite{HeathBrownIwaniec}, Pintz \cite{Pintz1} \cite{Pintz2}, Iwaniec and Pintz \cite{IwaniecPintz}, Mozzochi \cite{Mozzochi} and Lou and Yao \cite{LouYao3564} \cite{LouYao} \cite{LouYao2} \cite{LouYao3} respectively.

In 1996, Baker and Harman \cite{BakerHarman} presented an alternative approach to this problem. They used the alternative sieve developed by Harman \cite{Harman1} \cite{Harman2} to reduce $\theta$ to $0.535$. 
Finally, Baker, Harman and Pintz (BHP) \cite{BHP} further developed this sieve process and combined it with Watt's power mean value theorem on Dirichlet polynomials \cite{WattTheorem} and showed $\theta \geqslant 0.525$. As Friedlander and Iwaniec mentioned in their book [\cite{Opera}, Chapter 23], ``their method uses many powerful tools and arguments, both analytic and combinatorial, and these are extremely complicated.'' However, they omitted almost all calculation details in \cite{BakerHarman} and \cite{BHP}, which makes the papers very hard to read and check. 
In 2014, Pintz \cite{PintzGoldbachSejtesrol} pointed out that ``the Baker--Harman--Pintz result with $\theta=0.525$ actually leads to a slightly better value'' in his lecture. Harman [\cite{HarmanBOOK}, Chapter 7.10] also mentioned that $\theta=0.52$ might be achievable with an incredibly long and boring argument. In a personal communication, Kumchev mentioned that BHP tried and discovered that $0.52$ was out of reach of the existing techniques.
In 2024, Starichkova \cite{Starichkova} provided full details of \cite{BakerHarman} in her PhD thesis, so we turn our attention to \cite{BHP}. In this paper, we provide the calculation details and sharpen the main theorem proved in \cite{BHP}.
\begin{theorem1number}\label{t1}
For all sufficiently large $x$, the interval $[x - x^{0.52}, x]$ contains prime numbers.
\end{theorem1number}
Theorem~\ref{t1} is a direct corollary of the following result, which gives nontrivial upper and lower bounds with correct order of magnitude for the number of primes in intervals of length between $x^{0.52}$ and $x^{0.525}$. Here we say a trivial upper bound is the bound with upper constant $\frac{2}{\theta}$ obtained by Montgomery and Vaughan \cite{MontgomeryVaughanLargeSieve}, and a trivial lower bound is of course zero. Note that in \cite{IwaniecBT}, \cite{LouYaoUpperBound}, \cite{LouYao} and \cite{BHP2} nontrivial upper bounds for the number of primes in other intervals are also given.
\begin{theorem1number}\label{t1B}
Let $0.52 \leqslant \theta \leqslant 0.525$ and $\varepsilon >0$. Then we have
$$
\mathbf{LB}(\theta) \frac{x^{\theta + \varepsilon}}{\log x} \leqslant \pi(x)-\pi(x-x^{\theta + \varepsilon}) \leqslant \mathbf{UB}(\theta) \frac{x^{\theta + \varepsilon}}{\log x}
$$
for all sufficiently large $x$, where the values of functions $\mathbf{LB}(\theta)$ and $\mathbf{UB}(\theta)$ satisfy the following condition table.

\begin{center}
\begin{tabular}{|c|c|c|}
\hline \boldmath{$\theta$} & \boldmath{$\operatorname{LB}(\theta)$} & \boldmath{$\operatorname{UB}(\theta)$} \\
\hline \boldmath{$0.520$} & $>0.004$ & $<2.874$  \\
\hline \boldmath{$0.521$} & $>0.075$ & $<2.700$  \\
\hline \boldmath{$0.522$} & $>0.134$ & $<2.583$  \\
\hline \boldmath{$0.523$} & $>0.169$ & $<2.536$  \\
\hline \boldmath{$0.524$} & $>0.209$ & $<2.437$  \\
\hline \boldmath{$0.525$} & $>0.249$ & $<2.347$  \\
\hline
\end{tabular}
\end{center}
\end{theorem1number}
Obviously, our result confirms Harman and Pintz's argument and goes beyond the $0.52$-barrier mentioned by Kumchev. Although our upper constant for $\theta = 0.52$ is a little bit weaker than Iwaniec's (see Section 6 of \cite{IwaniecBT}), our result comes from Harman's sieve which leads to much better results for intervals longer than $x^{0.522}$. Our improvement on the lower bounds comes mainly from the following 2 aspects:

\textbf{1.} We have a careful discussion on the original sieve decomposing process of BHP and optimize some of their arguments. Specifically, the most important optimization we do is that we perform a role-reversal after a Buchstab iteration for some four-dimensional sums that are not decomposed in BHP's original arguments, replacing a larger 4D loss by a much smaller 6D loss. We also consider 8D losses after more iteration steps.

\textbf{2.} We prove some new arithmetic information outside of those in \cite{BHP} and \cite{HarmanBOOK} which gives room for improvement, see Lemmas~\ref{l34}--\ref{l35}. We also find that Lemma 18 of \cite{BHP} and Lemma 7.22 of \cite{HarmanBOOK} actually cover some non-overlapping three-dimensional regions when $\theta \geqslant 0.52$, so using them simultaneously yields a better result.

All the numerical values of the integrals in Sections 5 and 6 are calculated using C$++$, and we use Mathematica 14 to calculate them again for cross-checking. We use an Intel(R) Xeon(R) Platinum 8383C CPU with 160 threads (80 Wolfram kernels) to run the code. The C$++$ code can be found in the ancillary files, and the websites for Mathematica code can be found in Table 6, Appendix 2.

Throughout this paper, we always assume that $\varepsilon$ is a sufficiently small positive constant, $x$ is a sufficiently large integer and $K > 0$. The appearance of $K$ in the exponent of a logarithm will always signify that the result holds for every $K > 0$ with an implied constant that depends on $K$. Let $\theta$ be a positive number which will be fixed later. The letter $p$, with or without subscript, is reserved for prime numbers. We write $m \sim M$ to mean that $M \leqslant m < 2M$. We use $M(s)$, $N(s)$, $R(s)$ and some other capital letters (with or without subscript) to denote some divisor-bounded Dirichlet polynomials
$$
M(s)=\sum_{m \sim M} a_{m} m^{-s}, \quad N(s)=\sum_{n \sim N} b_{n} n^{-s}, \quad R(s)=\sum_{r \sim R} c_{r} r^{-s}.
$$
We say a Dirichlet polynomial $M(s)$ is \textit{prime-factored} if we have
$$
\left|M \left( \frac{1}{2}+it \right)\right| \ll M^{\frac{1}{2}} (\log x)^{-K}
$$
for $\exp \left((\log x)^{1/3} \right) < |t| < x^{1-\theta+\varepsilon}$. In fact, this holds when $a_m$ is the characteristic function of primes or of numbers with a bounded number of prime factors restricted to certain ranges. For example, if we have
$$
a_m = \sum_{m = p_1 \cdots p_k} 1
$$
and the least prime factor of $m$ is $\gg \exp \left((\log x)^{4/5} \right)$, then we know that $M(s)$ is prime-factored by [\cite{HarmanBOOK}, Lemma 1.5]. We also say a Dirichlet polynomial is \textit{decomposable} if it can be written as the form
$$
\sum_{p_{i} \sim P_{i}}\left(p_{1} \ldots p_{u}\right)^{-s}.
$$
That is, we can decompose this Dirichlet polynomial.

\section{An outline of the proof}
Let $0.505 \leqslant \theta \leqslant 0.535$, $\mathcal{C}$ denote a finite set of positive integers, $y=x^{\theta+\varepsilon}$, $y_{1}=x \exp \left(-3 (\log x)^{1/3} \right)$,
$$
\mathcal{A}=\{a: a \in \mathbb{Z},\ x-y \leqslant a < x\}, \quad \mathcal{B}=\{b: b \in \mathbb{Z},\ x-y_{1} \leqslant b < x\},
$$
$$
\mathcal{C}_d=\{a: a d \in \mathcal{C}\}, \quad P(z)=\prod_{p<z} p, \quad S\left(\mathcal{C}, z\right)=\sum_{\substack{a \in \mathcal{C} \\ (a, P(z))=1}} 1.
$$
Then we have
\begin{equation}
\pi(x)-\pi(x-y)=S\left(\mathcal{A}, x^{\frac{1}{2}}\right).
\end{equation}

\textit{Buchstab's identity} is the equation
$$
S\left(\mathcal{C}, z\right) = S\left(\mathcal{C}, w\right) - \sum_{w \leqslant p < z} S\left(\mathcal{C}_{p}, p\right),
$$
where $2 \leqslant w < z$.

In order to prove Theorem~\ref{t1B}, we only need to give upper and lower bounds for $S\left(\mathcal{A}, x^{\frac{1}{2}}\right)$. Our aim is to show that the sparser set $\mathcal{A}$ contains the expected proportion of primes compared to the larger set $\mathcal{B}$, which requires us to decompose $S\left(\mathcal{A}, x^{\frac{1}{2}}\right)$ using the above Buchstab's identity, prove asymptotic formulas of the form
\begin{equation}
S\left(\mathcal{A}, z\right) = \frac{y}{y_1} (1+o(1)) S\left(\mathcal{B}, z\right)
\end{equation}
for some parts of it, and drop the remaining parts. The dropped parts must be positive in the lower bound case and must be negative in the upper bound case. We say a term $S\left(\mathcal{A}, z\right)$ \textit{has an asymptotic formula} if (2) holds for this term.

In Section 3 we provide asymptotic formulas for terms of the form $S\left(\mathcal{A}_{p_{1} \ldots p_{n}}, x^{\nu}\right)$ (which requires both Type-I and Type-II information) and in Section 4 we provide asymptotic formulas for terms of the form $S\left(\mathcal{A}_{p_{1} \ldots p_{n}}, p_{n} \right)$ (which only requires Type-II information). In Sections 5 and 6 we will make further use of Buchstab's identity to decompose $S\left(\mathcal{A}, x^{\frac{1}{2}}\right)$ and prove Theorem~\ref{t1B} for $\theta = 0.52$. We omit the proof of numerical bounds for other values of $\theta$ for the sake of simplicity.

\section{Sieve asymptotic formulas I}
Now we follow \cite{BHP} directly to get some sieve asymptotic formulas. For a positive integer $h$, we define the interval
\begin{equation}
I_{h}=\left[\frac{1}{2}-2 h\left(\theta-\frac{1}{2}\right),\ \frac{1}{2}-(2 h-2)\left(\theta-\frac{1}{2}\right)\right),
\end{equation}
and we define the piecewise-linear function $\nu=\nu(\alpha)$ and the function $\alpha^{*} = \alpha^{*}(\alpha)$, $0 \leqslant \alpha \leqslant \frac{1}{2}$ as follows: if $\alpha \in I_{h}$ then
\begin{equation}
\nu(\alpha)= \min \left(\frac{2 (\theta-\alpha)}{2 h-1},\ \gamma(\theta) \right) \text{ for } h \geqslant 1,
\end{equation}
\begin{equation}
\alpha^{*}=\max \left(\frac{2 h(1-\theta)-\alpha}{2 h-1},\ \frac{2(h-1) \theta+\alpha}{2 h-1}\right) ,
\end{equation}
where $\gamma(\theta)$ will be defined in next section. Note that we have $\nu(\alpha) \geqslant 2 \theta -1$ and $1-\theta \leqslant \alpha^{*} \leqslant \frac{1}{2}+\varepsilon$.

Now we provide some lemmas which will be used to give asymptotic formulas for sieve functions of the form $S\left(\mathcal{A}_{p_{1} \ldots p_{n}}, x^{\nu}\right)$. The next two lemmas can be deduced from [\cite{HarmanBOOK}, Lemma 7.3], some combinatorial lemmas together with [\cite{BHP}, Lemma 2] which is a generalized version of Watt's theorem \cite{WattTheorem}.

\begin{lemma}\label{l21} ([\cite{BHP}, Lemma 12], [\cite{HarmanBOOK}, Lemma 7.15]).
Let $M=x^{\alpha_1}, N=x^{\alpha_2}$ where $M(s)$ and $N(s)$ are decomposable. Suppose that $\alpha_1 \leqslant \frac{1}{2}$ and
$$
\alpha_2 \leqslant \min \left(\frac{3 \theta+1-4 \alpha_1^{*}}{2},\ \frac{3+\theta-4 \alpha_1^{*}}{5}\right)-2 \varepsilon.
$$
Then
$$
\sum_{\substack{m \sim M \\ n \sim N}} a_{m} b_{n} S\left(\mathcal{A}_{m n}, x^{\nu}\right)=\frac{y}{y_{1}}(1+o(1)) \sum_{\substack{m \sim M \\ n \sim N}} a_{m} b_{n} S\left(\mathcal{B}_{m n}, x^{\nu}\right)
$$
holds for every $\nu \leqslant \nu(\alpha_1)$.
\end{lemma}

\begin{lemma}\label{l22} ([\cite{BHP}, Lemma 13], [\cite{HarmanBOOK}, Lemma 7.16]).
Let $M=x^{\alpha_1}, N_{1}=x^{\alpha_2}, N_{2}=x^{\alpha_3}$ where $M(s), N_{1}(s)$ and $N_{2}(s)$ are decomposable. Suppose that $\alpha_1 \leqslant \frac{1}{2}$ and either
$$
2 \alpha_2+\alpha_3 \leqslant 1+\theta-2 \alpha_1^{*}-2 \varepsilon, \quad \alpha_3 \leqslant \frac{1+3 \theta}{4}-\alpha_1^{*}-\varepsilon, \quad 
2 \alpha_2+3 \alpha_3 \leqslant \frac{3+\theta}{2}-2 \alpha_1^{*}-2 \varepsilon
$$
or
$$
\alpha_2 \leqslant \frac{1-\theta}{2}, \quad \alpha_3 \leqslant \frac{1+3 \theta - 4 \alpha_1^{*}}{8} -\varepsilon.
$$
Then
$$
\sum_{\substack{m \sim M \\ n_{1} \sim N_{1} \\ n_{2} \sim N_{2} }} a_{m} b_{n_{1}} c_{n_{2}} S\left(\mathcal{A}_{m n_1 n_2}, x^{\nu}\right)=\frac{y}{y_{1}}(1+o(1)) \sum_{\substack{m \sim M \\ n_{1} \sim N_{1} \\ n_{2} \sim N_{2} }} a_{m} b_{n_{1}} c_{n_{2}} S\left(\mathcal{B}_{m n_1 n_2}, x^{\nu}\right)
$$
holds for every $\nu \leqslant \nu(\alpha_1)$.
\end{lemma}

The next lemma is obtained by a two-dimensional sieve together with Lemma~\ref{l21}, and they will help us deal with the regions $A_2$ and $A_2^{\prime}$ in Section 6.

\begin{lemma}\label{l23} ([\cite{BHP}, Lemma 16], [\cite{HarmanBOOK}, Lemma 7.19]).
Let $M_1=x^{\alpha_1}, M_2=x^{\alpha_2}$. Suppose that
$$
\alpha_2 \leqslant \alpha_1,\ 2 \alpha_1 + \alpha_2 < 1 \text{ and } \alpha_2 < \frac{7}{2}\theta - \frac{3}{2}.
$$
Then
$$
\sum_{\substack{p_{1} \sim M_1 \\ p_{2} \sim M_2}} S\left(\mathcal{A}_{p_{1} p_{2}}, x^{\nu}\right)=\frac{y}{y_{1}}(1+o(1)) \sum_{\substack{p_{1} \sim M_1 \\ p_{2} \sim M_2}} S\left(\mathcal{B}_{p_{1} p_{2}}, x^{\nu}\right)
$$
holds for $\nu=2 \theta-1$.
\end{lemma}
\begin{remark}
For $\theta > \frac{11}{21} \approx 0.5238$, the third condition in Lemma~\ref{l23} can be simplified to $\alpha_2 < \frac{1}{3}$.
\end{remark}
\begin{remark}
One may use existing results on higher power moments of zeta function to get a minor improvement. For example, it is possible to use Heath-Brown's twelfth power moment \cite{HB12} together with Hölder's inequality (or results in [\cite{TYexponentpairs}, Section 2.1]) and mean value theorem to get an improvement on [\cite{HarmanBOOK}, Lemmas 7.9 and 7.10] which are essential in proving Lemma~\ref{l22}. Here we don't consider about them for the sake of simplicity.
\end{remark}

\section{Sieve asymptotic formulas II}
In this section we give asymptotic formulas for sieve functions of the form $S\left(\mathcal{A}_{p_{1} \ldots p_{n}}, p_{n}\right)$ or more general sums. These lemmas can be deduced from [\cite{BHP}, Lemma 6] together with some mean and large value theorems of Dirichlet polynomials.

\begin{lemma}\label{l31} ([\cite{BHP}, Lemma 9], [\cite{HarmanBOOK}, Lemma 7.3]).
Let $M=x^{\alpha_{1}}, N=x^{\alpha_{2}}$ with
$$
\left|\alpha_{1}-\alpha_{2}\right|<2 \theta-1, \quad \alpha_{1}+\alpha_{2}>1-\gamma(\theta)
$$
where
$$
\gamma(\theta)=\max _{g \in \mathbb{N}} \gamma_{g}(\theta),
$$
$$
\gamma_{g}(\theta)=\min \left(4 \theta-2,\ \frac{(8 g-4) \theta-(4 g-3)}{4 g-1},\ \frac{24 g \theta-(12 g+1)}{4 g-1}\right). 
$$
Moreover, let $R=x^{1-\alpha_1 -\alpha_2}$ and suppose that $R(s)$ is prime-factored. Then we can obtain an asymptotic formula for
$$
\sum_{\substack{m n r \in \mathcal{A} \\ m \sim M \\ n \sim N }} a_m b_n c_r \quad \text{and thus for} \quad \sum_{\substack{p_{j} \sim x^{\alpha_{j}} \\ 1 \leqslant j \leqslant 2}} S\left(\mathcal{A}_{p_{1} p_{2}}, p_{2}\right).
$$
Note that the dependencies between variables in the sieve functions (here and many others below) can be removed using a truncated Perron's formula as in [\cite{BakerHarmanRivat}, Lemma 11]. Moreover, we have 
\begin{align*}
\gamma(\theta) =
\begin{cases}
\gamma_{6}(\theta) = 4 \theta -2, & \quad 0.52 \leqslant \theta < \frac{25}{48} \approx 0.5208, \\
\gamma_{6}(\theta) = \frac{44 \theta -21}{23}, & \quad \frac{25}{48} \leqslant \theta < \frac{251}{481} \approx 0.5218, \\
\gamma_{5}(\theta) = \frac{120 \theta -61}{19}, & \quad \frac{251}{481} \leqslant \theta < \frac{23}{44} \approx 0.5227, \\
\gamma_{5}(\theta) = 4 \theta -2, & \quad \frac{23}{44} \leqslant \theta < 0.525. \\
\end{cases}
\end{align*}
\end{lemma}

\begin{lemma}\label{l32}
Let $L_{1} L_{2} L_{3} L_{4}=x, L_{j}=x^{\alpha_{j}}, \alpha_{j} \geqslant \varepsilon$ and suppose that $L_{j}(s)$ is prime-factored for $j \geqslant 2$. If any of the following conditions hold:
\begin{align}
\nonumber & \alpha_{1} \geqslant 1-\theta,\ \alpha_{2} \geqslant \frac{(1-\theta)}{2},\ \alpha_{3} \geqslant \frac{(1-\theta)}{4},\ 1-\alpha_{1}-\alpha_2 -\alpha_3 \geqslant \frac{2(1-\theta)}{7}; \\
\nonumber & \alpha_{1} \geqslant 1-\theta,\ \alpha_{2} \geqslant \frac{(1-\theta)}{2},\ \alpha_{3} \geqslant \frac{(1-\theta)}{3},\ 1-\alpha_{1}-\alpha_2 -\alpha_3 \geqslant \frac{2(1-\theta)}{11}; \\
\nonumber & \alpha_{1} \geqslant 1-\theta,\ \alpha_{2} \geqslant \frac{(1-\theta)}{3},\ 1-\alpha_{1}-\alpha_2 +\alpha_3 \geqslant 1-\theta,\ 1-\alpha_{1}-\alpha_2 -\alpha_3 \geqslant \frac{2(1-\theta)}{5}; \\
\nonumber & \alpha_{1} \geqslant 1-\theta,\ \alpha_{2} \leqslant \frac{(1-\theta)}{3},\ \alpha_{3} \leqslant \frac{(1-\theta)}{3},\ \alpha_{2}+\alpha_{3} \geqslant \frac{4(1-\theta)}{7},\ 1-\alpha_{1} \geqslant \frac{14(1-\theta)}{13}.
\end{align}
Then we can obtain an asymptotic formula for
$$
\sum_{\substack{l_1 l_2 l_3 l_4 \in \mathcal{A} \\ l_j \sim L_j \\ 1 \leqslant j \leqslant 3 }} a_{l_1} b_{l_2} c_{l_3} d_{l_4} \quad \text{and thus for} \quad \sum_{\substack{p_{j} \sim x^{\alpha_{j}} \\ 1 \leqslant j \leqslant 3}} S\left(\mathcal{A}_{p_{1} p_{2} p_{3}}, p_{3}\right).
$$
\end{lemma}
\begin{proof}
The proof is completely same as the proof of [\cite{HarmanBOOK}, Lemma 7.22] except for the second case. For the second case, we need to make a small modification by choosing $h = 5$ instead of $h = 3$ in [\cite{HarmanBOOK}, Lemma 7.21]. Note that this modification also occurred in \cite{HarmanKumchevLewis} with $\theta = 0.53$.
\end{proof}

\begin{lemma}\label{l33}
Let $L_{j}$ satisfies the conditions in Lemma~\ref{l32}. Then we can obtain an asymptotic formula for
$$
\sum_{\substack{l_1 l_2 l_3 l_4 \in \mathcal{A} \\ l_j \sim L_j \\ 1 \leqslant j \leqslant 3 }} a_{l_1} b_{l_2} c_{l_3} d_{l_4} \quad \text{and thus for} \quad \sum_{\substack{p_{j} \sim x^{\alpha_{j}} \\ 1 \leqslant j \leqslant 3}} S\left(\mathcal{A}_{p_{1} p_{2} p_{3}}, p_{3}\right)
$$
if the following conditions hold:
\begin{align}
\nonumber & 1-\alpha_{1}-\alpha_{2}-\alpha_{3} \geqslant \frac{1}{g_4}(1-\theta), \\
\nonumber & \alpha_{2}\left(\frac{1}{4} h+\frac{1}{2} g_{2} b_{1}\right)+\alpha_{3}\left(-\frac{1}{2} g_{3} c_{1}+\frac{1}{4} h-\frac{h k_{1}}{4 g_{1}}+\frac{1}{2} k_{2} b_{1}\right) >b_{1}(1-\theta), \\
\nonumber & \alpha_{1}\left(\frac{1}{4} h+\frac{1}{2} g_{1} a_{2}\right)+\alpha_{3}\left(-\frac{1}{2} g_{3} c_{2}+\frac{1}{4} h-\frac{h k_{2}}{4 g_{2}}+\frac{1}{2} k_{1} a_{2}\right) >a_{2}(1-\theta), \\
\nonumber & \alpha_{2}\left(\frac{1}{4} h+\frac{1}{2} g_{2} b_{3}\right)+\alpha_{3}\left(\frac{1}{2} g_{3} c_{3}+\frac{1}{4} h-\frac{h k_{1}}{4 g_{1}}+\frac{1}{2} k_{2} b_{3}\right) >\left(u-\frac{h}{2 g_{1}}\right)(1-\theta), \\
\nonumber & \alpha_{1}\left(\frac{1}{4} h+\frac{1}{2} g_{1} a_{4}\right)+\alpha_{3}\left(\frac{1}{2} g_{3} c_{4}+\frac{1}{4} h-\frac{h k_{2}}{4 g_{2}}+\frac{1}{2} k_{1} a_{4}\right) >\left(u-\frac{h}{2 g_{2}}\right)(1-\theta), \\
\nonumber & \alpha_{1}\left(\frac{1}{4} h+\frac{1}{2} g_{1} a_{5}\right)+\alpha_{2}\left(\frac{1}{4} h+\frac{1}{2} g_{2} b_{5}\right)+\alpha_{3}\left(-\frac{1}{2} g_{3} c_{5}+\frac{1}{4} h+\frac{1}{2} k_{1} a_{5}+\frac{1}{2} k_{2} b_{5}\right) >\left(a_{5}+b_{5}\right)(1-\theta), \\
\nonumber & \alpha_{1}\left(\frac{1}{4} h+\frac{1}{2} g_{1} a_{6}\right)+\alpha_{2}\left(\frac{1}{4} h+\frac{1}{2} g_{2} b_{6}\right)+\alpha_{3}\left(\frac{1}{2} g_{3} c_{6}+\frac{1}{4} h+\frac{1}{2} k_{1} a_{6}+\frac{1}{2} k_{2} b_{6}\right) > u(1-\theta), \\
\nonumber & \alpha_{3}\left(\frac{1}{2} g_{3}\left(u-\frac{h}{2 g_{1}}-\frac{h}{2 g_{2}}\right)+\frac{1}{4} h v\right) > \left(u-\frac{h}{2 g_{1}}-\frac{h}{2 g_{2}}\right)(1-\theta),
\end{align}
where $h=1$, $g_1 =1$, $g_2 =2$, $g_3 =3$, $g_4 =d$, $d=4\text{ or }5$, $k_1 =k_2 =0$, $u=1-\frac{1}{2d}$, $v=1$,
\begin{align}
\nonumber & \left(b_{1}, c_{1}\right)=\left(\frac{1}{3}-\frac{1}{2 d}, \frac{1}{6}\right), \quad \left(a_{2}, c_{2}\right)=\left(\frac{7}{12}-\frac{1}{2 d}, \frac{1}{6}\right), \\
\nonumber & \left(b_{3}, c_{3}\right)= \left(\frac{1}{3}-\frac{1}{2 d}, \frac{1}{6}\right) \text{ or } \left(\frac{1}{4}, \frac{1}{4}-\frac{1}{2 d}\right), \\
\nonumber & \left(a_{4}, c_{4}\right)= \left(\frac{1}{2}, \frac{1}{4}-\frac{1}{2 d}\right) \text{ or } \left(\frac{7}{12}-\frac{1}{2 d}, \frac{1}{6}\right), \\
\nonumber & \left(a_{5}, b_{5}, c_{5}\right)= \left(\frac{1}{2}, \frac{1}{3}-\frac{1}{2 d}, \frac{1}{6}\right) \text{ or } \left(\frac{7}{12}-\frac{1}{2 d}, \frac{1}{4}, \frac{1}{6}\right), \\
\nonumber & \left(a_{6}, b_{6}, c_{6}\right)= \left(\frac{1}{2}, \frac{1}{4}, \frac{1}{4}-\frac{1}{2 d}\right) \text{ or } \left(\frac{7}{12}-\frac{1}{2 d}, \frac{1}{4}, \frac{1}{6}\right) \text{ or } \left(\frac{1}{2}, \frac{1}{3}-\frac{1}{2 d}, \frac{1}{6}\right).
\end{align}
\end{lemma}
\begin{proof}
This is a special case of [\cite{BHP}, Lemma 18].
\end{proof}

The next two lemmas are our new arithmetic information, which can be seen as generalizations of [\cite{HarmanBOOK}, Lemma 7.22]. These can be used to estimate ``Type-II$_5$'' and ``Type-II$_6$'' sums mentioned in \cite{HarmanErato}. The proof is similar to the proof of first and second case of Lemma~\ref{l32}. One can generalize these to ``Type-II$_n$'' sums with $n \geqslant 7$, but the corresponding results will be very complicated and not very numerically significant since the contribution of those high-dimensional sums is already quite small.
\begin{lemma}\label{l34}
Let $L_{1} L_{2} L_{3} L_{4} L_{5}=x, L_{j}=x^{\alpha_{j}}, \alpha_{j} \geqslant \varepsilon$ and suppose that $L_{j}(s)$ is prime-factored for $j \geqslant 2$. If any of the following 9 conditions hold:
\begin{align}
\nonumber & \alpha_{1} \geqslant 1-\theta,\ \alpha_{2} \geqslant \frac{(1-\theta)}{2},\ \alpha_{3} \geqslant \frac{(1-\theta)}{3},\ \alpha_{4} \geqslant \frac{(1-\theta)}{7},\ 1-\alpha_{1}-\alpha_2 -\alpha_3 -\alpha_4 \geqslant \frac{2(1-\theta)}{83}; \\
\nonumber & \alpha_{1} \geqslant 1-\theta,\ \alpha_{2} \geqslant \frac{(1-\theta)}{2},\ \alpha_{3} \geqslant \frac{(1-\theta)}{3},\ \alpha_{4} \geqslant \frac{(1-\theta)}{8},\ 1-\alpha_{1}-\alpha_2 -\alpha_3 -\alpha_4 \geqslant \frac{2(1-\theta)}{47}; \\
\nonumber & \alpha_{1} \geqslant 1-\theta,\ \alpha_{2} \geqslant \frac{(1-\theta)}{2},\ \alpha_{3} \geqslant \frac{(1-\theta)}{3},\ \alpha_{4} \geqslant \frac{(1-\theta)}{9},\ 1-\alpha_{1}-\alpha_2 -\alpha_3 -\alpha_4 \geqslant \frac{2(1-\theta)}{35}; \\
\nonumber & \alpha_{1} \geqslant 1-\theta,\ \alpha_{2} \geqslant \frac{(1-\theta)}{2},\ \alpha_{3} \geqslant \frac{(1-\theta)}{3},\ \alpha_{4} \geqslant \frac{(1-\theta)}{10},\ 1-\alpha_{1}-\alpha_2 -\alpha_3 -\alpha_4 \geqslant \frac{2(1-\theta)}{29}; \\
\nonumber & \alpha_{1} \geqslant 1-\theta,\ \alpha_{2} \geqslant \frac{(1-\theta)}{2},\ \alpha_{3} \geqslant \frac{(1-\theta)}{3},\ \alpha_{4} \geqslant \frac{(1-\theta)}{12},\ 1-\alpha_{1}-\alpha_2 -\alpha_3 -\alpha_4 \geqslant \frac{2(1-\theta)}{23}; \\
\nonumber & \alpha_{1} \geqslant 1-\theta,\ \alpha_{2} \geqslant \frac{(1-\theta)}{2},\ \alpha_{3} \geqslant \frac{(1-\theta)}{4},\ \alpha_{4} \geqslant \frac{(1-\theta)}{5},\ 1-\alpha_{1}-\alpha_2 -\alpha_3 -\alpha_4 \geqslant \frac{2(1-\theta)}{39}; \\
\nonumber & \alpha_{1} \geqslant 1-\theta,\ \alpha_{2} \geqslant \frac{(1-\theta)}{2},\ \alpha_{3} \geqslant \frac{(1-\theta)}{4},\ \alpha_{4} \geqslant \frac{(1-\theta)}{6},\ 1-\alpha_{1}-\alpha_2 -\alpha_3 -\alpha_4 \geqslant \frac{2(1-\theta)}{23}; \\
\nonumber & \alpha_{1} \geqslant 1-\theta,\ \alpha_{2} \geqslant \frac{(1-\theta)}{2},\ \alpha_{3} \geqslant \frac{(1-\theta)}{4},\ \alpha_{4} \geqslant \frac{(1-\theta)}{8},\ 1-\alpha_{1}-\alpha_2 -\alpha_3 -\alpha_4 \geqslant \frac{2(1-\theta)}{15}; \\
\nonumber & \alpha_{1} \geqslant 1-\theta,\ \alpha_{2} \geqslant \frac{(1-\theta)}{2},\ \alpha_{3} \geqslant \frac{(1-\theta)}{5},\ \alpha_{4} \geqslant \frac{(1-\theta)}{5},\ 1-\alpha_{1}-\alpha_2 -\alpha_3 -\alpha_4 \geqslant \frac{2(1-\theta)}{19}.
\end{align}
Then we can obtain an asymptotic formula for
$$
\sum_{\substack{l_1 l_2 l_3 l_4 l_5 \in \mathcal{A} \\ l_j \sim L_j \\ 1 \leqslant j \leqslant 4 }} a_{l_1} b_{l_2} c_{l_3} d_{l_4} e_{l_5} \quad \text{and thus for} \quad \sum_{\substack{p_{j} \sim x^{\alpha_{j}} \\ 1 \leqslant j \leqslant 4}} S\left(\mathcal{A}_{p_{1} p_{2} p_{3} p_{4}}, p_{4}\right).
$$
\end{lemma}

\begin{lemma}\label{l35}
Let $L_{1} L_{2} L_{3} L_{4} L_{5} L_{6}=x, L_{j}=x^{\alpha_{j}}, \alpha_{j} \geqslant \varepsilon$ and suppose that $L_{j}(s)$ is prime-factored for $j \geqslant 2$. If any of the following 87 conditions hold:
\begin{align}
\nonumber & \alpha_{1} \geqslant 1-\theta,\ \alpha_{2} \geqslant \frac{(1-\theta)}{2},\ \alpha_{3} \geqslant \frac{(1-\theta)}{3},\ \alpha_{4} \geqslant \frac{(1-\theta)}{7},\ \alpha_{5} \geqslant \frac{(1-\theta)}{43},\ 1-\alpha_{1}-\alpha_2 -\alpha_3 -\alpha_4 -\alpha_5 \geqslant \frac{2(1-\theta)}{3611}; \\
\nonumber & \alpha_{1} \geqslant 1-\theta,\ \alpha_{2} \geqslant \frac{(1-\theta)}{2},\ \alpha_{3} \geqslant \frac{(1-\theta)}{3},\ \alpha_{4} \geqslant \frac{(1-\theta)}{7},\ \alpha_{5} \geqslant \frac{(1-\theta)}{44},\ 1-\alpha_{1}-\alpha_2 -\alpha_3 -\alpha_4 -\alpha_5 \geqslant \frac{2(1-\theta)}{1847}; \\
\nonumber & \alpha_{1} \geqslant 1-\theta,\ \alpha_{2} \geqslant \frac{(1-\theta)}{2},\ \alpha_{3} \geqslant \frac{(1-\theta)}{3},\ \alpha_{4} \geqslant \frac{(1-\theta)}{7},\ \alpha_{5} \geqslant \frac{(1-\theta)}{45},\ 1-\alpha_{1}-\alpha_2 -\alpha_3 -\alpha_4 -\alpha_5 \geqslant \frac{2(1-\theta)}{1259}; \\
\nonumber & \alpha_{1} \geqslant 1-\theta,\ \alpha_{2} \geqslant \frac{(1-\theta)}{2},\ \alpha_{3} \geqslant \frac{(1-\theta)}{3},\ \alpha_{4} \geqslant \frac{(1-\theta)}{7},\ \alpha_{5} \geqslant \frac{(1-\theta)}{46},\ 1-\alpha_{1}-\alpha_2 -\alpha_3 -\alpha_4 -\alpha_5 \geqslant \frac{2(1-\theta)}{965}; \\
\nonumber & \alpha_{1} \geqslant 1-\theta,\ \alpha_{2} \geqslant \frac{(1-\theta)}{2},\ \alpha_{3} \geqslant \frac{(1-\theta)}{3},\ \alpha_{4} \geqslant \frac{(1-\theta)}{7},\ \alpha_{5} \geqslant \frac{(1-\theta)}{48},\ 1-\alpha_{1}-\alpha_2 -\alpha_3 -\alpha_4 -\alpha_5 \geqslant \frac{2(1-\theta)}{671}; \\
\nonumber & \alpha_{1} \geqslant 1-\theta,\ \alpha_{2} \geqslant \frac{(1-\theta)}{2},\ \alpha_{3} \geqslant \frac{(1-\theta)}{3},\ \alpha_{4} \geqslant \frac{(1-\theta)}{7},\ \alpha_{5} \geqslant \frac{(1-\theta)}{49},\ 1-\alpha_{1}-\alpha_2 -\alpha_3 -\alpha_4 -\alpha_5 \geqslant \frac{2(1-\theta)}{587}; \\
\nonumber & \alpha_{1} \geqslant 1-\theta,\ \alpha_{2} \geqslant \frac{(1-\theta)}{2},\ \alpha_{3} \geqslant \frac{(1-\theta)}{3},\ \alpha_{4} \geqslant \frac{(1-\theta)}{7},\ \alpha_{5} \geqslant \frac{(1-\theta)}{51},\ 1-\alpha_{1}-\alpha_2 -\alpha_3 -\alpha_4 -\alpha_5 \geqslant \frac{2(1-\theta)}{475}; \\
\nonumber & \alpha_{1} \geqslant 1-\theta,\ \alpha_{2} \geqslant \frac{(1-\theta)}{2},\ \alpha_{3} \geqslant \frac{(1-\theta)}{3},\ \alpha_{4} \geqslant \frac{(1-\theta)}{7},\ \alpha_{5} \geqslant \frac{(1-\theta)}{54},\ 1-\alpha_{1}-\alpha_2 -\alpha_3 -\alpha_4 -\alpha_5 \geqslant \frac{2(1-\theta)}{377}; \\
\nonumber & \alpha_{1} \geqslant 1-\theta,\ \alpha_{2} \geqslant \frac{(1-\theta)}{2},\ \alpha_{3} \geqslant \frac{(1-\theta)}{3},\ \alpha_{4} \geqslant \frac{(1-\theta)}{7},\ \alpha_{5} \geqslant \frac{(1-\theta)}{56},\ 1-\alpha_{1}-\alpha_2 -\alpha_3 -\alpha_4 -\alpha_5 \geqslant \frac{2(1-\theta)}{335}; \\
\nonumber & \alpha_{1} \geqslant 1-\theta,\ \alpha_{2} \geqslant \frac{(1-\theta)}{2},\ \alpha_{3} \geqslant \frac{(1-\theta)}{3},\ \alpha_{4} \geqslant \frac{(1-\theta)}{7},\ \alpha_{5} \geqslant \frac{(1-\theta)}{60},\ 1-\alpha_{1}-\alpha_2 -\alpha_3 -\alpha_4 -\alpha_5 \geqslant \frac{2(1-\theta)}{279}; \\
\nonumber & \alpha_{1} \geqslant 1-\theta,\ \alpha_{2} \geqslant \frac{(1-\theta)}{2},\ \alpha_{3} \geqslant \frac{(1-\theta)}{3},\ \alpha_{4} \geqslant \frac{(1-\theta)}{7},\ \alpha_{5} \geqslant \frac{(1-\theta)}{63},\ 1-\alpha_{1}-\alpha_2 -\alpha_3 -\alpha_4 -\alpha_5 \geqslant \frac{2(1-\theta)}{251}; \\
\nonumber & \alpha_{1} \geqslant 1-\theta,\ \alpha_{2} \geqslant \frac{(1-\theta)}{2},\ \alpha_{3} \geqslant \frac{(1-\theta)}{3},\ \alpha_{4} \geqslant \frac{(1-\theta)}{7},\ \alpha_{5} \geqslant \frac{(1-\theta)}{70},\ 1-\alpha_{1}-\alpha_2 -\alpha_3 -\alpha_4 -\alpha_5 \geqslant \frac{2(1-\theta)}{209}; \\
\nonumber & \alpha_{1} \geqslant 1-\theta,\ \alpha_{2} \geqslant \frac{(1-\theta)}{2},\ \alpha_{3} \geqslant \frac{(1-\theta)}{3},\ \alpha_{4} \geqslant \frac{(1-\theta)}{7},\ \alpha_{5} \geqslant \frac{(1-\theta)}{78},\ 1-\alpha_{1}-\alpha_2 -\alpha_3 -\alpha_4 -\alpha_5 \geqslant \frac{2(1-\theta)}{181}; \\
\nonumber & \alpha_{1} \geqslant 1-\theta,\ \alpha_{2} \geqslant \frac{(1-\theta)}{2},\ \alpha_{3} \geqslant \frac{(1-\theta)}{3},\ \alpha_{4} \geqslant \frac{(1-\theta)}{7},\ \alpha_{5} \geqslant \frac{(1-\theta)}{84},\ 1-\alpha_{1}-\alpha_2 -\alpha_3 -\alpha_4 -\alpha_5 \geqslant \frac{2(1-\theta)}{167}; \\
\nonumber & \alpha_{1} \geqslant 1-\theta,\ \alpha_{2} \geqslant \frac{(1-\theta)}{2},\ \alpha_{3} \geqslant \frac{(1-\theta)}{3},\ \alpha_{4} \geqslant \frac{(1-\theta)}{8},\ \alpha_{5} \geqslant \frac{(1-\theta)}{25},\ 1-\alpha_{1}-\alpha_2 -\alpha_3 -\alpha_4 -\alpha_5 \geqslant \frac{2(1-\theta)}{1199}; \\
\nonumber & \alpha_{1} \geqslant 1-\theta,\ \alpha_{2} \geqslant \frac{(1-\theta)}{2},\ \alpha_{3} \geqslant \frac{(1-\theta)}{3},\ \alpha_{4} \geqslant \frac{(1-\theta)}{8},\ \alpha_{5} \geqslant \frac{(1-\theta)}{26},\ 1-\alpha_{1}-\alpha_2 -\alpha_3 -\alpha_4 -\alpha_5 \geqslant \frac{2(1-\theta)}{623}; \\
\nonumber & \alpha_{1} \geqslant 1-\theta,\ \alpha_{2} \geqslant \frac{(1-\theta)}{2},\ \alpha_{3} \geqslant \frac{(1-\theta)}{3},\ \alpha_{4} \geqslant \frac{(1-\theta)}{8},\ \alpha_{5} \geqslant \frac{(1-\theta)}{27},\ 1-\alpha_{1}-\alpha_2 -\alpha_3 -\alpha_4 -\alpha_5 \geqslant \frac{2(1-\theta)}{431}; \\
\nonumber & \alpha_{1} \geqslant 1-\theta,\ \alpha_{2} \geqslant \frac{(1-\theta)}{2},\ \alpha_{3} \geqslant \frac{(1-\theta)}{3},\ \alpha_{4} \geqslant \frac{(1-\theta)}{8},\ \alpha_{5} \geqslant \frac{(1-\theta)}{28},\ 1-\alpha_{1}-\alpha_2 -\alpha_3 -\alpha_4 -\alpha_5 \geqslant \frac{2(1-\theta)}{335}; \\
\nonumber & \alpha_{1} \geqslant 1-\theta,\ \alpha_{2} \geqslant \frac{(1-\theta)}{2},\ \alpha_{3} \geqslant \frac{(1-\theta)}{3},\ \alpha_{4} \geqslant \frac{(1-\theta)}{8},\ \alpha_{5} \geqslant \frac{(1-\theta)}{30},\ 1-\alpha_{1}-\alpha_2 -\alpha_3 -\alpha_4 -\alpha_5 \geqslant \frac{2(1-\theta)}{239}; \\
\nonumber & \alpha_{1} \geqslant 1-\theta,\ \alpha_{2} \geqslant \frac{(1-\theta)}{2},\ \alpha_{3} \geqslant \frac{(1-\theta)}{3},\ \alpha_{4} \geqslant \frac{(1-\theta)}{8},\ \alpha_{5} \geqslant \frac{(1-\theta)}{32},\ 1-\alpha_{1}-\alpha_2 -\alpha_3 -\alpha_4 -\alpha_5 \geqslant \frac{2(1-\theta)}{191}; \\
\nonumber & \alpha_{1} \geqslant 1-\theta,\ \alpha_{2} \geqslant \frac{(1-\theta)}{2},\ \alpha_{3} \geqslant \frac{(1-\theta)}{3},\ \alpha_{4} \geqslant \frac{(1-\theta)}{8},\ \alpha_{5} \geqslant \frac{(1-\theta)}{33},\ 1-\alpha_{1}-\alpha_2 -\alpha_3 -\alpha_4 -\alpha_5 \geqslant \frac{2(1-\theta)}{175}; \\
\nonumber & \alpha_{1} \geqslant 1-\theta,\ \alpha_{2} \geqslant \frac{(1-\theta)}{2},\ \alpha_{3} \geqslant \frac{(1-\theta)}{3},\ \alpha_{4} \geqslant \frac{(1-\theta)}{8},\ \alpha_{5} \geqslant \frac{(1-\theta)}{36},\ 1-\alpha_{1}-\alpha_2 -\alpha_3 -\alpha_4 -\alpha_5 \geqslant \frac{2(1-\theta)}{143}; \\
\nonumber & \alpha_{1} \geqslant 1-\theta,\ \alpha_{2} \geqslant \frac{(1-\theta)}{2},\ \alpha_{3} \geqslant \frac{(1-\theta)}{3},\ \alpha_{4} \geqslant \frac{(1-\theta)}{8},\ \alpha_{5} \geqslant \frac{(1-\theta)}{40},\ 1-\alpha_{1}-\alpha_2 -\alpha_3 -\alpha_4 -\alpha_5 \geqslant \frac{2(1-\theta)}{119}; \\
\nonumber & \alpha_{1} \geqslant 1-\theta,\ \alpha_{2} \geqslant \frac{(1-\theta)}{2},\ \alpha_{3} \geqslant \frac{(1-\theta)}{3},\ \alpha_{4} \geqslant \frac{(1-\theta)}{8},\ \alpha_{5} \geqslant \frac{(1-\theta)}{42},\ 1-\alpha_{1}-\alpha_2 -\alpha_3 -\alpha_4 -\alpha_5 \geqslant \frac{2(1-\theta)}{111}; \\
\nonumber & \alpha_{1} \geqslant 1-\theta,\ \alpha_{2} \geqslant \frac{(1-\theta)}{2},\ \alpha_{3} \geqslant \frac{(1-\theta)}{3},\ \alpha_{4} \geqslant \frac{(1-\theta)}{8},\ \alpha_{5} \geqslant \frac{(1-\theta)}{48},\ 1-\alpha_{1}-\alpha_2 -\alpha_3 -\alpha_4 -\alpha_5 \geqslant \frac{2(1-\theta)}{95}; \\
\nonumber & \alpha_{1} \geqslant 1-\theta,\ \alpha_{2} \geqslant \frac{(1-\theta)}{2},\ \alpha_{3} \geqslant \frac{(1-\theta)}{3},\ \alpha_{4} \geqslant \frac{(1-\theta)}{9},\ \alpha_{5} \geqslant \frac{(1-\theta)}{19},\ 1-\alpha_{1}-\alpha_2 -\alpha_3 -\alpha_4 -\alpha_5 \geqslant \frac{2(1-\theta)}{683}; \\
\nonumber & \alpha_{1} \geqslant 1-\theta,\ \alpha_{2} \geqslant \frac{(1-\theta)}{2},\ \alpha_{3} \geqslant \frac{(1-\theta)}{3},\ \alpha_{4} \geqslant \frac{(1-\theta)}{9},\ \alpha_{5} \geqslant \frac{(1-\theta)}{20},\ 1-\alpha_{1}-\alpha_2 -\alpha_3 -\alpha_4 -\alpha_5 \geqslant \frac{2(1-\theta)}{359}; \\
\nonumber & \alpha_{1} \geqslant 1-\theta,\ \alpha_{2} \geqslant \frac{(1-\theta)}{2},\ \alpha_{3} \geqslant \frac{(1-\theta)}{3},\ \alpha_{4} \geqslant \frac{(1-\theta)}{9},\ \alpha_{5} \geqslant \frac{(1-\theta)}{21},\ 1-\alpha_{1}-\alpha_2 -\alpha_3 -\alpha_4 -\alpha_5 \geqslant \frac{2(1-\theta)}{251}; \\
\nonumber & \alpha_{1} \geqslant 1-\theta,\ \alpha_{2} \geqslant \frac{(1-\theta)}{2},\ \alpha_{3} \geqslant \frac{(1-\theta)}{3},\ \alpha_{4} \geqslant \frac{(1-\theta)}{9},\ \alpha_{5} \geqslant \frac{(1-\theta)}{22},\ 1-\alpha_{1}-\alpha_2 -\alpha_3 -\alpha_4 -\alpha_5 \geqslant \frac{2(1-\theta)}{197}; \\
\nonumber & \alpha_{1} \geqslant 1-\theta,\ \alpha_{2} \geqslant \frac{(1-\theta)}{2},\ \alpha_{3} \geqslant \frac{(1-\theta)}{3},\ \alpha_{4} \geqslant \frac{(1-\theta)}{9},\ \alpha_{5} \geqslant \frac{(1-\theta)}{24},\ 1-\alpha_{1}-\alpha_2 -\alpha_3 -\alpha_4 -\alpha_5 \geqslant \frac{2(1-\theta)}{143}; \\
\nonumber & \alpha_{1} \geqslant 1-\theta,\ \alpha_{2} \geqslant \frac{(1-\theta)}{2},\ \alpha_{3} \geqslant \frac{(1-\theta)}{3},\ \alpha_{4} \geqslant \frac{(1-\theta)}{9},\ \alpha_{5} \geqslant \frac{(1-\theta)}{27},\ 1-\alpha_{1}-\alpha_2 -\alpha_3 -\alpha_4 -\alpha_5 \geqslant \frac{2(1-\theta)}{107}; \\
\nonumber & \alpha_{1} \geqslant 1-\theta,\ \alpha_{2} \geqslant \frac{(1-\theta)}{2},\ \alpha_{3} \geqslant \frac{(1-\theta)}{3},\ \alpha_{4} \geqslant \frac{(1-\theta)}{9},\ \alpha_{5} \geqslant \frac{(1-\theta)}{30},\ 1-\alpha_{1}-\alpha_2 -\alpha_3 -\alpha_4 -\alpha_5 \geqslant \frac{2(1-\theta)}{89}; \\
\nonumber & \alpha_{1} \geqslant 1-\theta,\ \alpha_{2} \geqslant \frac{(1-\theta)}{2},\ \alpha_{3} \geqslant \frac{(1-\theta)}{3},\ \alpha_{4} \geqslant \frac{(1-\theta)}{9},\ \alpha_{5} \geqslant \frac{(1-\theta)}{36},\ 1-\alpha_{1}-\alpha_2 -\alpha_3 -\alpha_4 -\alpha_5 \geqslant \frac{2(1-\theta)}{71}; \\
\nonumber & \alpha_{1} \geqslant 1-\theta,\ \alpha_{2} \geqslant \frac{(1-\theta)}{2},\ \alpha_{3} \geqslant \frac{(1-\theta)}{3},\ \alpha_{4} \geqslant \frac{(1-\theta)}{10},\ \alpha_{5} \geqslant \frac{(1-\theta)}{16},\ 1-\alpha_{1}-\alpha_2 -\alpha_3 -\alpha_4 -\alpha_5 \geqslant \frac{2(1-\theta)}{479}; \\
\nonumber & \alpha_{1} \geqslant 1-\theta,\ \alpha_{2} \geqslant \frac{(1-\theta)}{2},\ \alpha_{3} \geqslant \frac{(1-\theta)}{3},\ \alpha_{4} \geqslant \frac{(1-\theta)}{10},\ \alpha_{5} \geqslant \frac{(1-\theta)}{18},\ 1-\alpha_{1}-\alpha_2 -\alpha_3 -\alpha_4 -\alpha_5 \geqslant \frac{2(1-\theta)}{179}; \\
\nonumber & \alpha_{1} \geqslant 1-\theta,\ \alpha_{2} \geqslant \frac{(1-\theta)}{2},\ \alpha_{3} \geqslant \frac{(1-\theta)}{3},\ \alpha_{4} \geqslant \frac{(1-\theta)}{10},\ \alpha_{5} \geqslant \frac{(1-\theta)}{20},\ 1-\alpha_{1}-\alpha_2 -\alpha_3 -\alpha_4 -\alpha_5 \geqslant \frac{2(1-\theta)}{119}; \\
\nonumber & \alpha_{1} \geqslant 1-\theta,\ \alpha_{2} \geqslant \frac{(1-\theta)}{2},\ \alpha_{3} \geqslant \frac{(1-\theta)}{3},\ \alpha_{4} \geqslant \frac{(1-\theta)}{10},\ \alpha_{5} \geqslant \frac{(1-\theta)}{24},\ 1-\alpha_{1}-\alpha_2 -\alpha_3 -\alpha_4 -\alpha_5 \geqslant \frac{2(1-\theta)}{79}; \\
\nonumber & \alpha_{1} \geqslant 1-\theta,\ \alpha_{2} \geqslant \frac{(1-\theta)}{2},\ \alpha_{3} \geqslant \frac{(1-\theta)}{3},\ \alpha_{4} \geqslant \frac{(1-\theta)}{10},\ \alpha_{5} \geqslant \frac{(1-\theta)}{30},\ 1-\alpha_{1}-\alpha_2 -\alpha_3 -\alpha_4 -\alpha_5 \geqslant \frac{2(1-\theta)}{59}; \\
\nonumber & \alpha_{1} \geqslant 1-\theta,\ \alpha_{2} \geqslant \frac{(1-\theta)}{2},\ \alpha_{3} \geqslant \frac{(1-\theta)}{3},\ \alpha_{4} \geqslant \frac{(1-\theta)}{11},\ \alpha_{5} \geqslant \frac{(1-\theta)}{14},\ 1-\alpha_{1}-\alpha_2 -\alpha_3 -\alpha_4 -\alpha_5 \geqslant \frac{2(1-\theta)}{461}; \\
\nonumber & \alpha_{1} \geqslant 1-\theta,\ \alpha_{2} \geqslant \frac{(1-\theta)}{2},\ \alpha_{3} \geqslant \frac{(1-\theta)}{3},\ \alpha_{4} \geqslant \frac{(1-\theta)}{11},\ \alpha_{5} \geqslant \frac{(1-\theta)}{15},\ 1-\alpha_{1}-\alpha_2 -\alpha_3 -\alpha_4 -\alpha_5 \geqslant \frac{2(1-\theta)}{219}; \\
\nonumber & \alpha_{1} \geqslant 1-\theta,\ \alpha_{2} \geqslant \frac{(1-\theta)}{2},\ \alpha_{3} \geqslant \frac{(1-\theta)}{3},\ \alpha_{4} \geqslant \frac{(1-\theta)}{11},\ \alpha_{5} \geqslant \frac{(1-\theta)}{22},\ 1-\alpha_{1}-\alpha_2 -\alpha_3 -\alpha_4 -\alpha_5 \geqslant \frac{2(1-\theta)}{65}; \\
\nonumber & \alpha_{1} \geqslant 1-\theta,\ \alpha_{2} \geqslant \frac{(1-\theta)}{2},\ \alpha_{3} \geqslant \frac{(1-\theta)}{3},\ \alpha_{4} \geqslant \frac{(1-\theta)}{12},\ \alpha_{5} \geqslant \frac{(1-\theta)}{13},\ 1-\alpha_{1}-\alpha_2 -\alpha_3 -\alpha_4 -\alpha_5 \geqslant \frac{2(1-\theta)}{311}; \\
\nonumber & \alpha_{1} \geqslant 1-\theta,\ \alpha_{2} \geqslant \frac{(1-\theta)}{2},\ \alpha_{3} \geqslant \frac{(1-\theta)}{3},\ \alpha_{4} \geqslant \frac{(1-\theta)}{12},\ \alpha_{5} \geqslant \frac{(1-\theta)}{14},\ 1-\alpha_{1}-\alpha_2 -\alpha_3 -\alpha_4 -\alpha_5 \geqslant \frac{2(1-\theta)}{167}; \\
\nonumber & \alpha_{1} \geqslant 1-\theta,\ \alpha_{2} \geqslant \frac{(1-\theta)}{2},\ \alpha_{3} \geqslant \frac{(1-\theta)}{3},\ \alpha_{4} \geqslant \frac{(1-\theta)}{12},\ \alpha_{5} \geqslant \frac{(1-\theta)}{15},\ 1-\alpha_{1}-\alpha_2 -\alpha_3 -\alpha_4 -\alpha_5 \geqslant \frac{2(1-\theta)}{119}; \\
\nonumber & \alpha_{1} \geqslant 1-\theta,\ \alpha_{2} \geqslant \frac{(1-\theta)}{2},\ \alpha_{3} \geqslant \frac{(1-\theta)}{3},\ \alpha_{4} \geqslant \frac{(1-\theta)}{12},\ \alpha_{5} \geqslant \frac{(1-\theta)}{16},\ 1-\alpha_{1}-\alpha_2 -\alpha_3 -\alpha_4 -\alpha_5 \geqslant \frac{2(1-\theta)}{95}; \\
\nonumber & \alpha_{1} \geqslant 1-\theta,\ \alpha_{2} \geqslant \frac{(1-\theta)}{2},\ \alpha_{3} \geqslant \frac{(1-\theta)}{3},\ \alpha_{4} \geqslant \frac{(1-\theta)}{12},\ \alpha_{5} \geqslant \frac{(1-\theta)}{18},\ 1-\alpha_{1}-\alpha_2 -\alpha_3 -\alpha_4 -\alpha_5 \geqslant \frac{2(1-\theta)}{71}; \\
\nonumber & \alpha_{1} \geqslant 1-\theta,\ \alpha_{2} \geqslant \frac{(1-\theta)}{2},\ \alpha_{3} \geqslant \frac{(1-\theta)}{3},\ \alpha_{4} \geqslant \frac{(1-\theta)}{12},\ \alpha_{5} \geqslant \frac{(1-\theta)}{20},\ 1-\alpha_{1}-\alpha_2 -\alpha_3 -\alpha_4 -\alpha_5 \geqslant \frac{2(1-\theta)}{59}; \\
\nonumber & \alpha_{1} \geqslant 1-\theta,\ \alpha_{2} \geqslant \frac{(1-\theta)}{2},\ \alpha_{3} \geqslant \frac{(1-\theta)}{3},\ \alpha_{4} \geqslant \frac{(1-\theta)}{12},\ \alpha_{5} \geqslant \frac{(1-\theta)}{21},\ 1-\alpha_{1}-\alpha_2 -\alpha_3 -\alpha_4 -\alpha_5 \geqslant \frac{2(1-\theta)}{55}; \\
\nonumber & \alpha_{1} \geqslant 1-\theta,\ \alpha_{2} \geqslant \frac{(1-\theta)}{2},\ \alpha_{3} \geqslant \frac{(1-\theta)}{3},\ \alpha_{4} \geqslant \frac{(1-\theta)}{13},\ \alpha_{5} \geqslant \frac{(1-\theta)}{13},\ 1-\alpha_{1}-\alpha_2 -\alpha_3 -\alpha_4 -\alpha_5 \geqslant \frac{2(1-\theta)}{155}; \\
\nonumber & \alpha_{1} \geqslant 1-\theta,\ \alpha_{2} \geqslant \frac{(1-\theta)}{2},\ \alpha_{3} \geqslant \frac{(1-\theta)}{3},\ \alpha_{4} \geqslant \frac{(1-\theta)}{14},\ \alpha_{5} \geqslant \frac{(1-\theta)}{15},\ 1-\alpha_{1}-\alpha_2 -\alpha_3 -\alpha_4 -\alpha_5 \geqslant \frac{2(1-\theta)}{69}; \\
\nonumber & \alpha_{1} \geqslant 1-\theta,\ \alpha_{2} \geqslant \frac{(1-\theta)}{2},\ \alpha_{3} \geqslant \frac{(1-\theta)}{3},\ \alpha_{4} \geqslant \frac{(1-\theta)}{14},\ \alpha_{5} \geqslant \frac{(1-\theta)}{21},\ 1-\alpha_{1}-\alpha_2 -\alpha_3 -\alpha_4 -\alpha_5 \geqslant \frac{2(1-\theta)}{41}; \\
\nonumber & \alpha_{1} \geqslant 1-\theta,\ \alpha_{2} \geqslant \frac{(1-\theta)}{2},\ \alpha_{3} \geqslant \frac{(1-\theta)}{3},\ \alpha_{4} \geqslant \frac{(1-\theta)}{15},\ \alpha_{5} \geqslant \frac{(1-\theta)}{15},\ 1-\alpha_{1}-\alpha_2 -\alpha_3 -\alpha_4 -\alpha_5 \geqslant \frac{2(1-\theta)}{59}; \\
\nonumber & \alpha_{1} \geqslant 1-\theta,\ \alpha_{2} \geqslant \frac{(1-\theta)}{2},\ \alpha_{3} \geqslant \frac{(1-\theta)}{3},\ \alpha_{4} \geqslant \frac{(1-\theta)}{15},\ \alpha_{5} \geqslant \frac{(1-\theta)}{20},\ 1-\alpha_{1}-\alpha_2 -\alpha_3 -\alpha_4 -\alpha_5 \geqslant \frac{2(1-\theta)}{39}; \\
\nonumber & \alpha_{1} \geqslant 1-\theta,\ \alpha_{2} \geqslant \frac{(1-\theta)}{2},\ \alpha_{3} \geqslant \frac{(1-\theta)}{4},\ \alpha_{4} \geqslant \frac{(1-\theta)}{5},\ \alpha_{5} \geqslant \frac{(1-\theta)}{21},\ 1-\alpha_{1}-\alpha_2 -\alpha_3 -\alpha_4 -\alpha_5 \geqslant \frac{2(1-\theta)}{839}; \\
\nonumber & \alpha_{1} \geqslant 1-\theta,\ \alpha_{2} \geqslant \frac{(1-\theta)}{2},\ \alpha_{3} \geqslant \frac{(1-\theta)}{4},\ \alpha_{4} \geqslant \frac{(1-\theta)}{5},\ \alpha_{5} \geqslant \frac{(1-\theta)}{22},\ 1-\alpha_{1}-\alpha_2 -\alpha_3 -\alpha_4 -\alpha_5 \geqslant \frac{2(1-\theta)}{439}; \\
\nonumber & \alpha_{1} \geqslant 1-\theta,\ \alpha_{2} \geqslant \frac{(1-\theta)}{2},\ \alpha_{3} \geqslant \frac{(1-\theta)}{4},\ \alpha_{4} \geqslant \frac{(1-\theta)}{5},\ \alpha_{5} \geqslant \frac{(1-\theta)}{24},\ 1-\alpha_{1}-\alpha_2 -\alpha_3 -\alpha_4 -\alpha_5 \geqslant \frac{2(1-\theta)}{239}; \\
\nonumber & \alpha_{1} \geqslant 1-\theta,\ \alpha_{2} \geqslant \frac{(1-\theta)}{2},\ \alpha_{3} \geqslant \frac{(1-\theta)}{4},\ \alpha_{4} \geqslant \frac{(1-\theta)}{5},\ \alpha_{5} \geqslant \frac{(1-\theta)}{25},\ 1-\alpha_{1}-\alpha_2 -\alpha_3 -\alpha_4 -\alpha_5 \geqslant \frac{2(1-\theta)}{199}; \\
\nonumber & \alpha_{1} \geqslant 1-\theta,\ \alpha_{2} \geqslant \frac{(1-\theta)}{2},\ \alpha_{3} \geqslant \frac{(1-\theta)}{4},\ \alpha_{4} \geqslant \frac{(1-\theta)}{5},\ \alpha_{5} \geqslant \frac{(1-\theta)}{28},\ 1-\alpha_{1}-\alpha_2 -\alpha_3 -\alpha_4 -\alpha_5 \geqslant \frac{2(1-\theta)}{139}; \\
\nonumber & \alpha_{1} \geqslant 1-\theta,\ \alpha_{2} \geqslant \frac{(1-\theta)}{2},\ \alpha_{3} \geqslant \frac{(1-\theta)}{4},\ \alpha_{4} \geqslant \frac{(1-\theta)}{5},\ \alpha_{5} \geqslant \frac{(1-\theta)}{30},\ 1-\alpha_{1}-\alpha_2 -\alpha_3 -\alpha_4 -\alpha_5 \geqslant \frac{2(1-\theta)}{119}; \\
\nonumber & \alpha_{1} \geqslant 1-\theta,\ \alpha_{2} \geqslant \frac{(1-\theta)}{2},\ \alpha_{3} \geqslant \frac{(1-\theta)}{4},\ \alpha_{4} \geqslant \frac{(1-\theta)}{5},\ \alpha_{5} \geqslant \frac{(1-\theta)}{36},\ 1-\alpha_{1}-\alpha_2 -\alpha_3 -\alpha_4 -\alpha_5 \geqslant \frac{2(1-\theta)}{89}; \\
\nonumber & \alpha_{1} \geqslant 1-\theta,\ \alpha_{2} \geqslant \frac{(1-\theta)}{2},\ \alpha_{3} \geqslant \frac{(1-\theta)}{4},\ \alpha_{4} \geqslant \frac{(1-\theta)}{5},\ \alpha_{5} \geqslant \frac{(1-\theta)}{40},\ 1-\alpha_{1}-\alpha_2 -\alpha_3 -\alpha_4 -\alpha_5 \geqslant \frac{2(1-\theta)}{79}; \\
\nonumber & \alpha_{1} \geqslant 1-\theta,\ \alpha_{2} \geqslant \frac{(1-\theta)}{2},\ \alpha_{3} \geqslant \frac{(1-\theta)}{4},\ \alpha_{4} \geqslant \frac{(1-\theta)}{6},\ \alpha_{5} \geqslant \frac{(1-\theta)}{13},\ 1-\alpha_{1}-\alpha_2 -\alpha_3 -\alpha_4 -\alpha_5 \geqslant \frac{2(1-\theta)}{311}; \\
\nonumber & \alpha_{1} \geqslant 1-\theta,\ \alpha_{2} \geqslant \frac{(1-\theta)}{2},\ \alpha_{3} \geqslant \frac{(1-\theta)}{4},\ \alpha_{4} \geqslant \frac{(1-\theta)}{6},\ \alpha_{5} \geqslant \frac{(1-\theta)}{14},\ 1-\alpha_{1}-\alpha_2 -\alpha_3 -\alpha_4 -\alpha_5 \geqslant \frac{2(1-\theta)}{167}; \\
\nonumber & \alpha_{1} \geqslant 1-\theta,\ \alpha_{2} \geqslant \frac{(1-\theta)}{2},\ \alpha_{3} \geqslant \frac{(1-\theta)}{4},\ \alpha_{4} \geqslant \frac{(1-\theta)}{6},\ \alpha_{5} \geqslant \frac{(1-\theta)}{15},\ 1-\alpha_{1}-\alpha_2 -\alpha_3 -\alpha_4 -\alpha_5 \geqslant \frac{2(1-\theta)}{119}; \\
\nonumber & \alpha_{1} \geqslant 1-\theta,\ \alpha_{2} \geqslant \frac{(1-\theta)}{2},\ \alpha_{3} \geqslant \frac{(1-\theta)}{4},\ \alpha_{4} \geqslant \frac{(1-\theta)}{6},\ \alpha_{5} \geqslant \frac{(1-\theta)}{16},\ 1-\alpha_{1}-\alpha_2 -\alpha_3 -\alpha_4 -\alpha_5 \geqslant \frac{2(1-\theta)}{95}; \\
\nonumber & \alpha_{1} \geqslant 1-\theta,\ \alpha_{2} \geqslant \frac{(1-\theta)}{2},\ \alpha_{3} \geqslant \frac{(1-\theta)}{4},\ \alpha_{4} \geqslant \frac{(1-\theta)}{6},\ \alpha_{5} \geqslant \frac{(1-\theta)}{18},\ 1-\alpha_{1}-\alpha_2 -\alpha_3 -\alpha_4 -\alpha_5 \geqslant \frac{2(1-\theta)}{71}; \\
\nonumber & \alpha_{1} \geqslant 1-\theta,\ \alpha_{2} \geqslant \frac{(1-\theta)}{2},\ \alpha_{3} \geqslant \frac{(1-\theta)}{4},\ \alpha_{4} \geqslant \frac{(1-\theta)}{6},\ \alpha_{5} \geqslant \frac{(1-\theta)}{20},\ 1-\alpha_{1}-\alpha_2 -\alpha_3 -\alpha_4 -\alpha_5 \geqslant \frac{2(1-\theta)}{59}; \\
\nonumber & \alpha_{1} \geqslant 1-\theta,\ \alpha_{2} \geqslant \frac{(1-\theta)}{2},\ \alpha_{3} \geqslant \frac{(1-\theta)}{4},\ \alpha_{4} \geqslant \frac{(1-\theta)}{6},\ \alpha_{5} \geqslant \frac{(1-\theta)}{21},\ 1-\alpha_{1}-\alpha_2 -\alpha_3 -\alpha_4 -\alpha_5 \geqslant \frac{2(1-\theta)}{55}; \\
\nonumber & \alpha_{1} \geqslant 1-\theta,\ \alpha_{2} \geqslant \frac{(1-\theta)}{2},\ \alpha_{3} \geqslant \frac{(1-\theta)}{4},\ \alpha_{4} \geqslant \frac{(1-\theta)}{6},\ \alpha_{5} \geqslant \frac{(1-\theta)}{24},\ 1-\alpha_{1}-\alpha_2 -\alpha_3 -\alpha_4 -\alpha_5 \geqslant \frac{2(1-\theta)}{47}; \\
\nonumber & \alpha_{1} \geqslant 1-\theta,\ \alpha_{2} \geqslant \frac{(1-\theta)}{2},\ \alpha_{3} \geqslant \frac{(1-\theta)}{4},\ \alpha_{4} \geqslant \frac{(1-\theta)}{7},\ \alpha_{5} \geqslant \frac{(1-\theta)}{10},\ 1-\alpha_{1}-\alpha_2 -\alpha_3 -\alpha_4 -\alpha_5 \geqslant \frac{2(1-\theta)}{279}; \\
\nonumber & \alpha_{1} \geqslant 1-\theta,\ \alpha_{2} \geqslant \frac{(1-\theta)}{2},\ \alpha_{3} \geqslant \frac{(1-\theta)}{4},\ \alpha_{4} \geqslant \frac{(1-\theta)}{7},\ \alpha_{5} \geqslant \frac{(1-\theta)}{14},\ 1-\alpha_{1}-\alpha_2 -\alpha_3 -\alpha_4 -\alpha_5 \geqslant \frac{2(1-\theta)}{55}; \\
\nonumber & \alpha_{1} \geqslant 1-\theta,\ \alpha_{2} \geqslant \frac{(1-\theta)}{2},\ \alpha_{3} \geqslant \frac{(1-\theta)}{4},\ \alpha_{4} \geqslant \frac{(1-\theta)}{8},\ \alpha_{5} \geqslant \frac{(1-\theta)}{9},\ 1-\alpha_{1}-\alpha_2 -\alpha_3 -\alpha_4 -\alpha_5 \geqslant \frac{2(1-\theta)}{143}; \\
\nonumber & \alpha_{1} \geqslant 1-\theta,\ \alpha_{2} \geqslant \frac{(1-\theta)}{2},\ \alpha_{3} \geqslant \frac{(1-\theta)}{4},\ \alpha_{4} \geqslant \frac{(1-\theta)}{8},\ \alpha_{5} \geqslant \frac{(1-\theta)}{10},\ 1-\alpha_{1}-\alpha_2 -\alpha_3 -\alpha_4 -\alpha_5 \geqslant \frac{2(1-\theta)}{79}; \\
\nonumber & \alpha_{1} \geqslant 1-\theta,\ \alpha_{2} \geqslant \frac{(1-\theta)}{2},\ \alpha_{3} \geqslant \frac{(1-\theta)}{4},\ \alpha_{4} \geqslant \frac{(1-\theta)}{8},\ \alpha_{5} \geqslant \frac{(1-\theta)}{16},\ 1-\alpha_{1}-\alpha_2 -\alpha_3 -\alpha_4 -\alpha_5 \geqslant \frac{2(1-\theta)}{31}; \\
\nonumber & \alpha_{1} \geqslant 1-\theta,\ \alpha_{2} \geqslant \frac{(1-\theta)}{2},\ \alpha_{3} \geqslant \frac{(1-\theta)}{4},\ \alpha_{4} \geqslant \frac{(1-\theta)}{9},\ \alpha_{5} \geqslant \frac{(1-\theta)}{9},\ 1-\alpha_{1}-\alpha_2 -\alpha_3 -\alpha_4 -\alpha_5 \geqslant \frac{2(1-\theta)}{71}; \\
\nonumber & \alpha_{1} \geqslant 1-\theta,\ \alpha_{2} \geqslant \frac{(1-\theta)}{2},\ \alpha_{3} \geqslant \frac{(1-\theta)}{5},\ \alpha_{4} \geqslant \frac{(1-\theta)}{5},\ \alpha_{5} \geqslant \frac{(1-\theta)}{11},\ 1-\alpha_{1}-\alpha_2 -\alpha_3 -\alpha_4 -\alpha_5 \geqslant \frac{2(1-\theta)}{219}; \\
\nonumber & \alpha_{1} \geqslant 1-\theta,\ \alpha_{2} \geqslant \frac{(1-\theta)}{2},\ \alpha_{3} \geqslant \frac{(1-\theta)}{5},\ \alpha_{4} \geqslant \frac{(1-\theta)}{5},\ \alpha_{5} \geqslant \frac{(1-\theta)}{12},\ 1-\alpha_{1}-\alpha_2 -\alpha_3 -\alpha_4 -\alpha_5 \geqslant \frac{2(1-\theta)}{119}; \\
\nonumber & \alpha_{1} \geqslant 1-\theta,\ \alpha_{2} \geqslant \frac{(1-\theta)}{2},\ \alpha_{3} \geqslant \frac{(1-\theta)}{5},\ \alpha_{4} \geqslant \frac{(1-\theta)}{5},\ \alpha_{5} \geqslant \frac{(1-\theta)}{14},\ 1-\alpha_{1}-\alpha_2 -\alpha_3 -\alpha_4 -\alpha_5 \geqslant \frac{2(1-\theta)}{69}; \\
\nonumber & \alpha_{1} \geqslant 1-\theta,\ \alpha_{2} \geqslant \frac{(1-\theta)}{2},\ \alpha_{3} \geqslant \frac{(1-\theta)}{5},\ \alpha_{4} \geqslant \frac{(1-\theta)}{5},\ \alpha_{5} \geqslant \frac{(1-\theta)}{15},\ 1-\alpha_{1}-\alpha_2 -\alpha_3 -\alpha_4 -\alpha_5 \geqslant \frac{2(1-\theta)}{59}; \\
\nonumber & \alpha_{1} \geqslant 1-\theta,\ \alpha_{2} \geqslant \frac{(1-\theta)}{2},\ \alpha_{3} \geqslant \frac{(1-\theta)}{5},\ \alpha_{4} \geqslant \frac{(1-\theta)}{6},\ \alpha_{5} \geqslant \frac{(1-\theta)}{8},\ 1-\alpha_{1}-\alpha_2 -\alpha_3 -\alpha_4 -\alpha_5 \geqslant \frac{2(1-\theta)}{239}; \\
\nonumber & \alpha_{1} \geqslant 1-\theta,\ \alpha_{2} \geqslant \frac{(1-\theta)}{2},\ \alpha_{3} \geqslant \frac{(1-\theta)}{5},\ \alpha_{4} \geqslant \frac{(1-\theta)}{6},\ \alpha_{5} \geqslant \frac{(1-\theta)}{9},\ 1-\alpha_{1}-\alpha_2 -\alpha_3 -\alpha_4 -\alpha_5 \geqslant \frac{2(1-\theta)}{89}; \\
\nonumber & \alpha_{1} \geqslant 1-\theta,\ \alpha_{2} \geqslant \frac{(1-\theta)}{2},\ \alpha_{3} \geqslant \frac{(1-\theta)}{5},\ \alpha_{4} \geqslant \frac{(1-\theta)}{6},\ \alpha_{5} \geqslant \frac{(1-\theta)}{10},\ 1-\alpha_{1}-\alpha_2 -\alpha_3 -\alpha_4 -\alpha_5 \geqslant \frac{2(1-\theta)}{59}; \\
\nonumber & \alpha_{1} \geqslant 1-\theta,\ \alpha_{2} \geqslant \frac{(1-\theta)}{2},\ \alpha_{3} \geqslant \frac{(1-\theta)}{5},\ \alpha_{4} \geqslant \frac{(1-\theta)}{6},\ \alpha_{5} \geqslant \frac{(1-\theta)}{15},\ 1-\alpha_{1}-\alpha_2 -\alpha_3 -\alpha_4 -\alpha_5 \geqslant \frac{2(1-\theta)}{29}; \\
\nonumber & \alpha_{1} \geqslant 1-\theta,\ \alpha_{2} \geqslant \frac{(1-\theta)}{2},\ \alpha_{3} \geqslant \frac{(1-\theta)}{5},\ \alpha_{4} \geqslant \frac{(1-\theta)}{7},\ \alpha_{5} \geqslant \frac{(1-\theta)}{7},\ 1-\alpha_{1}-\alpha_2 -\alpha_3 -\alpha_4 -\alpha_5 \geqslant \frac{2(1-\theta)}{139}; \\
\nonumber & \alpha_{1} \geqslant 1-\theta,\ \alpha_{2} \geqslant \frac{(1-\theta)}{2},\ \alpha_{3} \geqslant \frac{(1-\theta)}{6},\ \alpha_{4} \geqslant \frac{(1-\theta)}{7},\ \alpha_{5} \geqslant \frac{(1-\theta)}{7},\ 1-\alpha_{1}-\alpha_2 -\alpha_3 -\alpha_4 -\alpha_5 \geqslant \frac{2(1-\theta)}{41}; \\
\nonumber & \alpha_{1} \geqslant 1-\theta,\ \alpha_{2} \geqslant \frac{(1-\theta)}{2},\ \alpha_{3} \geqslant \frac{(1-\theta)}{6},\ \alpha_{4} \geqslant \frac{(1-\theta)}{9},\ \alpha_{5} \geqslant \frac{(1-\theta)}{9},\ 1-\alpha_{1}-\alpha_2 -\alpha_3 -\alpha_4 -\alpha_5 \geqslant \frac{2(1-\theta)}{17}; \\
\nonumber & \alpha_{1} \geqslant 1-\theta,\ \alpha_{2} \geqslant \frac{(1-\theta)}{2},\ \alpha_{3} \geqslant \frac{(1-\theta)}{7},\ \alpha_{4} \geqslant \frac{(1-\theta)}{7},\ \alpha_{5} \geqslant \frac{(1-\theta)}{7},\ 1-\alpha_{1}-\alpha_2 -\alpha_3 -\alpha_4 -\alpha_5 \geqslant \frac{2(1-\theta)}{27}.
\end{align}
Then we can obtain an asymptotic formula for
$$
\sum_{\substack{l_1 l_2 l_3 l_4 l_5 l_6 \in \mathcal{A} \\ l_j \sim L_j \\ 1 \leqslant j \leqslant 5 }} a_{l_1} b_{l_2} c_{l_3} d_{l_4} e_{l_5} f_{l_6} \quad \text{and thus for} \quad \sum_{\substack{p_{j} \sim x^{\alpha_{j}} \\ 1 \leqslant j \leqslant 5}} S\left(\mathcal{A}_{p_{1} p_{2} p_{3} p_{4} p_{5}}, p_{5}\right).
$$
\end{lemma}
\begin{proof}
We follow the steps in the proof of [\cite{HarmanBOOK}, Lemma 7.22]. Write $T_0 = \exp \left((\log x)^{1/3} \right)$ and $T = x^{1-\theta-\varepsilon/2}$. Let $B > 0$ and $n = 5 \text{ or } 6$. It suffices to show that
$$
\int_{T_0}^{T} \left|L_{1}\left(\frac{1}{2} + i t\right) \cdots L_{n}\left(\frac{1}{2} + i t\right)\right| d t \ll x^{\frac{1}{2}} (\log x)^{-K}.
$$
By Hölder's inequality we have
\begin{align}
\nonumber &\ \int_{T_0}^{T} \left|L_{1}\left(\frac{1}{2} + i t\right) \cdots L_{n}\left(\frac{1}{2} + i t\right)\right| d t \\
\nonumber \leqslant&\ \left( \int_{T_0}^{T} \left|L_{1}\left(\frac{1}{2} + i t\right) \right|^{2} d t \right)^{\frac{1}{2}} \left( \int_{T_0}^{T} \left|L_{2}\left(\frac{1}{2} + i t\right) \right|^{\delta_{2}} d t \right)^{\frac{1}{\delta_{2}}} \cdots \left( \int_{T_0}^{T} \left|L_{n}\left(\frac{1}{2} + i t\right) \right|^{\delta_{n}} d t \right)^{\frac{1}{\delta_{n}}},
\end{align}
where $\delta_{j}$ is even for all $2 \leqslant j \leqslant n - 1$ and
$$
\sum_{j = 2}^{n} \frac{1}{\delta_{j}} = \frac{1}{2}.
$$

Since $L_1 \gg T$ ($\alpha_1 \geqslant 1 - \theta$), by the mean value theorem for Dirichlet polynomials we know that
$$
\int_{T_0}^{T} \left|L_{1}\left(\frac{1}{2} + i t\right) \right|^{2} d t \ll x^{\alpha_{1}} (\log x)^{B}.
$$
Similarly, for $2 \leqslant j \leqslant n - 1$ we can also use the mean value theorem. If $L_j \gg T^{\frac{2}{\delta_j}}$ ($\alpha_j \geqslant \frac{2(1 - \theta)}{\delta_j}$), then
$$
\left( \int_{T_0}^{T} \left|L_{j}\left(\frac{1}{2} + i t\right) \right|^{\delta_j} d t \right)^{\frac{1}{\delta_j}} \ll x^{\frac{\alpha_j}{2}} (\log x)^{B}.
$$

For the remaining prime-factored Dirichlet polynomial $L_{n}(s)$, by [\cite{HarmanBOOK}, Lemma 7.21], we have
$$
\left( \int_{T_0}^{T} \left|L_{n}\left(\frac{1}{2} + i t\right) \right|^{\delta_n} d t \right)^{\frac{1}{\delta_n}} \ll x^{\frac{\alpha_n}{2}} (\log x)^{-K}
$$
under the conditions
$$
L_n \geqslant T^{\frac{1}{\delta_0}}, \qquad 1 \leqslant \delta_0 < B, \qquad \delta_n \geqslant 4 \delta_0 - 2 h + \varepsilon,
$$
where $h$ is an integer such that
$$
2 h - \varepsilon < \delta_n < 6 h - \varepsilon.
$$
We want to find the largest possible $\delta_0$. For $h \geqslant 1$, we can take
$$
\delta_n = 2 h + 2
$$
and thus
$$
2 h + 2 \geqslant 4 \delta_0 - 2 h + \varepsilon.
$$
Now we can take
$$
\nonumber \delta_0 = h + \frac{1}{2} - \varepsilon = \frac{\delta_n - 1}{2} - \varepsilon
$$
and all conditions are fulfilled.

Finally, we need to find $n$-tuples $(\delta_1 = 2, \delta_2, \ldots, \delta_n)$ such that
$$
\sum_{j = 1}^{n} \frac{1}{\delta_j} = 1.
$$
For each $n$-tuple that satisfies the above condition, we have a corresponding asymptotic formula region for ``Type-II$_n$'' sums determined by the following inequalities:
\begin{align}
\nonumber \alpha_1 \geqslant&\ 1 - \theta, \qquad \qquad \qquad \qquad \qquad \qquad \qquad \left(L_1 \gg T\right) \\
\nonumber \alpha_j \geqslant&\ \frac{2}{\delta_j}(1 - \theta) \quad \text{for all } 2 \leqslant j \leqslant n - 1, \qquad \left(L_j \gg T^{\frac{2}{\delta_j}}\right) \\
\nonumber \alpha_n \geqslant&\ \frac{2}{\delta_n - 1}(1 - \theta). \qquad \qquad \qquad \qquad \qquad \left(L_n \gg T^{\frac{1}{\delta_0}}\right)
\end{align}
By combining all satisfactory $n$-tuples for $n = 5$ and $n = 6$, Lemmas~\ref{l34} and \ref{l35} are proved.
\end{proof}
\begin{remark}
Here we shall explain why the Type-II information arises from Lemmas~\ref{l34}--\ref{l35} covers new regions out of previous information. Without loss of generality, we assume that $2 = \delta_1 \leqslant \delta_2 \leqslant \cdots \leqslant \delta_n$. Write $A_n = \alpha_1 + \cdots + \alpha_n$, we know that
$$
\frac{A_n}{1 - \theta} \geqslant \sum_{j = 1}^{n - 1} \frac{2}{\delta_j} + \frac{2}{\delta_n - 1} > \sum_{j = 1}^{n} \frac{2}{\delta_j} = 2.
$$
Since $\theta \geqslant 0.505$, we have $2(1 - \theta) \leqslant 0.99 < 1$, which means that we still have some satisfactory $n$-tuples $(\alpha_1, \ldots, \alpha_n)$ on $(0, 0.5)^n$, and the asymptotic region is not empty. In fact, the difference between $2$ and
$$
\sum_{j = 1}^{n - 1} \frac{2}{\delta_j} + \frac{2}{\delta_n - 1}
$$
is just
$$
\frac{2}{\delta_n - 1} - \frac{2}{\delta_n} = \frac{2}{\delta_n^2 - \delta_n}.
$$

Now, since the value of $\delta_n$ usually become larger as $n$ increases, $A_n$ will decrease and approach to the ``best possible'' value $2(1 - \theta)$ as $n$ increases. However, the asymptotic region becomes ``useless'' when $\alpha_n < \nu$. Hence considering ``Type-II$_n$'' information with $n \geqslant 7$ is not so worthwhile.
\end{remark}

\section{The final decomposition I: Lower Bound}
In this section, we ignore the presence of $\varepsilon$ for clarity. Let $\omega(u)$ denote the Buchstab function determined by the following differential-difference equation
\begin{align*}
\begin{cases}
\omega(u)=\frac{1}{u}, & \quad 1 \leqslant u \leqslant 2, \\
(u \omega(u))^{\prime}= \omega(u-1), & \quad u \geqslant 2 .
\end{cases}
\end{align*}
Moreover, we have the upper and lower bounds for $\omega(u)$:
\begin{align*}
\omega(u) \geqslant \omega_{0}(u) =
\begin{cases}
\frac{1}{u}, & \quad 1 \leqslant u < 2, \\
\frac{1+\log(u-1)}{u}, & \quad 2 \leqslant u < 3, \\
\frac{1+\log(u-1)}{u} + \frac{1}{u} \int_{2}^{u-1}\frac{\log(t-1)}{t} d t \geqslant 0.5607, & \quad 3 \leqslant u < 4, \\
0.5612, & \quad u \geqslant 4, \\
\end{cases}
\end{align*}
\begin{align*}
\omega(u) \leqslant \omega_{1}(u) =
\begin{cases}
\frac{1}{u}, & \quad 1 \leqslant u < 2, \\
\frac{1+\log(u-1)}{u}, & \quad 2 \leqslant u < 3, \\
\frac{1+\log(u-1)}{u} + \frac{1}{u} \int_{2}^{u-1}\frac{\log(t-1)}{t} d t \leqslant 0.5644, & \quad 3 \leqslant u < 4, \\
0.5617, & \quad u \geqslant 4. \\
\end{cases}
\end{align*}
We shall use $\omega_0(u)$ and $\omega_1(u)$ to give numerical bounds for some sieve functions discussed below. We shall also use the simple upper bound $\omega(u) \leqslant \max(\frac{1}{u}, 0.5672)$ (see Lemma 8(iii) of \cite{JiaPSV}) to estimate high-dimensional integrals.

By Prime Number Theorem with Vinogradov's error term and the inductive arguments in [\cite{HarmanBOOK}, Chapter A.2], we know that, for sufficiently large $z$,
\begin{equation}
S\left(\mathcal{B}, z\right) = \sum_{\substack{a \in \mathcal{B} \\ (a, P(z))=1}} 1 = (1+o(1)) \frac{y_1}{\log z} \omega\left(\frac{\log x}{\log z}\right),
\end{equation}
and we expect that the similar relation also holds for $S\left(\mathcal{A}, z\right)$:
\begin{equation}
S\left(\mathcal{A}, z\right) = \sum_{\substack{a \in \mathcal{A} \\ (a, P(z))=1}} 1 = (1+o(1)) \frac{y}{\log z} \omega\left(\frac{\log x}{\log z}\right).
\end{equation}
If (2) holds for $S\left(\mathcal{A}, z\right)$, then we can deduce (7) easily from (2) and (6). Otherwise we must drop this $S\left(\mathcal{A}, z\right)$. We define the \textit{loss} from this term by the size of corresponding $S\left(\mathcal{B}, z\right)$:
\begin{equation}
S\left(\mathcal{B}, z\right) = (\textit{loss}+o(1)) \frac{y_1}{\log x}.
\end{equation}
We note that for the lower bound problem, we can only drop positive parts and the total loss of the dropped parts must be less than $1$.

Fix $\theta =0.52$, $\nu_0 = \nu_{\min} = 2 \theta -1 = 0.04$ and let $p_{j}=x^{\alpha_{j}}$. By Buchstab's identity, we have
\begin{align}
\nonumber S\left(\mathcal{A}, x^{\frac{1}{2}}\right) =&\ S\left(\mathcal{A}, x^{\nu(0)}\right)-\sum_{\nu(0) \leqslant \alpha_{1}< \frac{1}{2}} S\left(\mathcal{A}_{p_{1}}, x^{\nu\left(\alpha_{1}\right)}\right) +\sum_{\substack{\nu(0) \leqslant \alpha_{1}<\frac{1}{2}\\
\nu\left(\alpha_{1}\right) \leqslant \alpha_{2}<\min \left(\alpha_{1},\frac{1}{2}\left(1-\alpha_{1}\right)\right)}} S\left(\mathcal{A}_{p_{1} p_{2}}, p_{2}\right)  \\
=&\ \sum_{1} - \sum_{2} + \sum_{3}.
\end{align}
We can give asymptotic formulas for $\sum_{1}$ and $\sum_{2}$. For $\sum_{3}$, We begin with some notation needed to describe the further decompositions. Following \cite{HarmanBOOK} directly, we use the bold capital letters $\boldsymbol{G}$ and $\boldsymbol{D}$ to represent sets that have asymptotic formulas and that can perform further decompositions directly. We write $\boldsymbol{\alpha}_{n}$ to denote $(\alpha_1, \ldots, \alpha_n)$ and similarly for $\boldsymbol{t}_{n}$. Let $\boldsymbol{G}_{n}$ denote the set of $\boldsymbol{\alpha}_{n}$ such that an asymptotic formula can be obtained for
$$
\sum_{\alpha_{1}, \ldots, \alpha_{n}} S\left(\mathcal{A}_{p_{1} \ldots p_{n}}, p_{n}\right),
$$
so we can define sets $\boldsymbol{G}_2$ and $\boldsymbol{G}_3$ by using Lemmas~\ref{l31}--\ref{l33}. We also need to define $\boldsymbol{G}_i$ with $i \geqslant 4$ in order to perform our calculation. For this, we can define them by checking whether a region $\boldsymbol{\alpha}_{i}$ can be partitioned into $(m, n) \in \boldsymbol{G}_{2}$ or $(m, n, h) \in \boldsymbol{G}_{3}$ in any order as long as the conditions are satisfied. By Lemmas~\ref{l21}--\ref{l22}, we put
\begin{align}
\nonumber \boldsymbol{D}_{0} =&\ \left\{\boldsymbol{\alpha}_{2}: 0 \leqslant \alpha_{1} \leqslant \frac{1}{2},\ 0 \leqslant \alpha_{2} \leqslant \min \left(\frac{3 \theta+1-4 \alpha_1^{*}}{2}, \frac{3+\theta-4 \alpha_1^{*}}{5}\right)\right\}, \\
\nonumber \boldsymbol{D}_{1} =&\ \left\{\boldsymbol{\alpha}_{3}: 0 \leqslant \alpha_{1} \leqslant \frac{1}{2},\ \alpha_{3} \leqslant \frac{1+3 \theta}{4}-\alpha_1^{*},\ 2 \alpha_{2}+\alpha_{3} \leqslant 1+\theta-2 \alpha_1^{*},\ 2 \alpha_{2}+3 \alpha_{3} \leqslant \frac{3+\theta}{2}-2 \alpha_1^{*}\right\}, \\
\nonumber \boldsymbol{D}_{2} =&\ \left\{\boldsymbol{\alpha}_{3}: 0 \leqslant \alpha_1 \leqslant \frac{1}{2},\ \alpha_{2} \leqslant \frac{1-\theta}{2},\ \alpha_{3} \leqslant \frac{1+3 \theta- 4\alpha_1^{*}}{8} \right\} , \\
\nonumber \boldsymbol{D}_{0}^{\prime} =&\ \left\{\boldsymbol{\alpha}_{2}: 0 \leqslant \alpha_{1} \leqslant \frac{1}{2},\ 0 \leqslant \alpha_{2} \leqslant \min \left(\frac{3 \theta -1}{2}, \frac{1+\theta}{5}\right)\right\}, \\
\nonumber \boldsymbol{D}_{1}^{\prime} =&\ \left\{\boldsymbol{\alpha}_{3}: 0 \leqslant \alpha_{1} \leqslant \frac{1}{2},\ \alpha_{3} \leqslant \frac{3 \theta -1}{4},\ 2 \alpha_{2}+\alpha_{3} \leqslant \theta,\ 2 \alpha_{2}+3 \alpha_{3} \leqslant \frac{1+\theta}{2}\right\}, \\
\nonumber \boldsymbol{D}_{2}^{\prime} =&\ \left\{\boldsymbol{\alpha}_{3}: 0 \leqslant \alpha_1 \leqslant \frac{1}{2},\ \alpha_{2} \leqslant \frac{1-\theta}{2},\ \alpha_{3} \leqslant \frac{3 \theta -1}{8} \right\} , \\
\nonumber \boldsymbol{D}^{*} =&\ \left\{\boldsymbol{\alpha}_{4}: \left(\alpha_{1}, \alpha_{2}, \alpha_{3}, \alpha_{4}, \alpha_{4}\right) \text { can be partitioned into } (m, n) \in \boldsymbol{D}_{0}^{\prime} \text { or }(m, n, h) \in \boldsymbol{D}_{1}^{\prime} \cup \boldsymbol{D}_{2}^{\prime} \right\}, \\
\nonumber \boldsymbol{D}^{**} =&\ \left\{\boldsymbol{\alpha}_{6}: \left(\alpha_{1}, \alpha_{2}, \alpha_{3}, \alpha_{4}, \alpha_{5}, \alpha_{6}, \alpha_{6}\right) \text { can be partitioned into } (m, n) \in \boldsymbol{D}_{0}^{\prime} \text { or }(m, n, h) \in \boldsymbol{D}_{1}^{\prime} \cup \boldsymbol{D}_{2}^{\prime} \right\}, \\
\nonumber \boldsymbol{D}^{\dag} =&\ \left\{\boldsymbol{\alpha}_{4}: \boldsymbol{\alpha}_{4} \text { can be partitioned into } (m, n) \in \boldsymbol{D}_{0} \text { or }(m, n, h) \in \boldsymbol{D}_{1} \cup \boldsymbol{D}_{2} \right\}, \\
\nonumber \boldsymbol{D}^{\ddag} =&\ \left\{\boldsymbol{\alpha}_{4}: \boldsymbol{\alpha}_{4} \in \boldsymbol{D}^{\dag},\ \left(1- \alpha_{1} - \alpha_2 - \alpha_3 - \alpha_4, \alpha_{2}, \alpha_{3}, \alpha_{4} \right) \in \boldsymbol{D}^{\dag} \right\},
\end{align}
where the sets $\boldsymbol{D}_0$, $\boldsymbol{D}_1$ and $\boldsymbol{D}_2$ correspond to conditions on variables that allow a further decomposition (that is, we apply Buchstab's identity twice), $\boldsymbol{D}_{i}^{\prime}$ is a simplified version of $\boldsymbol{D}_{i}$ for $i \in \{0,1,2\}$, $\boldsymbol{D}^{*}$ and $\boldsymbol{D}^{**}$ allow two and three further decompositions respectively, and $\boldsymbol{D}^{\ddag}$ allows two further decompositions with a variable role-reversal. For example, if we have $\boldsymbol{\alpha}_{4} \in \boldsymbol{D}^{*}$ after applying Buchstab's identity twice, we can apply Buchstab's identity twice more because we can obtain an asymptotic formula for $S\left(\mathcal{A}_{p_1 p_2 p_3 p_4 p_5}, x^{\nu}\right)$ for all $\nu \leqslant \alpha_5 < \alpha_4$. The definition of a role-reversal can be found in \cite{BHP} at the bottom of page 533 and at the top of page 534. It can be seen as ``an application of Buchstab's identity on $p_1$''. In regions corresponding to neither $\boldsymbol{G}$ nor $\boldsymbol{D}$, we sometimes need role-reversals to perform further decompositions.

We remark that if $\left(\alpha_{1}, \alpha_{2}, \alpha_{3}, \alpha_{4}, \alpha_{4}\right)$ can be partitioned into $(m, n) \in \boldsymbol{D}_{0}$ with $m<n$ or one element $\alpha_{4}$ is partitioned into $n$, then we still have this $\boldsymbol{\alpha}_{4} \in \boldsymbol{D}^{*}$ even if $\left(\alpha_{1}, \alpha_{2}, \alpha_{3}, \alpha_{4}, \alpha_{4}\right)$ cannot be partitioned into $(m, n) \in \boldsymbol{D}_{0}^{\prime}$. This is because sometimes we can group $\left(\alpha_{1}, \alpha_{2}, \alpha_{3}, \alpha_{4}, \alpha_{4}\right)$ into $(m, n) \in \boldsymbol{D}_{0}$ but cannot group $\left(\alpha_{1}, \alpha_{2}, \alpha_{3}, \alpha_{4}, \alpha_{5}\right)$ into $(m, n) \in \boldsymbol{D}_{0}$ for some $\nu \leqslant \alpha_5 < \alpha_4$ due to the involvement of $\alpha^{*}(m)$ in the upper bound of $n$. That is, if we group $\left(\alpha_{1}, \alpha_{2}, \alpha_{3}, \alpha_{4}, \alpha_{4}\right)$ into $(m, n)$ which lies in the areas above the line $n = \frac{3\theta-1}{2}$ (see the protrusions at the top of the region $\boldsymbol{D}_{0}$ in Figure 1 in Appendix 1) and all of two elements $\alpha_4$ are partitioned into $m$ (note that $n$ is a constant in this partition), then we cannot group all $\left(\alpha_{1}, \alpha_{2}, \alpha_{3}, \alpha_{4}, \alpha_{5}\right)$ with $\nu \leqslant \alpha_5 < \alpha_4$ because some values of $\alpha_5$ may leads to some smaller $m$ where the new $(m, n)$ lies in the concave areas on the left of the original $(m, n)$. Otherwise we have at least one $\alpha_4$ is in $n$, and we can let this $\alpha_4$ to be the variable $\alpha_5$ runs over values less than $\alpha_4$.
In $\boldsymbol{D}_{0}^{\prime}$ the function $\alpha^{*}$ is replaced by an upper bound $\frac{1}{2}$, and the shape of $\boldsymbol{D}_{0}^{\prime}$ is a rectangle with bounds $0 \leqslant m \leqslant \frac{1}{2}$ and $0 \leqslant n \leqslant \frac{3 \theta-1}{2}$. Thus, we can group $\left(\alpha_{1}, \alpha_{2}, \alpha_{3}, \alpha_{4}, \alpha_{5}\right)$ into $(m, n) \in \boldsymbol{D}_{0}^{\prime}$ for any $\nu \leqslant \alpha_5 < \alpha_4$ if we can group $\left(\alpha_{1}, \alpha_{2}, \alpha_{3}, \alpha_{4}, \alpha_{4}\right)$ into $(m, n) \in \boldsymbol{D}_{0}^{\prime}$. However, if $m<n$, we can change the role of $m$ and $n$ so that at least one element $\alpha_4$ is in $n$, hence we can group $\left(\alpha_{1}, \alpha_{2}, \alpha_{3}, \alpha_{4}, \alpha_{5}\right)$ into $(m, n) \in \boldsymbol{D}_{0}$ for any $\nu \leqslant \alpha_5 < \alpha_4$. The similar phenomenon holds for $\boldsymbol{D}^{**}$.

Now we split the region defined by $\sum_3$ into three subregions $A, B, C$ corresponding to the different techniques that should be applied. The plot of these regions can be found in Appendix 1.
\begin{align}
\nonumber A =&\ \left\{\boldsymbol{\alpha}_{2}: \frac{1}{4} \leqslant \alpha_{1} \leqslant \frac{2}{5},\ \frac{1}{3}\left(1-\alpha_{1}\right) \leqslant \alpha_{2} \leqslant \min \left(\alpha_{1}, 1-2 \alpha_{1}\right) \right\};\\
\nonumber B =&\ \left\{\boldsymbol{\alpha}_{2}: \frac{1}{3} \leqslant \alpha_{1} \leqslant \frac{1}{2},\ \max \left(\frac{1}{2} \alpha_{1}, 1-2 \alpha_{1}\right) \leqslant \alpha_{2} \leqslant \frac{1}{2}\left(1-\alpha_{1}\right) \right\};\\
\nonumber C =&\ \left\{\boldsymbol{\alpha}_{2}: \nu(0) \leqslant \alpha_{1} \leqslant \frac{1}{2},\ \nu\left(\alpha_{1}\right) \leqslant \alpha_{2} \leqslant \min \left(\alpha_{1}, \frac{1}{2}\left(1-\alpha_{1}\right)\right),\ \boldsymbol{\alpha}_{2} \notin A \cup B \right\}.
\end{align}
We note that $\left(\alpha_{1}, \alpha_{2}\right) \in A \Leftrightarrow\left(1-\alpha_{1}-\alpha_{2}, \alpha_{2}\right) \in B$. Since in $A \cup B$ only products of three primes are counted, we have
\begin{align}
\sum_{\boldsymbol{\alpha}_{2} \in A} S\left(\mathcal{A}_{p_{1} p_{2}}, p_{2}\right) &= \sum_{\boldsymbol{\alpha}_{2} \in B} S\left(\mathcal{A}_{p_{1} p_{2}}, p_{2}\right),
\end{align}
hence
\begin{align}
\nonumber \sum_{3} =&\ 2 \sum_{\boldsymbol{\alpha}_{2} \in A} S\left(\mathcal{A}_{p_{1} p_{2}}, p_{2}\right) + \sum_{\boldsymbol{\alpha}_{2} \in C} S\left(\mathcal{A}_{p_{1} p_{2}}, p_{2}\right) \\
=&\ 2 \sum_{A} + \sum_{C}.
\end{align}

We first consider $\sum_{A}$. Discarding the whole of $\sum_{A}$ leads to a loss of
\begin{equation}
\int_{\frac{1}{4}}^{\frac{2}{5}} \int_{\frac{1 - t_1}{3}}^{\min\left(t_1, 1 - 2 t_1\right)} \frac{\omega\left(\frac{1 - t_1 - t_2}{t_2}\right)}{t_1 t_2^2} d t_2 d t_1 < 0.240227.
\end{equation}
For Theorem~\ref{t1B} we use the bound (12) directly. However, we shall give a method on how to make possible savings over this sum (and we will actually use this process as the first step in our decomposition in region $C$). Let $A^{\prime}$ denote a part of region $A$ that satisfies $\alpha_2 \leqslant \frac{3 \theta - 1}{2}$. By Buchstab's identity, we have
\begin{equation}
\sum_{\boldsymbol{\alpha}_{2} \in A^{\prime}} S\left(\mathcal{A}_{p_{1} p_{2}}, p_{2}\right) = \sum_{\boldsymbol{\alpha}_{2} \in A^{\prime}} S\left(\mathcal{A}_{p_{1} p_{2}}, x^{\nu_0}\right)-\sum_{\substack{\boldsymbol{\alpha}_{2} \in A^{\prime} \\ \nu_0 \leqslant \alpha_{3} < \min\left(\alpha_2, \frac{1}{2}(1-\alpha_1 -\alpha_2)\right) } }S\left(\mathcal{A}_{p_{1} p_{2} p_{3}}, p_{3}\right).
\end{equation}
We can give an asymptotic formula for the first sum on the right-hand side. For the second sum, we can perform a straightforward decomposition by applying Buchstab's identity twice more if we can group $\boldsymbol{\alpha}_{3}$ into $(m, n) \in \boldsymbol{D}_{0}$ or $(m, n, h) \in \boldsymbol{D}_{1} \cup \boldsymbol{D}_{2}$. For the remaining part of the second sum, we note that $\alpha_1 +\alpha_2 \geqslant \frac{1}{2}$ and $\alpha_3 < \alpha_2 \leqslant \frac{3\theta -1}{2} \leqslant \frac{3 \theta+1-4 \alpha_1^{*}}{2}$. Let $p_{1} p_{2} p_{3} \beta$ denote the numbers counted by $S\left(\mathcal{A}_{p_{1} p_{2} p_{3}}, p_{3}\right)$, we have $\beta \sim x^{1-\alpha_1 -\alpha_2 -\alpha_3}$. Thus, we can perform a role-reversal on the ramaining part because we have $((1-\alpha_1 -\alpha_2 -\alpha_3)+ \alpha_3, \alpha_2) \in \boldsymbol{D}_{0}$ in this case. Altogether, we have the following expression after the first decomposition procedure:
\begin{align}
\nonumber \sum_{\boldsymbol{\alpha}_{2} \in A^{\prime}} S\left(\mathcal{A}_{p_{1} p_{2}}, p_{2}\right) =&\ \sum_{\boldsymbol{\alpha}_{2} \in A^{\prime}} S\left(\mathcal{A}_{p_{1} p_{2}}, x^{\nu_0}\right)-\sum_{\substack{\boldsymbol{\alpha}_{2} \in A^{\prime} \\ \nu_0 \leqslant \alpha_{3} < \min\left(\alpha_2, \frac{1}{2}(1-\alpha_1 -\alpha_2)\right) } }S\left(\mathcal{A}_{p_{1} p_{2} p_{3}}, p_{3}\right) \\
\nonumber =&\ \sum_{\boldsymbol{\alpha}_{2} \in A^{\prime}} S\left(\mathcal{A}_{p_{1} p_{2}}, x^{\nu_0}\right)-\sum_{\substack{\boldsymbol{\alpha}_{2} \in A^{\prime} \\ \nu_0 \leqslant \alpha_{3} < \min\left(\alpha_2, \frac{1}{2}(1-\alpha_1 -\alpha_2)\right) \\ \boldsymbol{\alpha}_{3} \in \boldsymbol{G}_{3} } }S\left(\mathcal{A}_{p_{1} p_{2} p_{3}}, p_{3}\right) \\
\nonumber & -\sum_{\substack{\boldsymbol{\alpha}_{2} \in A^{\prime} \\ \nu_0 \leqslant \alpha_{3} < \min\left(\alpha_2, \frac{1}{2}(1-\alpha_1 -\alpha_2)\right) \\ \boldsymbol{\alpha}_{3} \notin \boldsymbol{G}_{3} \\ \boldsymbol{\alpha}_{3} \text{ can be partitioned into } (m, n) \in \boldsymbol{D}_{0} \text{ or } (m, n, h) \in \boldsymbol{D}_{1} \cup \boldsymbol{D}_{2} } }S\left(\mathcal{A}_{p_{1} p_{2} p_{3}}, p_{3}\right) \\
\nonumber & -\sum_{\substack{\boldsymbol{\alpha}_{2} \in A^{\prime} \\ \nu_0 \leqslant \alpha_{3} < \min\left(\alpha_2, \frac{1}{2}(1-\alpha_1 -\alpha_2)\right) \\ \boldsymbol{\alpha}_{3} \notin \boldsymbol{G}_{3} \\ \boldsymbol{\alpha}_{3} \text{ cannot be partitioned into } (m, n) \in \boldsymbol{D}_{0} \text{ or } (m, n, h) \in \boldsymbol{D}_{1} \cup \boldsymbol{D}_{2} } }S\left(\mathcal{A}_{p_{1} p_{2} p_{3}}, p_{3}\right) \\
\nonumber =&\ \sum_{\boldsymbol{\alpha}_{2} \in A^{\prime}} S\left(\mathcal{A}_{p_{1} p_{2}}, x^{\nu_0}\right)-\sum_{\substack{\boldsymbol{\alpha}_{2} \in A^{\prime} \\ \nu_0 \leqslant \alpha_{3} < \min\left(\alpha_2, \frac{1}{2}(1-\alpha_1 -\alpha_2)\right) \\ \boldsymbol{\alpha}_{3} \in \boldsymbol{G}_{3} } }S\left(\mathcal{A}_{p_{1} p_{2} p_{3}}, p_{3}\right) \\
\nonumber & -\sum_{\substack{\boldsymbol{\alpha}_{2} \in A^{\prime} \\ \nu_0 \leqslant \alpha_{3} < \min\left(\alpha_2, \frac{1}{2}(1-\alpha_1 -\alpha_2)\right) \\ \boldsymbol{\alpha}_{3} \notin \boldsymbol{G}_{3} \\ \boldsymbol{\alpha}_{3} \text{ can be partitioned into } (m, n) \in \boldsymbol{D}_{0} \text{ or } (m, n, h) \in \boldsymbol{D}_{1} \cup \boldsymbol{D}_{2} } }S\left(\mathcal{A}_{p_{1} p_{2} p_{3}}, x^{\nu_0}\right) \\
\nonumber & +\sum_{\substack{\boldsymbol{\alpha}_{2} \in A^{\prime} \\ \nu_0 \leqslant \alpha_{3} < \min\left(\alpha_2, \frac{1}{2}(1-\alpha_1 -\alpha_2)\right) \\ \boldsymbol{\alpha}_{3} \notin \boldsymbol{G}_{3} \\ \boldsymbol{\alpha}_{3} \text{ can be partitioned into } (m, n) \in \boldsymbol{D}_{0} \text{ or } (m, n, h) \in \boldsymbol{D}_{1} \cup \boldsymbol{D}_{2} \\ \nu_0 \leqslant \alpha_{4} < \min\left(\alpha_3, \frac{1}{2}(1-\alpha_1 -\alpha_2 -\alpha_3)\right) \\ \boldsymbol{\alpha}_{4} \in \boldsymbol{G}_{4} } }S\left(\mathcal{A}_{p_{1} p_{2} p_{3} p_{4}}, p_{4}\right) \\
\nonumber & +\sum_{\substack{\boldsymbol{\alpha}_{2} \in A^{\prime} \\ \nu_0 \leqslant \alpha_{3} < \min\left(\alpha_2, \frac{1}{2}(1-\alpha_1 -\alpha_2)\right) \\ \boldsymbol{\alpha}_{3} \notin \boldsymbol{G}_{3} \\ \boldsymbol{\alpha}_{3} \text{ can be partitioned into } (m, n) \in \boldsymbol{D}_{0} \text{ or } (m, n, h) \in \boldsymbol{D}_{1} \cup \boldsymbol{D}_{2} \\ \nu_0 \leqslant \alpha_{4} < \min\left(\alpha_3, \frac{1}{2}(1-\alpha_1 -\alpha_2 -\alpha_3)\right) \\ \boldsymbol{\alpha}_{4} \notin \boldsymbol{G}_{4} } }S\left(\mathcal{A}_{p_{1} p_{2} p_{3} p_{4}}, p_{4}\right) \\
\nonumber & -\sum_{\substack{\boldsymbol{\alpha}_{2} \in A^{\prime} \\ \nu_0 \leqslant \alpha_{3} < \min\left(\alpha_2, \frac{1}{2}(1-\alpha_1 -\alpha_2)\right) \\ \boldsymbol{\alpha}_{3} \notin \boldsymbol{G}_{3} \\ \boldsymbol{\alpha}_{3} \text{ cannot be partitioned into } (m, n) \in \boldsymbol{D}_{0} \text{ or } (m, n, h) \in \boldsymbol{D}_{1} \cup \boldsymbol{D}_{2} } }S\left(\mathcal{A}_{\beta p_{2} p_{3}}, x^{\nu_0}\right) \\
\nonumber & +\sum_{\substack{\boldsymbol{\alpha}_{2} \in A^{\prime} \\ \nu_0 \leqslant \alpha_{3} < \min\left(\alpha_2, \frac{1}{2}(1-\alpha_1 -\alpha_2)\right) \\ \boldsymbol{\alpha}_{3} \notin \boldsymbol{G}_{3} \\ \boldsymbol{\alpha}_{3} \text{ cannot be partitioned into } (m, n) \in \boldsymbol{D}_{0} \text{ or } (m, n, h) \in \boldsymbol{D}_{1} \cup \boldsymbol{D}_{2} \\ \nu_0 \leqslant \alpha_{4} < \frac{1}{2}\alpha_1 \\ \boldsymbol{\alpha}_{4}^{\prime} \in \boldsymbol{G}_{4} } }S\left(\mathcal{A}_{\beta p_{2} p_{3} p_{4}}, p_{4}\right) \\
\nonumber & +\sum_{\substack{\boldsymbol{\alpha}_{2} \in A^{\prime} \\ \nu_0 \leqslant \alpha_{3} < \min\left(\alpha_2, \frac{1}{2}(1-\alpha_1 -\alpha_2)\right) \\ \boldsymbol{\alpha}_{3} \notin \boldsymbol{G}_{3} \\ \boldsymbol{\alpha}_{3} \text{ cannot be partitioned into } (m, n) \in \boldsymbol{D}_{0} \text{ or } (m, n, h) \in \boldsymbol{D}_{1} \cup \boldsymbol{D}_{2} \\ \nu_0 \leqslant \alpha_{4} < \frac{1}{2}\alpha_1 \\ \boldsymbol{\alpha}_{4}^{\prime} \notin \boldsymbol{G}_{4} } }S\left(\mathcal{A}_{\beta p_{2} p_{3} p_{4}}, p_{4}\right) \\
=&\ S_{01} - S_{02} - S_{03} + S_{04} + T_{01} - S_{05} + S_{06} + T_{02},
\end{align}
where $\left(\beta, P(p_{3})\right)=1$ and
$$
\boldsymbol{\alpha}_{4}^{\prime}=(1-\alpha_1-\alpha_2-\alpha_3,\ \alpha_2,\ \alpha_3,\ \alpha_4).
$$

We can give asymptotic formulas for $S_{01}$--$S_{06}$. For $T_{01}$ we can perform Buchstab's identity twice more to reach a six-dimensional sum if $\boldsymbol{\alpha}_{4} \in \boldsymbol{D}^{*}$, and we can use Buchstab's identity in a different way (which will be explained later) to make some savings on the remaining parts. After the second decomposition procedure on $T_{01}$, we have
\begin{align}
\nonumber T_{01} =&\ \sum_{\substack{\boldsymbol{\alpha}_{2} \in A^{\prime} \\ \nu_0 \leqslant \alpha_{3} < \min\left(\alpha_2, \frac{1}{2}(1-\alpha_1 -\alpha_2)\right) \\ \boldsymbol{\alpha}_{3} \notin \boldsymbol{G}_{3} \\ \boldsymbol{\alpha}_{3} \text{ can be partitioned into } (m, n) \in \boldsymbol{D}_{0} \text{ or } (m, n, h) \in \boldsymbol{D}_{1} \cup \boldsymbol{D}_{2} \\ \nu_0 \leqslant \alpha_{4} < \min\left(\alpha_3, \frac{1}{2}(1-\alpha_1 -\alpha_2 -\alpha_3)\right) \\ \boldsymbol{\alpha}_{4} \notin \boldsymbol{G}_{4} } }S\left(\mathcal{A}_{p_{1} p_{2} p_{3} p_{4}}, p_{4}\right) \\
\nonumber =&\ \sum_{\substack{\boldsymbol{\alpha}_{2} \in A^{\prime} \\ \nu_0 \leqslant \alpha_{3} < \min\left(\alpha_2, \frac{1}{2}(1-\alpha_1 -\alpha_2)\right) \\ \boldsymbol{\alpha}_{3} \notin \boldsymbol{G}_{3} \\ \boldsymbol{\alpha}_{3} \text{ can be partitioned into } (m, n) \in \boldsymbol{D}_{0} \text{ or } (m, n, h) \in \boldsymbol{D}_{1} \cup \boldsymbol{D}_{2} \\ \nu_0 \leqslant \alpha_{4} < \min\left(\alpha_3, \frac{1}{2}(1-\alpha_1 -\alpha_2 -\alpha_3)\right) \\ \boldsymbol{\alpha}_{4} \notin \boldsymbol{G}_{4},\ \boldsymbol{\alpha}_{4} \notin \boldsymbol{D}^{*} } }S\left(\mathcal{A}_{p_{1} p_{2} p_{3} p_{4}}, p_{4}\right) \\
\nonumber &+ \sum_{\substack{\boldsymbol{\alpha}_{2} \in A^{\prime} \\ \nu_0 \leqslant \alpha_{3} < \min\left(\alpha_2, \frac{1}{2}(1-\alpha_1 -\alpha_2)\right) \\ \boldsymbol{\alpha}_{3} \notin \boldsymbol{G}_{3} \\ \boldsymbol{\alpha}_{3} \text{ can be partitioned into } (m, n) \in \boldsymbol{D}_{0} \text{ or } (m, n, h) \in \boldsymbol{D}_{1} \cup \boldsymbol{D}_{2} \\ \nu_0 \leqslant \alpha_{4} < \min\left(\alpha_3, \frac{1}{2}(1-\alpha_1 -\alpha_2 -\alpha_3)\right) \\ \boldsymbol{\alpha}_{4} \notin \boldsymbol{G}_{4},\ \boldsymbol{\alpha}_{4} \in \boldsymbol{D}^{*} } }S\left(\mathcal{A}_{p_{1} p_{2} p_{3} p_{4}}, x^{\nu_0}\right) \\
\nonumber =&\ \sum_{\substack{\boldsymbol{\alpha}_{2} \in A^{\prime} \\ \nu_0 \leqslant \alpha_{3} < \min\left(\alpha_2, \frac{1}{2}(1-\alpha_1 -\alpha_2)\right) \\ \boldsymbol{\alpha}_{3} \notin \boldsymbol{G}_{3} \\ \boldsymbol{\alpha}_{3} \text{ can be partitioned into } (m, n) \in \boldsymbol{D}_{0} \text{ or } (m, n, h) \in \boldsymbol{D}_{1} \cup \boldsymbol{D}_{2} \\ \nu_0 \leqslant \alpha_{4} < \min\left(\alpha_3, \frac{1}{2}(1-\alpha_1 -\alpha_2 -\alpha_3)\right) \\ \boldsymbol{\alpha}_{4} \notin \boldsymbol{G}_{4},\ \boldsymbol{\alpha}_{4} \notin \boldsymbol{D}^{*} } }S\left(\mathcal{A}_{p_{1} p_{2} p_{3} p_{4}}, p_{4}\right) \\
\nonumber &+ \sum_{\substack{\boldsymbol{\alpha}_{2} \in A^{\prime} \\ \nu_0 \leqslant \alpha_{3} < \min\left(\alpha_2, \frac{1}{2}(1-\alpha_1 -\alpha_2)\right) \\ \boldsymbol{\alpha}_{3} \notin \boldsymbol{G}_{3} \\ \boldsymbol{\alpha}_{3} \text{ can be partitioned into } (m, n) \in \boldsymbol{D}_{0} \text{ or } (m, n, h) \in \boldsymbol{D}_{1} \cup \boldsymbol{D}_{2} \\ \nu_0 \leqslant \alpha_{4} < \min\left(\alpha_3, \frac{1}{2}(1-\alpha_1 -\alpha_2 -\alpha_3)\right) \\ \boldsymbol{\alpha}_{4} \notin \boldsymbol{G}_{4},\ \boldsymbol{\alpha}_{4} \in \boldsymbol{D}^{*} } }S\left(\mathcal{A}_{p_{1} p_{2} p_{3} p_{4}}, x^{\nu_0}\right) \\
\nonumber &- \sum_{\substack{\boldsymbol{\alpha}_{2} \in A^{\prime} \\ \nu_0 \leqslant \alpha_{3} < \min\left(\alpha_2, \frac{1}{2}(1-\alpha_1 -\alpha_2)\right) \\ \boldsymbol{\alpha}_{3} \notin \boldsymbol{G}_{3} \\ \boldsymbol{\alpha}_{3} \text{ can be partitioned into } (m, n) \in \boldsymbol{D}_{0} \text{ or } (m, n, h) \in \boldsymbol{D}_{1} \cup \boldsymbol{D}_{2} \\ \nu_0 \leqslant \alpha_{4} < \min\left(\alpha_3, \frac{1}{2}(1-\alpha_1 -\alpha_2 -\alpha_3)\right) \\ \boldsymbol{\alpha}_{4} \notin \boldsymbol{G}_{4},\ \boldsymbol{\alpha}_{4} \in \boldsymbol{D}^{*} \\ \nu_0 \leqslant \alpha_{5} < \min\left(\alpha_4, \frac{1}{2}(1-\alpha_1 -\alpha_2 -\alpha_3 - \alpha_4)\right) \\ \boldsymbol{\alpha}_{5} \in \boldsymbol{G}_{5} } }S\left(\mathcal{A}_{p_{1} p_{2} p_{3} p_{4} p_{5}}, p_{5}\right) \\
\nonumber &- \sum_{\substack{\boldsymbol{\alpha}_{2} \in A^{\prime} \\ \nu_0 \leqslant \alpha_{3} < \min\left(\alpha_2, \frac{1}{2}(1-\alpha_1 -\alpha_2)\right) \\ \boldsymbol{\alpha}_{3} \notin \boldsymbol{G}_{3} \\ \boldsymbol{\alpha}_{3} \text{ can be partitioned into } (m, n) \in \boldsymbol{D}_{0} \text{ or } (m, n, h) \in \boldsymbol{D}_{1} \cup \boldsymbol{D}_{2} \\ \nu_0 \leqslant \alpha_{4} < \min\left(\alpha_3, \frac{1}{2}(1-\alpha_1 -\alpha_2 -\alpha_3)\right) \\ \boldsymbol{\alpha}_{4} \notin \boldsymbol{G}_{4},\ \boldsymbol{\alpha}_{4} \in \boldsymbol{D}^{*} \\ \nu_0 \leqslant \alpha_{5} < \min\left(\alpha_4, \frac{1}{2}(1-\alpha_1 -\alpha_2 -\alpha_3 - \alpha_4)\right) \\ \boldsymbol{\alpha}_{5} \notin \boldsymbol{G}_{5} } }S\left(\mathcal{A}_{p_{1} p_{2} p_{3} p_{4} p_{5}}, x^{\nu_0}\right) \\
\nonumber &+ \sum_{\substack{\boldsymbol{\alpha}_{2} \in A^{\prime} \\ \nu_0 \leqslant \alpha_{3} < \min\left(\alpha_2, \frac{1}{2}(1-\alpha_1 -\alpha_2)\right) \\ \boldsymbol{\alpha}_{3} \notin \boldsymbol{G}_{3} \\ \boldsymbol{\alpha}_{3} \text{ can be partitioned into } (m, n) \in \boldsymbol{D}_{0} \text{ or } (m, n, h) \in \boldsymbol{D}_{1} \cup \boldsymbol{D}_{2} \\ \nu_0 \leqslant \alpha_{4} < \min\left(\alpha_3, \frac{1}{2}(1-\alpha_1 -\alpha_2 -\alpha_3)\right) \\ \boldsymbol{\alpha}_{4} \notin \boldsymbol{G}_{4},\ \boldsymbol{\alpha}_{4} \in \boldsymbol{D}^{*} \\ \nu_0 \leqslant \alpha_{5} < \min\left(\alpha_4, \frac{1}{2}(1-\alpha_1 -\alpha_2 -\alpha_3 - \alpha_4)\right) \\ \boldsymbol{\alpha}_{5} \notin \boldsymbol{G}_{5} \\ \nu_0 \leqslant \alpha_{6} < \min\left(\alpha_5, \frac{1}{2}(1-\alpha_1 -\alpha_2 -\alpha_3 - \alpha_5 - \alpha_6)\right) \\ \boldsymbol{\alpha}_{6} \in \boldsymbol{G}_{6}} }S\left(\mathcal{A}_{p_{1} p_{2} p_{3} p_{4} p_{5} p_{6}}, p_{6}\right) \\
\nonumber &+ \sum_{\substack{\boldsymbol{\alpha}_{2} \in A^{\prime} \\ \nu_0 \leqslant \alpha_{3} < \min\left(\alpha_2, \frac{1}{2}(1-\alpha_1 -\alpha_2)\right) \\ \boldsymbol{\alpha}_{3} \notin \boldsymbol{G}_{3} \\ \boldsymbol{\alpha}_{3} \text{ can be partitioned into } (m, n) \in \boldsymbol{D}_{0} \text{ or } (m, n, h) \in \boldsymbol{D}_{1} \cup \boldsymbol{D}_{2} \\ \nu_0 \leqslant \alpha_{4} < \min\left(\alpha_3, \frac{1}{2}(1-\alpha_1 -\alpha_2 -\alpha_3)\right) \\ \boldsymbol{\alpha}_{4} \notin \boldsymbol{G}_{4},\ \boldsymbol{\alpha}_{4} \in \boldsymbol{D}^{*} \\ \nu_0 \leqslant \alpha_{5} < \min\left(\alpha_4, \frac{1}{2}(1-\alpha_1 -\alpha_2 -\alpha_3 - \alpha_4)\right) \\ \boldsymbol{\alpha}_{5} \notin \boldsymbol{G}_{5} \\ \nu_0 \leqslant \alpha_{6} < \min\left(\alpha_5, \frac{1}{2}(1-\alpha_1 -\alpha_2 -\alpha_3 - \alpha_5 - \alpha_6)\right) \\ \boldsymbol{\alpha}_{6} \notin \boldsymbol{G}_{6}} }S\left(\mathcal{A}_{p_{1} p_{2} p_{3} p_{4} p_{5} p_{6}}, p_{6}\right) \\
=&\ T_{011} + S_{07} - S_{08} - S_{09} + S_{10} + T_{012}.
\end{align}
We can give asymptotic formulas for $S_{07}$--$S_{10}$. The sum $T_{011}$ can be further decomposed to
\begin{align}
\nonumber T_{011} =&\ \sum_{\substack{\boldsymbol{\alpha}_{2} \in A^{\prime} \\ \nu_0 \leqslant \alpha_{3} < \min\left(\alpha_2, \frac{1}{2}(1-\alpha_1 -\alpha_2)\right) \\ \boldsymbol{\alpha}_{3} \notin \boldsymbol{G}_{3} \\ \boldsymbol{\alpha}_{3} \text{ can be partitioned into } (m, n) \in \boldsymbol{D}_{0} \text{ or } (m, n, h) \in \boldsymbol{D}_{1} \cup \boldsymbol{D}_{2} \\ \nu_0 \leqslant \alpha_{4} < \min\left(\alpha_3, \frac{1}{2}(1-\alpha_1 -\alpha_2 -\alpha_3)\right) \\ \boldsymbol{\alpha}_{4} \notin \boldsymbol{G}_{4},\ \boldsymbol{\alpha}_{4} \notin \boldsymbol{D}^{*} } }S\left(\mathcal{A}_{p_{1} p_{2} p_{3} p_{4}}, p_{4}\right) \\
\nonumber =&\ \sum_{\substack{\boldsymbol{\alpha}_{2} \in A^{\prime} \\ \nu_0 \leqslant \alpha_{3} < \min\left(\alpha_2, \frac{1}{2}(1-\alpha_1 -\alpha_2)\right) \\ \boldsymbol{\alpha}_{3} \notin \boldsymbol{G}_{3} \\ \boldsymbol{\alpha}_{3} \text{ can be partitioned into } (m, n) \in \boldsymbol{D}_{0} \text{ or } (m, n, h) \in \boldsymbol{D}_{1} \cup \boldsymbol{D}_{2} \\ \nu_0 \leqslant \alpha_{4} < \min\left(\alpha_3, \frac{1}{2}(1-\alpha_1 -\alpha_2 -\alpha_3)\right) \\ \boldsymbol{\alpha}_{4} \notin \boldsymbol{G}_{4},\ \boldsymbol{\alpha}_{4} \notin \boldsymbol{D}^{*},\ \alpha_4 \geqslant \frac{1}{2}(1-\alpha_1 -\alpha_2 -\alpha_3 -\alpha_4) } }S\left(\mathcal{A}_{p_{1} p_{2} p_{3} p_{4}}, p_{4}\right) \\
\nonumber &+ \sum_{\substack{\boldsymbol{\alpha}_{2} \in A^{\prime} \\ \nu_0 \leqslant \alpha_{3} < \min\left(\alpha_2, \frac{1}{2}(1-\alpha_1 -\alpha_2)\right) \\ \boldsymbol{\alpha}_{3} \notin \boldsymbol{G}_{3} \\ \boldsymbol{\alpha}_{3} \text{ can be partitioned into } (m, n) \in \boldsymbol{D}_{0} \text{ or } (m, n, h) \in \boldsymbol{D}_{1} \cup \boldsymbol{D}_{2} \\ \nu_0 \leqslant \alpha_{4} < \min\left(\alpha_3, \frac{1}{2}(1-\alpha_1 -\alpha_2 -\alpha_3)\right) \\ \boldsymbol{\alpha}_{4} \notin \boldsymbol{G}_{4},\ \boldsymbol{\alpha}_{4} \notin \boldsymbol{D}^{*},\ \alpha_4 < \frac{1}{2}(1-\alpha_1 -\alpha_2 -\alpha_3 -\alpha_4) } }S\left(\mathcal{A}_{p_{1} p_{2} p_{3} p_{4}}, p_{4}\right) \\
\nonumber =&\ \sum_{\substack{\boldsymbol{\alpha}_{2} \in A^{\prime} \\ \nu_0 \leqslant \alpha_{3} < \min\left(\alpha_2, \frac{1}{2}(1-\alpha_1 -\alpha_2)\right) \\ \boldsymbol{\alpha}_{3} \notin \boldsymbol{G}_{3} \\ \boldsymbol{\alpha}_{3} \text{ can be partitioned into } (m, n) \in \boldsymbol{D}_{0} \text{ or } (m, n, h) \in \boldsymbol{D}_{1} \cup \boldsymbol{D}_{2} \\ \nu_0 \leqslant \alpha_{4} < \min\left(\alpha_3, \frac{1}{2}(1-\alpha_1 -\alpha_2 -\alpha_3)\right) \\ \boldsymbol{\alpha}_{4} \notin \boldsymbol{G}_{4},\ \boldsymbol{\alpha}_{4} \notin \boldsymbol{D}^{*},\ \alpha_4 \geqslant \frac{1}{2}(1-\alpha_1 -\alpha_2 -\alpha_3 -\alpha_4) } }S\left(\mathcal{A}_{p_{1} p_{2} p_{3} p_{4}}, p_{4}\right) \\
\nonumber &+ \sum_{\substack{\boldsymbol{\alpha}_{2} \in A^{\prime} \\ \nu_0 \leqslant \alpha_{3} < \min\left(\alpha_2, \frac{1}{2}(1-\alpha_1 -\alpha_2)\right) \\ \boldsymbol{\alpha}_{3} \notin \boldsymbol{G}_{3} \\ \boldsymbol{\alpha}_{3} \text{ can be partitioned into } (m, n) \in \boldsymbol{D}_{0} \text{ or } (m, n, h) \in \boldsymbol{D}_{1} \cup \boldsymbol{D}_{2} \\ \nu_0 \leqslant \alpha_{4} < \min\left(\alpha_3, \frac{1}{2}(1-\alpha_1 -\alpha_2 -\alpha_3)\right) \\ \boldsymbol{\alpha}_{4} \notin \boldsymbol{G}_{4},\ \boldsymbol{\alpha}_{4} \notin \boldsymbol{D}^{*},\ \alpha_4 < \frac{1}{2}(1-\alpha_1 -\alpha_2 -\alpha_3 -\alpha_4) } }S\left(\mathcal{A}_{p_{1} p_{2} p_{3} p_{4}}, \left(\frac{x}{p_1 p_2 p_3 p_4}\right)^{\frac{1}{2}}\right) \\
\nonumber &+ \sum_{\substack{\boldsymbol{\alpha}_{2} \in A^{\prime} \\ \nu_0 \leqslant \alpha_{3} < \min\left(\alpha_2, \frac{1}{2}(1-\alpha_1 -\alpha_2)\right) \\ \boldsymbol{\alpha}_{3} \notin \boldsymbol{G}_{3} \\ \boldsymbol{\alpha}_{3} \text{ can be partitioned into } (m, n) \in \boldsymbol{D}_{0} \text{ or } (m, n, h) \in \boldsymbol{D}_{1} \cup \boldsymbol{D}_{2} \\ \nu_0 \leqslant \alpha_{4} < \min\left(\alpha_3, \frac{1}{2}(1-\alpha_1 -\alpha_2 -\alpha_3)\right) \\ \boldsymbol{\alpha}_{4} \notin \boldsymbol{G}_{4},\ \boldsymbol{\alpha}_{4} \notin \boldsymbol{D}^{*} \\ \alpha_4 < \alpha_5 < \frac{1}{2}(1-\alpha_1 -\alpha_2 -\alpha_3 -\alpha_4) \\ \boldsymbol{\alpha}_{5} \notin \boldsymbol{G}_{5} } }S\left(\mathcal{A}_{p_{1} p_{2} p_{3} p_{4} p_{5}}, p_{5}\right) \\
&+ \sum_{\substack{\boldsymbol{\alpha}_{2} \in A^{\prime} \\ \nu_0 \leqslant \alpha_{3} < \min\left(\alpha_2, \frac{1}{2}(1-\alpha_1 -\alpha_2)\right) \\ \boldsymbol{\alpha}_{3} \notin \boldsymbol{G}_{3} \\ \boldsymbol{\alpha}_{3} \text{ can be partitioned into } (m, n) \in \boldsymbol{D}_{0} \text{ or } (m, n, h) \in \boldsymbol{D}_{1} \cup \boldsymbol{D}_{2} \\ \nu_0 \leqslant \alpha_{4} < \min\left(\alpha_3, \frac{1}{2}(1-\alpha_1 -\alpha_2 -\alpha_3)\right) \\ \boldsymbol{\alpha}_{4} \notin \boldsymbol{G}_{4},\ \boldsymbol{\alpha}_{4} \notin \boldsymbol{D}^{*} \\ \alpha_4 < \alpha_5 < \frac{1}{2}(1-\alpha_1 -\alpha_2 -\alpha_3 -\alpha_4) \\ \boldsymbol{\alpha}_{5} \in \boldsymbol{G}_{5} } }S\left(\mathcal{A}_{p_{1} p_{2} p_{3} p_{4} p_{5}}, p_{5}\right).
\end{align}
The decomposing process (16) can be viewed as a ``splitting'' argument: the numbers $T_{011}$ counts is of the form $p_1 p_2 p_3 p_4 m_0$, where the smallest prime factor of $m_0$ is larger than $p_4$. If $\alpha_4 > \frac{1}{2}(1-\alpha_1 -\alpha_2 -\alpha_3 -\alpha_4)$, we know that $m_0$ cannot have two or more prime factors and therefore must be a prime. Otherwise $m_0$ may have more than one prime factor, which means that we can ``split'' the smallest prime factor of $m$ and consider it separately. The sum $S\left(\mathcal{A}_{p_{1} p_{2} p_{3} p_{4}}, \left(\frac{x}{p_1 p_2 p_3 p_4}\right)^{\frac{1}{2}}\right)$ counts numbers of the form $p_1 p_2 p_3 p_4 p^{\prime}$, where $p^{\prime} = m_0$. The sum $S\left(\mathcal{A}_{p_{1} p_{2} p_{3} p_{4} p_{5}}, p_{5}\right)$ counts numbers of the form $p_1 p_2 p_3 p_4 (p_5 m_1)$, where $p_5 m_1 = m_0$, $p_5 > p_4$ and all prime factors of $m_1$ are larger than $p_5$. In this sum we have a new variable $\alpha_5$, which means that part of this sum may have an asymptotic formula. In this situation, we can give an asymptotic formula for the last sum on the right-hand side of (16), hence we can subtract its contribution from the loss from $T_{011}$. Again, for the remaining part of $S\left(\mathcal{A}_{p_{1} p_{2} p_{3} p_{4} p_{5}}, p_{5}\right)$, we can do a similar process as (16) to subtract the contribution of the numbers $p_1 p_2 p_3 p_4 p_5 (p_6 m_2)$ that have asymptotic formulas. The above process can be rewritten as
\begin{equation}
\sum_{\boldsymbol{\alpha}_{4}} S\left(\mathcal{A}_{p_{1} p_{2} p_{3} p_{4}}, \left(\frac{x}{p_1 p_2 p_3 p_4}\right)^{\frac{1}{2}}\right) = \sum_{\boldsymbol{\alpha}_{4}} S\left(\mathcal{A}_{p_{1} p_{2} p_{3} p_{4}}, p_{4}\right) - \sum_{\substack{\boldsymbol{\alpha}_{4} \\ \alpha_4 < \alpha_5 < \frac{1}{2}(1-\alpha_1 -\alpha_2 -\alpha_3 -\alpha_4)}} S\left(\mathcal{A}_{p_{1} p_{2} p_{3} p_{4} p_{5}}, p_{5}\right).
\end{equation}
Since (17) is a direct application of Buchstab's identity, we shall call this process ``Buchstab's identity in reverse'' or ``reversed Buchstab's identity'' in the rest of our paper. The same process can also be used to deal with $T_{02}$, but we choose to discard all of it for the sake of simplicity. There are also many possible decompositions in some subregions, but we don't consider them here. In fact, we shall consider some of them when decomposing $\sum_{C}$. 

Combining all the sums above with remaining parts of $A$ we get a loss from $\sum_{A}$ of
\begin{align}
\nonumber & \left( \int_{\boldsymbol{t}_{2} \in U_{A1}} \frac{\omega\left(\frac{1 - t_1 - t_2}{t_2}\right)}{t_1 t_2^2} d t_2 d t_1 \right) \\
\nonumber +& \left( \int_{\boldsymbol{t}_{4} \in U_{A2}} \frac{\omega \left(\frac{1 - t_1 - t_2 - t_3 - t_4}{t_4}\right)}{t_1 t_2 t_3 t_4^2} d t_4 d t_3 d t_2 d t_1 \right) \\ 
\nonumber -& \left( \int_{\boldsymbol{t}_{5} \in U_{A3}} \frac{\omega \left(\frac{1 - t_1 - t_2 - t_3 - t_4 - t_5}{t_5}\right)}{t_1 t_2 t_3 t_4 t_5^2} d t_5 d t_4 d t_3 d t_2 d t_1 \right) \\
\nonumber +& \left( \int_{\boldsymbol{t}_{6} \in U_{A4}} \frac{\omega \left(\frac{1 - t_1 - t_2 - t_3 - t_4 - t_5 - t_6}{t_6}\right)}{t_1 t_2 t_3 t_4 t_5 t_6^2} d t_6 d t_5 d t_4 d t_3 d t_2 d t_1 \right) \\
+& \left( \int_{\boldsymbol{t}_{4} \in U_{A5}} \frac{\omega \left(\frac{t_1 - t_4}{t_4}\right) \omega \left(\frac{1 - t_1 - t_2 - t_3}{t_3}\right)}{t_2 t_3^2 t_4^2} d t_4 d t_3 d t_2 d t_1 \right),
\end{align}
where
\begin{align}
\nonumber U_{A1}(\boldsymbol{\alpha}_{2}) :=&\ \left\{ \boldsymbol{\alpha}_{2} \in A \backslash A^{\prime},\ \nu_0 \leqslant \alpha_1 < \frac{1}{2},\ \nu\left(\alpha_1\right) \leqslant \alpha_2 < \min\left(\alpha_1, \frac{1}{2}(1-\alpha_1) \right) \right\}, \\
\nonumber U_{A2}(\boldsymbol{\alpha}_{4}) :=&\ \left\{ \boldsymbol{\alpha}_{2} \in A^{\prime},\ \nu_0 \leqslant \alpha_{3} < \min\left(\alpha_2, \frac{1}{2}(1-\alpha_1 -\alpha_2)\right),\ \boldsymbol{\alpha}_{3} \notin \boldsymbol{G}_{3}, \right. \\
\nonumber & \quad \boldsymbol{\alpha}_{3} \text{ can be partitioned into } (m, n) \in \boldsymbol{D}_{0} \text{ or } (m, n, h) \in \boldsymbol{D}_{1} \cup \boldsymbol{D}_{2}, \\
\nonumber & \quad \nu_0 \leqslant \alpha_{4} < \min\left(\alpha_3, \frac{1}{2}(1-\alpha_1 -\alpha_2 -\alpha_3)\right),\ \boldsymbol{\alpha}_{4} \notin \boldsymbol{G}_{4},\ \boldsymbol{\alpha}_{4} \notin \boldsymbol{D}^{*}, \\
\nonumber & \left. \quad \nu_0 \leqslant \alpha_1 < \frac{1}{2},\ \nu\left(\alpha_1\right) \leqslant \alpha_2 < \min\left(\alpha_1, \frac{1}{2}(1-\alpha_1) \right) \right\}, \\
\nonumber U_{A3}(\boldsymbol{\alpha}_{5}) :=&\ \left\{ \boldsymbol{\alpha}_{2} \in A^{\prime},\ \nu_0 \leqslant \alpha_{3} < \min\left(\alpha_2, \frac{1}{2}(1-\alpha_1 -\alpha_2)\right),\ \boldsymbol{\alpha}_{3} \notin \boldsymbol{G}_{3}, \right. \\
\nonumber & \quad \boldsymbol{\alpha}_{3} \text{ can be partitioned into } (m, n) \in \boldsymbol{D}_{0} \text{ or } (m, n, h) \in \boldsymbol{D}_{1} \cup \boldsymbol{D}_{2}, \\
\nonumber & \quad \nu_0 \leqslant \alpha_{4} < \min\left(\alpha_3, \frac{1}{2}(1-\alpha_1 -\alpha_2 -\alpha_3)\right),\ \boldsymbol{\alpha}_{4} \notin \boldsymbol{G}_{4},\ \boldsymbol{\alpha}_{4} \notin \boldsymbol{D}^{*}, \\
\nonumber & \quad \alpha_4 < \alpha_5 < \frac{1}{2}(1-\alpha_1 -\alpha_2 -\alpha_3 -\alpha_4),\ \boldsymbol{\alpha}_{5} \in \boldsymbol{G}_{5}, \\
\nonumber & \left. \quad \nu_0 \leqslant \alpha_1 < \frac{1}{2},\ \nu\left(\alpha_1\right) \leqslant \alpha_2 < \min\left(\alpha_1, \frac{1}{2}(1-\alpha_1) \right) \right\}, \\
\nonumber U_{A4}(\boldsymbol{\alpha}_{6}) :=&\ \left\{ \boldsymbol{\alpha}_{2} \in A^{\prime},\ \nu_0 \leqslant \alpha_{3} < \min\left(\alpha_2, \frac{1}{2}(1-\alpha_1 -\alpha_2)\right),\ \boldsymbol{\alpha}_{3} \notin \boldsymbol{G}_{3}, \right. \\
\nonumber & \quad \boldsymbol{\alpha}_{3} \text{ can be partitioned into } (m, n) \in \boldsymbol{D}_{0} \text{ or } (m, n, h) \in \boldsymbol{D}_{1} \cup \boldsymbol{D}_{2}, \\
\nonumber & \quad \nu_0 \leqslant \alpha_{4} < \min\left(\alpha_3, \frac{1}{2}(1-\alpha_1 -\alpha_2 -\alpha_3)\right),\ \boldsymbol{\alpha}_{4} \notin \boldsymbol{G}_{4},\ \boldsymbol{\alpha}_{4} \in \boldsymbol{D}^{*}, \\
\nonumber & \quad \nu_0 \leqslant \alpha_{5} < \min\left(\alpha_4, \frac{1}{2}(1-\alpha_1 -\alpha_2 -\alpha_3 -\alpha_4)\right),\ \boldsymbol{\alpha}_{5} \notin \boldsymbol{G}_{5}, \\
\nonumber & \quad \nu_0 \leqslant \alpha_{6} < \min\left(\alpha_5, \frac{1}{2}(1-\alpha_1 -\alpha_2 -\alpha_3 -\alpha_4 -\alpha_5)\right),\ \boldsymbol{\alpha}_{6} \notin \boldsymbol{G}_{6}, \\
\nonumber & \left. \quad \nu_0 \leqslant \alpha_1 < \frac{1}{2},\ \nu\left(\alpha_1\right) \leqslant \alpha_2 < \min\left(\alpha_1, \frac{1}{2}(1-\alpha_1) \right) \right\}, \\
\nonumber U_{A5}(\boldsymbol{\alpha}_{4}) :=&\ \left\{ \boldsymbol{\alpha}_{2} \in A^{\prime},\ \nu_0 \leqslant \alpha_{3} < \min\left(\alpha_2, \frac{1}{2}(1-\alpha_1 -\alpha_2)\right),\ \boldsymbol{\alpha}_{3} \notin \boldsymbol{G}_{3}, \right. \\
\nonumber & \quad \boldsymbol{\alpha}_{3} \text{ cannot be partitioned into } (m, n) \in \boldsymbol{D}_{0} \text{ or } (m, n, h) \in \boldsymbol{D}_{1} \cup \boldsymbol{D}_{2}, \\
\nonumber & \quad \nu_0 \leqslant \alpha_{4} < \frac{1}{2}\alpha_1,\ \boldsymbol{\alpha}_{4}^{\prime} \notin \boldsymbol{G}_{4}, \\
\nonumber & \left. \quad \nu_0 \leqslant \alpha_1 < \frac{1}{2},\ \nu\left(\alpha_1\right) \leqslant \alpha_2 < \min\left(\alpha_1, \frac{1}{2}(1-\alpha_1) \right) \right\}.
\end{align}
Note that the above integrals arise from sums $T_{011}$, $T_{012}$, $T_{02}$ and the two-dimensional sum over region $A \backslash A^{\prime}$, and one can compare our integrals to those in \cite{BHP} and \cite{HarmanBOOK}. For example, one can see the integrals corresponding to $U_{A2}$ and $U_{A3}$ as an simple explicit expression of the function $w(\boldsymbol{\alpha}_{4})$ defined in [\cite{HarmanBOOK}, Chapter 7.9]. In \cite{BHP} and \cite{HarmanBOOK} reversed Buchstab's identity has been used many many times, but we do not consider using it repeatedly since the savings over high-dimensional sums produced by this technique are very small. We have tried some acceptable region $A^{\prime}$, but the total loss in (18) exceeds the original two-dimensional loss from $A^{\prime}$. Some small savings may be obtained by using more careful decompositions and more powerful supercomputers to calculate the loss integrals.

Remember that we have $\alpha_3 < \alpha_2 \leqslant \frac{3\theta -1}{2} \leqslant \frac{3 \theta+1-4 \alpha_1^{*}}{2}$ and at least one of $\alpha_1 +\alpha_2$ and $1-\alpha_1 -\alpha_2$ is $\leqslant \frac{1}{2}$ when $(\alpha_1, \alpha_2) \in C$, further decompositions in region $C$ are possible. For $\sum_{C}$ we can redo the above decomposition procedure (15) on the entire region $C$ to reach two four-dimensional sums
\begin{equation}
T_{03} := \sum_{\substack{\boldsymbol{\alpha}_{2} \in C,\ \boldsymbol{\alpha}_{2} \notin \boldsymbol{G}_{2} \\ \nu_0 \leqslant \alpha_{3} < \min\left(\alpha_2, \frac{1}{2}(1-\alpha_1 -\alpha_2)\right) \\ \boldsymbol{\alpha}_{3} \notin \boldsymbol{G}_{3} \\ \boldsymbol{\alpha}_{3} \text{ can be partitioned into } (m, n) \in \boldsymbol{D}_{0} \text{ or } (m, n, h) \in \boldsymbol{D}_{1} \cup \boldsymbol{D}_{2} \\ \nu_0 \leqslant \alpha_{4} < \min\left(\alpha_3, \frac{1}{2}(1-\alpha_1 -\alpha_2 -\alpha_3)\right) \\ \boldsymbol{\alpha}_{4} \notin \boldsymbol{G}_{4} } }S\left(\mathcal{A}_{p_{1} p_{2} p_{3} p_{4}}, p_{4}\right)
\end{equation}
and
\begin{equation}
T_{04} := \sum_{\substack{\boldsymbol{\alpha}_{2} \in C,\ \boldsymbol{\alpha}_{2} \notin \boldsymbol{G}_{2} \\ \nu_0 \leqslant \alpha_{3} < \min\left(\alpha_2, \frac{1}{2}(1-\alpha_1 -\alpha_2)\right) \\ \boldsymbol{\alpha}_{3} \notin \boldsymbol{G}_{3} \\ \boldsymbol{\alpha}_{3} \text{ cannot be partitioned into } (m, n) \in \boldsymbol{D}_{0} \text{ or } (m, n, h) \in \boldsymbol{D}_{1} \cup \boldsymbol{D}_{2} \\ \nu_0 \leqslant \alpha_{4} < \frac{1}{2}\alpha_1 \\ \boldsymbol{\alpha}_{4}^{\prime} \notin \boldsymbol{G}_{4} } }S\left(\mathcal{A}_{\beta p_{2} p_{3} p_{4}}, p_{4}\right).
\end{equation}
For $T_{03}$, we can perform Buchstab's identity twice more if $\boldsymbol{\alpha}_{4} \in \boldsymbol{D}^{*}$, and we can use Buchstab's identity twice with a role-reversal if $\boldsymbol{\alpha}_{4} \in \boldsymbol{D}^{\ddag}$. Again, we can apply Buchstab's identity in reverse to gain some savings by making almost-primes visible. Similar to the decomposition procedure (15), we have the following expression after the decomposition of $T_{03}$:
\begin{align}
\nonumber T_{03} =&\ \sum_{\substack{\boldsymbol{\alpha}_{2} \in C,\ \boldsymbol{\alpha}_{2} \notin \boldsymbol{G}_{2} \\ \nu_0 \leqslant \alpha_{3} < \min\left(\alpha_2, \frac{1}{2}(1-\alpha_1 -\alpha_2)\right) \\ \boldsymbol{\alpha}_{3} \notin \boldsymbol{G}_{3} \\ \boldsymbol{\alpha}_{3} \text{ can be partitioned into } (m, n) \in \boldsymbol{D}_{0} \text{ or } (m, n, h) \in \boldsymbol{D}_{1} \cup \boldsymbol{D}_{2} \\ \nu_0 \leqslant \alpha_{4} < \min\left(\alpha_3, \frac{1}{2}(1-\alpha_1 -\alpha_2 -\alpha_3)\right) \\ \boldsymbol{\alpha}_{4} \notin \boldsymbol{G}_{4} } }S\left(\mathcal{A}_{p_{1} p_{2} p_{3} p_{4}}, p_{4}\right) \\
\nonumber =&\ \sum_{\substack{\boldsymbol{\alpha}_{2} \in C,\ \boldsymbol{\alpha}_{2} \notin \boldsymbol{G}_{2} \\ \nu_0 \leqslant \alpha_{3} < \min\left(\alpha_2, \frac{1}{2}(1-\alpha_1 -\alpha_2)\right) \\ \boldsymbol{\alpha}_{3} \notin \boldsymbol{G}_{3} \\ \boldsymbol{\alpha}_{3} \text{ can be partitioned into } (m, n) \in \boldsymbol{D}_{0} \text{ or } (m, n, h) \in \boldsymbol{D}_{1} \cup \boldsymbol{D}_{2} \\ \nu_0 \leqslant \alpha_{4} < \min\left(\alpha_3, \frac{1}{2}(1-\alpha_1 -\alpha_2 -\alpha_3)\right) \\ \boldsymbol{\alpha}_{4} \notin \boldsymbol{G}_{4},\ \boldsymbol{\alpha}_{4} \notin \boldsymbol{D}^{*},\ \boldsymbol{\alpha}_{4} \notin \boldsymbol{D}^{\ddag} } }S\left(\mathcal{A}_{p_{1} p_{2} p_{3} p_{4}}, p_{4}\right) \\
\nonumber &+ \sum_{\substack{\boldsymbol{\alpha}_{2} \in C,\ \boldsymbol{\alpha}_{2} \notin \boldsymbol{G}_{2} \\ \nu_0 \leqslant \alpha_{3} < \min\left(\alpha_2, \frac{1}{2}(1-\alpha_1 -\alpha_2)\right) \\ \boldsymbol{\alpha}_{3} \notin \boldsymbol{G}_{3} \\ \boldsymbol{\alpha}_{3} \text{ can be partitioned into } (m, n) \in \boldsymbol{D}_{0} \text{ or } (m, n, h) \in \boldsymbol{D}_{1} \cup \boldsymbol{D}_{2} \\ \nu_0 \leqslant \alpha_{4} < \min\left(\alpha_3, \frac{1}{2}(1-\alpha_1 -\alpha_2 -\alpha_3)\right) \\ \boldsymbol{\alpha}_{4} \notin \boldsymbol{G}_{4},\ \boldsymbol{\alpha}_{4} \in \boldsymbol{D}^{*} } }S\left(\mathcal{A}_{p_{1} p_{2} p_{3} p_{4}}, p_{4}\right) \\
\nonumber &+ \sum_{\substack{\boldsymbol{\alpha}_{2} \in C,\ \boldsymbol{\alpha}_{2} \notin \boldsymbol{G}_{2} \\ \nu_0 \leqslant \alpha_{3} < \min\left(\alpha_2, \frac{1}{2}(1-\alpha_1 -\alpha_2)\right) \\ \boldsymbol{\alpha}_{3} \notin \boldsymbol{G}_{3} \\ \boldsymbol{\alpha}_{3} \text{ can be partitioned into } (m, n) \in \boldsymbol{D}_{0} \text{ or } (m, n, h) \in \boldsymbol{D}_{1} \cup \boldsymbol{D}_{2} \\ \nu_0 \leqslant \alpha_{4} < \min\left(\alpha_3, \frac{1}{2}(1-\alpha_1 -\alpha_2 -\alpha_3)\right) \\ \boldsymbol{\alpha}_{4} \notin \boldsymbol{G}_{4},\ \boldsymbol{\alpha}_{4} \notin \boldsymbol{D}^{*},\ \boldsymbol{\alpha}_{4} \in \boldsymbol{D}^{\ddag} } }S\left(\mathcal{A}_{p_{1} p_{2} p_{3} p_{4}}, p_{4}\right) \\
\nonumber =&\ \sum_{\substack{\boldsymbol{\alpha}_{2} \in C,\ \boldsymbol{\alpha}_{2} \notin \boldsymbol{G}_{2} \\ \nu_0 \leqslant \alpha_{3} < \min\left(\alpha_2, \frac{1}{2}(1-\alpha_1 -\alpha_2)\right) \\ \boldsymbol{\alpha}_{3} \notin \boldsymbol{G}_{3} \\ \boldsymbol{\alpha}_{3} \text{ can be partitioned into } (m, n) \in \boldsymbol{D}_{0} \text{ or } (m, n, h) \in \boldsymbol{D}_{1} \cup \boldsymbol{D}_{2} \\ \nu_0 \leqslant \alpha_{4} < \min\left(\alpha_3, \frac{1}{2}(1-\alpha_1 -\alpha_2 -\alpha_3)\right) \\ \boldsymbol{\alpha}_{4} \notin \boldsymbol{G}_{4},\ \boldsymbol{\alpha}_{4} \notin \boldsymbol{D}^{*},\ \boldsymbol{\alpha}_{4} \notin \boldsymbol{D}^{\ddag} } }S\left(\mathcal{A}_{p_{1} p_{2} p_{3} p_{4}}, p_{4}\right) \\
\nonumber &+ \sum_{\substack{\boldsymbol{\alpha}_{2} \in C,\ \boldsymbol{\alpha}_{2} \notin \boldsymbol{G}_{2} \\ \nu_0 \leqslant \alpha_{3} < \min\left(\alpha_2, \frac{1}{2}(1-\alpha_1 -\alpha_2)\right) \\ \boldsymbol{\alpha}_{3} \notin \boldsymbol{G}_{3} \\ \boldsymbol{\alpha}_{3} \text{ can be partitioned into } (m, n) \in \boldsymbol{D}_{0} \text{ or } (m, n, h) \in \boldsymbol{D}_{1} \cup \boldsymbol{D}_{2} \\ \nu_0 \leqslant \alpha_{4} < \min\left(\alpha_3, \frac{1}{2}(1-\alpha_1 -\alpha_2 -\alpha_3)\right) \\ \boldsymbol{\alpha}_{4} \notin \boldsymbol{G}_{4},\ \boldsymbol{\alpha}_{4} \in \boldsymbol{D}^{*} } }S\left(\mathcal{A}_{p_{1} p_{2} p_{3} p_{4}}, x^{\nu_0}\right) \\
\nonumber &- \sum_{\substack{\boldsymbol{\alpha}_{2} \in C,\ \boldsymbol{\alpha}_{2} \notin \boldsymbol{G}_{2} \\ \nu_0 \leqslant \alpha_{3} < \min\left(\alpha_2, \frac{1}{2}(1-\alpha_1 -\alpha_2)\right) \\ \boldsymbol{\alpha}_{3} \notin \boldsymbol{G}_{3} \\ \boldsymbol{\alpha}_{3} \text{ can be partitioned into } (m, n) \in \boldsymbol{D}_{0} \text{ or } (m, n, h) \in \boldsymbol{D}_{1} \cup \boldsymbol{D}_{2} \\ \nu_0 \leqslant \alpha_{4} < \min\left(\alpha_3, \frac{1}{2}(1-\alpha_1 -\alpha_2 -\alpha_3)\right) \\ \boldsymbol{\alpha}_{4} \notin \boldsymbol{G}_{4},\ \boldsymbol{\alpha}_{4} \in \boldsymbol{D}^{*} \\ \nu_0 \leqslant \alpha_{5} < \min\left(\alpha_4, \frac{1}{2}(1-\alpha_1 -\alpha_2 -\alpha_3 -\alpha_4)\right) \\ \boldsymbol{\alpha}_{5} \in \boldsymbol{G}_{5} } }S\left(\mathcal{A}_{p_{1} p_{2} p_{3} p_{4} p_{5}}, p_{5}\right) \\
\nonumber &- \sum_{\substack{\boldsymbol{\alpha}_{2} \in C,\ \boldsymbol{\alpha}_{2} \notin \boldsymbol{G}_{2} \\ \nu_0 \leqslant \alpha_{3} < \min\left(\alpha_2, \frac{1}{2}(1-\alpha_1 -\alpha_2)\right) \\ \boldsymbol{\alpha}_{3} \notin \boldsymbol{G}_{3} \\ \boldsymbol{\alpha}_{3} \text{ can be partitioned into } (m, n) \in \boldsymbol{D}_{0} \text{ or } (m, n, h) \in \boldsymbol{D}_{1} \cup \boldsymbol{D}_{2} \\ \nu_0 \leqslant \alpha_{4} < \min\left(\alpha_3, \frac{1}{2}(1-\alpha_1 -\alpha_2 -\alpha_3)\right) \\ \boldsymbol{\alpha}_{4} \notin \boldsymbol{G}_{4},\ \boldsymbol{\alpha}_{4} \in \boldsymbol{D}^{*} \\ \nu_0 \leqslant \alpha_{5} < \min\left(\alpha_4, \frac{1}{2}(1-\alpha_1 -\alpha_2 -\alpha_3 -\alpha_4)\right) \\ \boldsymbol{\alpha}_{5} \notin \boldsymbol{G}_{5} } }S\left(\mathcal{A}_{p_{1} p_{2} p_{3} p_{4} p_{5}}, x^{\nu_0}\right) \\
\nonumber &+ \sum_{\substack{\boldsymbol{\alpha}_{2} \in C,\ \boldsymbol{\alpha}_{2} \notin \boldsymbol{G}_{2} \\ \nu_0 \leqslant \alpha_{3} < \min\left(\alpha_2, \frac{1}{2}(1-\alpha_1 -\alpha_2)\right) \\ \boldsymbol{\alpha}_{3} \notin \boldsymbol{G}_{3} \\ \boldsymbol{\alpha}_{3} \text{ can be partitioned into } (m, n) \in \boldsymbol{D}_{0} \text{ or } (m, n, h) \in \boldsymbol{D}_{1} \cup \boldsymbol{D}_{2} \\ \nu_0 \leqslant \alpha_{4} < \min\left(\alpha_3, \frac{1}{2}(1-\alpha_1 -\alpha_2 -\alpha_3)\right) \\ \boldsymbol{\alpha}_{4} \notin \boldsymbol{G}_{4},\ \boldsymbol{\alpha}_{4} \in \boldsymbol{D}^{*} \\ \nu_0 \leqslant \alpha_{5} < \min\left(\alpha_4, \frac{1}{2}(1-\alpha_1 -\alpha_2 -\alpha_3 -\alpha_4)\right) \\ \boldsymbol{\alpha}_{5} \notin \boldsymbol{G}_{5} \\ \nu_0 \leqslant \alpha_{6} < \min\left(\alpha_5, \frac{1}{2}(1-\alpha_1 -\alpha_2 -\alpha_3 -\alpha_4 -\alpha_5)\right) \\ \boldsymbol{\alpha}_{6} \in \boldsymbol{G}_{6} } }S\left(\mathcal{A}_{p_{1} p_{2} p_{3} p_{4} p_{5} p_{6}}, p_{6}\right) \\
\nonumber &+ \sum_{\substack{\boldsymbol{\alpha}_{2} \in C,\ \boldsymbol{\alpha}_{2} \notin \boldsymbol{G}_{2} \\ \nu_0 \leqslant \alpha_{3} < \min\left(\alpha_2, \frac{1}{2}(1-\alpha_1 -\alpha_2)\right) \\ \boldsymbol{\alpha}_{3} \notin \boldsymbol{G}_{3} \\ \boldsymbol{\alpha}_{3} \text{ can be partitioned into } (m, n) \in \boldsymbol{D}_{0} \text{ or } (m, n, h) \in \boldsymbol{D}_{1} \cup \boldsymbol{D}_{2} \\ \nu_0 \leqslant \alpha_{4} < \min\left(\alpha_3, \frac{1}{2}(1-\alpha_1 -\alpha_2 -\alpha_3)\right) \\ \boldsymbol{\alpha}_{4} \notin \boldsymbol{G}_{4},\ \boldsymbol{\alpha}_{4} \in \boldsymbol{D}^{*} \\ \nu_0 \leqslant \alpha_{5} < \min\left(\alpha_4, \frac{1}{2}(1-\alpha_1 -\alpha_2 -\alpha_3 -\alpha_4)\right) \\ \boldsymbol{\alpha}_{5} \notin \boldsymbol{G}_{5} \\ \nu_0 \leqslant \alpha_{6} < \min\left(\alpha_5, \frac{1}{2}(1-\alpha_1 -\alpha_2 -\alpha_3 -\alpha_4 -\alpha_5)\right) \\ \boldsymbol{\alpha}_{6} \notin \boldsymbol{G}_{6} } }S\left(\mathcal{A}_{p_{1} p_{2} p_{3} p_{4} p_{5} p_{6}}, p_{6}\right) \\
\nonumber &+ \sum_{\substack{\boldsymbol{\alpha}_{2} \in C,\ \boldsymbol{\alpha}_{2} \notin \boldsymbol{G}_{2} \\ \nu_0 \leqslant \alpha_{3} < \min\left(\alpha_2, \frac{1}{2}(1-\alpha_1 -\alpha_2)\right) \\ \boldsymbol{\alpha}_{3} \notin \boldsymbol{G}_{3} \\ \boldsymbol{\alpha}_{3} \text{ can be partitioned into } (m, n) \in \boldsymbol{D}_{0} \text{ or } (m, n, h) \in \boldsymbol{D}_{1} \cup \boldsymbol{D}_{2} \\ \nu_0 \leqslant \alpha_{4} < \min\left(\alpha_3, \frac{1}{2}(1-\alpha_1 -\alpha_2 -\alpha_3)\right) \\ \boldsymbol{\alpha}_{4} \notin \boldsymbol{G}_{4},\ \boldsymbol{\alpha}_{4} \notin \boldsymbol{D}^{*},\ \boldsymbol{\alpha}_{4} \in \boldsymbol{D}^{\ddag} } }S\left(\mathcal{A}_{p_{1} p_{2} p_{3} p_{4}}, x^{\nu_0}\right) \\
\nonumber &- \sum_{\substack{\boldsymbol{\alpha}_{2} \in C,\ \boldsymbol{\alpha}_{2} \notin \boldsymbol{G}_{2} \\ \nu_0 \leqslant \alpha_{3} < \min\left(\alpha_2, \frac{1}{2}(1-\alpha_1 -\alpha_2)\right) \\ \boldsymbol{\alpha}_{3} \notin \boldsymbol{G}_{3} \\ \boldsymbol{\alpha}_{3} \text{ can be partitioned into } (m, n) \in \boldsymbol{D}_{0} \text{ or } (m, n, h) \in \boldsymbol{D}_{1} \cup \boldsymbol{D}_{2} \\ \nu_0 \leqslant \alpha_{4} < \min\left(\alpha_3, \frac{1}{2}(1-\alpha_1 -\alpha_2 -\alpha_3)\right) \\ \boldsymbol{\alpha}_{4} \notin \boldsymbol{G}_{4},\ \boldsymbol{\alpha}_{4} \notin \boldsymbol{D}^{*},\ \boldsymbol{\alpha}_{4} \in \boldsymbol{D}^{\ddag} \\ \nu_0 \leqslant \alpha_{5} < \min\left(\alpha_4, \frac{1}{2}(1-\alpha_1 -\alpha_2 -\alpha_3 -\alpha_4)\right) \\ \boldsymbol{\alpha}_{5} \in \boldsymbol{G}_{5} } }S\left(\mathcal{A}_{p_{1} p_{2} p_{3} p_{4} p_{5}}, p_{5}\right) \\
\nonumber &- \sum_{\substack{\boldsymbol{\alpha}_{2} \in C,\ \boldsymbol{\alpha}_{2} \notin \boldsymbol{G}_{2} \\ \nu_0 \leqslant \alpha_{3} < \min\left(\alpha_2, \frac{1}{2}(1-\alpha_1 -\alpha_2)\right) \\ \boldsymbol{\alpha}_{3} \notin \boldsymbol{G}_{3} \\ \boldsymbol{\alpha}_{3} \text{ can be partitioned into } (m, n) \in \boldsymbol{D}_{0} \text{ or } (m, n, h) \in \boldsymbol{D}_{1} \cup \boldsymbol{D}_{2} \\ \nu_0 \leqslant \alpha_{4} < \min\left(\alpha_3, \frac{1}{2}(1-\alpha_1 -\alpha_2 -\alpha_3)\right) \\ \boldsymbol{\alpha}_{4} \notin \boldsymbol{G}_{4},\ \boldsymbol{\alpha}_{4} \notin \boldsymbol{D}^{*},\ \boldsymbol{\alpha}_{4} \in \boldsymbol{D}^{\ddag} \\ \nu_0 \leqslant \alpha_{5} < \min\left(\alpha_4, \frac{1}{2}(1-\alpha_1 -\alpha_2 -\alpha_3 -\alpha_4)\right) \\ \boldsymbol{\alpha}_{5} \notin \boldsymbol{G}_{5} } }S\left(\mathcal{A}_{\gamma p_{2} p_{3} p_{4} p_{5}}, x^{\nu_0}\right) \\
\nonumber &+ \sum_{\substack{\boldsymbol{\alpha}_{2} \in C,\ \boldsymbol{\alpha}_{2} \notin \boldsymbol{G}_{2} \\ \nu_0 \leqslant \alpha_{3} < \min\left(\alpha_2, \frac{1}{2}(1-\alpha_1 -\alpha_2)\right) \\ \boldsymbol{\alpha}_{3} \notin \boldsymbol{G}_{3} \\ \boldsymbol{\alpha}_{3} \text{ can be partitioned into } (m, n) \in \boldsymbol{D}_{0} \text{ or } (m, n, h) \in \boldsymbol{D}_{1} \cup \boldsymbol{D}_{2} \\ \nu_0 \leqslant \alpha_{4} < \min\left(\alpha_3, \frac{1}{2}(1-\alpha_1 -\alpha_2 -\alpha_3)\right) \\ \boldsymbol{\alpha}_{4} \notin \boldsymbol{G}_{4},\ \boldsymbol{\alpha}_{4} \notin \boldsymbol{D}^{*},\ \boldsymbol{\alpha}_{4} \in \boldsymbol{D}^{\ddag} \\ \nu_0 \leqslant \alpha_{5} < \min\left(\alpha_4, \frac{1}{2}(1-\alpha_1 -\alpha_2 -\alpha_3 -\alpha_4)\right) \\ \boldsymbol{\alpha}_{5} \notin \boldsymbol{G}_{5} \\ \nu_0 \leqslant \alpha_{6} < \frac{1}{2}\alpha_1 \\ \boldsymbol{\alpha}_{6}^{\ddag} \in \boldsymbol{G}_{6} } }S\left(\mathcal{A}_{\gamma p_{2} p_{3} p_{4} p_{5} p_{6}}, p_{6}\right) \\
\nonumber &+ \sum_{\substack{\boldsymbol{\alpha}_{2} \in C,\ \boldsymbol{\alpha}_{2} \notin \boldsymbol{G}_{2} \\ \nu_0 \leqslant \alpha_{3} < \min\left(\alpha_2, \frac{1}{2}(1-\alpha_1 -\alpha_2)\right) \\ \boldsymbol{\alpha}_{3} \notin \boldsymbol{G}_{3} \\ \boldsymbol{\alpha}_{3} \text{ can be partitioned into } (m, n) \in \boldsymbol{D}_{0} \text{ or } (m, n, h) \in \boldsymbol{D}_{1} \cup \boldsymbol{D}_{2} \\ \nu_0 \leqslant \alpha_{4} < \min\left(\alpha_3, \frac{1}{2}(1-\alpha_1 -\alpha_2 -\alpha_3)\right) \\ \boldsymbol{\alpha}_{4} \notin \boldsymbol{G}_{4},\ \boldsymbol{\alpha}_{4} \notin \boldsymbol{D}^{*},\ \boldsymbol{\alpha}_{4} \in \boldsymbol{D}^{\ddag} \\ \nu_0 \leqslant \alpha_{5} < \min\left(\alpha_4, \frac{1}{2}(1-\alpha_1 -\alpha_2 -\alpha_3 -\alpha_4)\right) \\ \boldsymbol{\alpha}_{5} \notin \boldsymbol{G}_{5} \\ \nu_0 \leqslant \alpha_{6} < \frac{1}{2}\alpha_1 \\ \boldsymbol{\alpha}_{6}^{\ddag} \notin \boldsymbol{G}_{6} } }S\left(\mathcal{A}_{\gamma p_{2} p_{3} p_{4} p_{5} p_{6}}, p_{6}\right) \\
=&\ T_{031} + S_{11} - S_{12} - S_{13} + S_{14} + T_{032} + S_{15} - S_{16} - S_{17} + S_{18} + T_{033},
\end{align}
where $\gamma \sim x^{1-\alpha_1 -\alpha_2 -\alpha_3 -\alpha_4 -\alpha_5}$, $\left(\gamma, P(p_{5})\right)=1$ and
$$
\boldsymbol{\alpha}_{6}^{\ddag} =(1-\alpha_1-\alpha_2-\alpha_3-\alpha_4-\alpha_5,\ \alpha_2,\ \alpha_3,\ \alpha_4,\ \alpha_5,\ \alpha_6).
$$
We can give asymptotic formulas for $S_{11}$--$S_{14}$, $S_{16}$ and $S_{18}$. For $S_{15}$ we note that $\boldsymbol{\alpha}_{4} \in D^{\dag}$ yields an asymptotic formula. For $S_{17}$, the variables $(\gamma,\ p_2,\ p_3,\ p_4,\ p_5)$ correspond to the variables $(1-\alpha_1-\alpha_2-\alpha_3-\alpha_4-\alpha_5,\ \alpha_2,\ \alpha_3,\ \alpha_4,\ \alpha_5)$. Combining the first and the last variables, we obtain a new set of variables $(1-\alpha_1-\alpha_2-\alpha_3-\alpha_4,\ \alpha_2,\ \alpha_3,\ \alpha_4)$. Now by the condition $(1-\alpha_1-\alpha_2-\alpha_3-\alpha_4,\ \alpha_2,\ \alpha_3,\ \alpha_4) \in D^{\dag}$, we know that $S_{17}$ also has an asymptotic formula. For $T_{031}$ we can use Buchstab's identity in reverse to subtract the sum
\begin{equation}
\sum_{\substack{\boldsymbol{\alpha}_{2} \in C,\ \boldsymbol{\alpha}_{2} \notin \boldsymbol{G}_{2} \\ \nu_0 \leqslant \alpha_{3} < \min\left(\alpha_2, \frac{1}{2}(1-\alpha_1 -\alpha_2)\right) \\ \boldsymbol{\alpha}_{3} \notin \boldsymbol{G}_{3} \\ \boldsymbol{\alpha}_{3} \text{ can be partitioned into } (m, n) \in \boldsymbol{D}_{0} \text{ or } (m, n, h) \in \boldsymbol{D}_{1} \cup \boldsymbol{D}_{2} \\ \nu_0 \leqslant \alpha_{4} < \min\left(\alpha_3, \frac{1}{2}(1-\alpha_1 -\alpha_2 -\alpha_3)\right) \\ \boldsymbol{\alpha}_{4} \notin \boldsymbol{G}_{4},\ \boldsymbol{\alpha}_{4} \notin \boldsymbol{D}^{*},\ \boldsymbol{\alpha}_{4} \notin \boldsymbol{D}^{\ddag} \\ \alpha_4 < \alpha_5 < \frac{1}{2}(1-\alpha_1 -\alpha_2 -\alpha_3 -\alpha_4) \\ \boldsymbol{\alpha}_{5} \in \boldsymbol{G}_{5} } }S\left(\mathcal{A}_{p_{1} p_{2} p_{3} p_{4} p_{5}}, p_{5}\right)
\end{equation}
from the loss, and for $T_{032}$ we can perform a straightforward decomposition if $\boldsymbol{\alpha}_{6} \in \boldsymbol{D}^{**}$, leading to an eight-dimensional sum
\begin{align}
\sum_{\substack{\boldsymbol{\alpha}_{2} \in C,\ \boldsymbol{\alpha}_{2} \notin \boldsymbol{G}_{2} \\ \nu_0 \leqslant \alpha_{3} < \min\left(\alpha_2, \frac{1}{2}(1-\alpha_1 -\alpha_2)\right) \\ \boldsymbol{\alpha}_{3} \notin \boldsymbol{G}_{3} \\ \boldsymbol{\alpha}_{3} \text{ can be partitioned into } (m, n) \in \boldsymbol{D}_{0} \text{ or } (m, n, h) \in \boldsymbol{D}_{1} \cup \boldsymbol{D}_{2} \\ \nu_0 \leqslant \alpha_{4} < \min\left(\alpha_3, \frac{1}{2}(1-\alpha_1 -\alpha_2 -\alpha_3)\right) \\ \boldsymbol{\alpha}_{4} \notin \boldsymbol{G}_{4},\ \boldsymbol{\alpha}_{4} \in \boldsymbol{D}^{*} \\ \nu_0 \leqslant \alpha_{5} < \min\left(\alpha_4, \frac{1}{2}(1-\alpha_1 -\alpha_2 -\alpha_3 -\alpha_4)\right) \\ \boldsymbol{\alpha}_{5} \notin \boldsymbol{G}_{5} \\ \nu_0 \leqslant \alpha_{6} < \min\left(\alpha_5, \frac{1}{2}(1-\alpha_1 -\alpha_2 -\alpha_3 -\alpha_4 -\alpha_5)\right) \\ \boldsymbol{\alpha}_{6} \notin \boldsymbol{G}_{6},\ \boldsymbol{\alpha}_{6} \in \boldsymbol{D}^{**} \\ \nu_0 \leqslant \alpha_{7} < \min\left(\alpha_6, \frac{1}{2}(1-\alpha_1 -\alpha_2 -\alpha_3 -\alpha_4 -\alpha_5 -\alpha_6)\right) \\ \boldsymbol{\alpha}_{7} \notin \boldsymbol{G}_{7} \\ \nu_0 \leqslant \alpha_{8} < \min\left(\alpha_7, \frac{1}{2}(1-\alpha_1 -\alpha_2 -\alpha_3 -\alpha_4 -\alpha_5 -\alpha_6 -\alpha_7)\right) \\ \boldsymbol{\alpha}_{8} \notin \boldsymbol{G}_{8} } }S\left(\mathcal{A}_{p_{1} p_{2} p_{3} p_{4} p_{5} p_{6} p_{7} p_{8}}. p_{8}\right)
\end{align}
Note that in $T_{033}$ we counts numbers of the form $\gamma p_2 p_3 p_4 p_5 p_6 \gamma_1$ with two almost-prime variables $\gamma$ and $\gamma_1 \sim x^{\alpha_1 -\alpha_6}$, hence we can perform a straightforward decomposition if either $\boldsymbol{\alpha}_{6}^{\ddag} \in \boldsymbol{D}^{**}$ or $\boldsymbol{\alpha}_{6}^{\ddag \prime} \in \boldsymbol{D}^{**}$, where
$$
\boldsymbol{\alpha}_{6}^{\ddag \prime} = (\alpha_1-\alpha_6,\ \alpha_2,\ \alpha_3,\ \alpha_4,\ \alpha_6,\ \alpha_5).
$$
That is, we write
\begin{align}
\nonumber T_{033} =&\ \sum_{\substack{\boldsymbol{\alpha}_{2} \in C,\ \boldsymbol{\alpha}_{2} \notin \boldsymbol{G}_{2} \\ \nu_0 \leqslant \alpha_{3} < \min\left(\alpha_2, \frac{1}{2}(1-\alpha_1 -\alpha_2)\right) \\ \boldsymbol{\alpha}_{3} \notin \boldsymbol{G}_{3} \\ \boldsymbol{\alpha}_{3} \text{ can be partitioned into } (m, n) \in \boldsymbol{D}_{0} \text{ or } (m, n, h) \in \boldsymbol{D}_{1} \cup \boldsymbol{D}_{2} \\ \nu_0 \leqslant \alpha_{4} < \min\left(\alpha_3, \frac{1}{2}(1-\alpha_1 -\alpha_2 -\alpha_3)\right) \\ \boldsymbol{\alpha}_{4} \notin \boldsymbol{G}_{4},\ \boldsymbol{\alpha}_{4} \notin \boldsymbol{D}^{*},\ \boldsymbol{\alpha}_{4} \in \boldsymbol{D}^{\ddag} \\ \nu_0 \leqslant \alpha_{5} < \min\left(\alpha_4, \frac{1}{2}(1-\alpha_1 -\alpha_2 -\alpha_3 -\alpha_4)\right) \\ \boldsymbol{\alpha}_{5} \notin \boldsymbol{G}_{5} \\ \nu_0 \leqslant \alpha_{6} < \frac{1}{2}\alpha_1 \\ \boldsymbol{\alpha}_{6}^{\ddag} \notin \boldsymbol{G}_{6} } }S\left(\mathcal{A}_{\gamma p_{2} p_{3} p_{4} p_{5} p_{6}}, p_{6}\right) \\
\nonumber =&\ \sum_{\substack{\boldsymbol{\alpha}_{2} \in C,\ \boldsymbol{\alpha}_{2} \notin \boldsymbol{G}_{2} \\ \nu_0 \leqslant \alpha_{3} < \min\left(\alpha_2, \frac{1}{2}(1-\alpha_1 -\alpha_2)\right) \\ \boldsymbol{\alpha}_{3} \notin \boldsymbol{G}_{3} \\ \boldsymbol{\alpha}_{3} \text{ can be partitioned into } (m, n) \in \boldsymbol{D}_{0} \text{ or } (m, n, h) \in \boldsymbol{D}_{1} \cup \boldsymbol{D}_{2} \\ \nu_0 \leqslant \alpha_{4} < \min\left(\alpha_3, \frac{1}{2}(1-\alpha_1 -\alpha_2 -\alpha_3)\right) \\ \boldsymbol{\alpha}_{4} \notin \boldsymbol{G}_{4},\ \boldsymbol{\alpha}_{4} \notin \boldsymbol{D}^{*},\ \boldsymbol{\alpha}_{4} \in \boldsymbol{D}^{\ddag} \\ \nu_0 \leqslant \alpha_{5} < \min\left(\alpha_4, \frac{1}{2}(1-\alpha_1 -\alpha_2 -\alpha_3 -\alpha_4)\right) \\ \boldsymbol{\alpha}_{5} \notin \boldsymbol{G}_{5} \\ \nu_0 \leqslant \alpha_{6} < \frac{1}{2}\alpha_1 \\ \boldsymbol{\alpha}_{6}^{\ddag} \notin \boldsymbol{G}_{6},\ \boldsymbol{\alpha}_{6}^{\ddag} \notin \boldsymbol{D}^{**},\ \boldsymbol{\alpha}_{6}^{\ddag \prime} \notin \boldsymbol{D}^{**} } }S\left(\mathcal{A}_{\gamma p_{2} p_{3} p_{4} p_{5} p_{6}}, p_{6}\right) \\
\nonumber &+ \sum_{\substack{\boldsymbol{\alpha}_{2} \in C,\ \boldsymbol{\alpha}_{2} \notin \boldsymbol{G}_{2} \\ \nu_0 \leqslant \alpha_{3} < \min\left(\alpha_2, \frac{1}{2}(1-\alpha_1 -\alpha_2)\right) \\ \boldsymbol{\alpha}_{3} \notin \boldsymbol{G}_{3} \\ \boldsymbol{\alpha}_{3} \text{ can be partitioned into } (m, n) \in \boldsymbol{D}_{0} \text{ or } (m, n, h) \in \boldsymbol{D}_{1} \cup \boldsymbol{D}_{2} \\ \nu_0 \leqslant \alpha_{4} < \min\left(\alpha_3, \frac{1}{2}(1-\alpha_1 -\alpha_2 -\alpha_3)\right) \\ \boldsymbol{\alpha}_{4} \notin \boldsymbol{G}_{4},\ \boldsymbol{\alpha}_{4} \notin \boldsymbol{D}^{*},\ \boldsymbol{\alpha}_{4} \in \boldsymbol{D}^{\ddag} \\ \nu_0 \leqslant \alpha_{5} < \min\left(\alpha_4, \frac{1}{2}(1-\alpha_1 -\alpha_2 -\alpha_3 -\alpha_4)\right) \\ \boldsymbol{\alpha}_{5} \notin \boldsymbol{G}_{5} \\ \nu_0 \leqslant \alpha_{6} < \frac{1}{2}\alpha_1 \\ \boldsymbol{\alpha}_{6}^{\ddag} \notin \boldsymbol{G}_{6},\ \boldsymbol{\alpha}_{6}^{\ddag} \in \boldsymbol{D}^{**} } }S\left(\mathcal{A}_{\gamma p_{2} p_{3} p_{4} p_{5} p_{6}}, p_{6}\right) \\
\nonumber &+ \sum_{\substack{\boldsymbol{\alpha}_{2} \in C,\ \boldsymbol{\alpha}_{2} \notin \boldsymbol{G}_{2} \\ \nu_0 \leqslant \alpha_{3} < \min\left(\alpha_2, \frac{1}{2}(1-\alpha_1 -\alpha_2)\right) \\ \boldsymbol{\alpha}_{3} \notin \boldsymbol{G}_{3} \\ \boldsymbol{\alpha}_{3} \text{ can be partitioned into } (m, n) \in \boldsymbol{D}_{0} \text{ or } (m, n, h) \in \boldsymbol{D}_{1} \cup \boldsymbol{D}_{2} \\ \nu_0 \leqslant \alpha_{4} < \min\left(\alpha_3, \frac{1}{2}(1-\alpha_1 -\alpha_2 -\alpha_3)\right) \\ \boldsymbol{\alpha}_{4} \notin \boldsymbol{G}_{4},\ \boldsymbol{\alpha}_{4} \notin \boldsymbol{D}^{*},\ \boldsymbol{\alpha}_{4} \in \boldsymbol{D}^{\ddag} \\ \nu_0 \leqslant \alpha_{5} < \min\left(\alpha_4, \frac{1}{2}(1-\alpha_1 -\alpha_2 -\alpha_3 -\alpha_4)\right) \\ \boldsymbol{\alpha}_{5} \notin \boldsymbol{G}_{5} \\ \nu_0 \leqslant \alpha_{6} < \frac{1}{2}\alpha_1 \\ \boldsymbol{\alpha}_{6}^{\ddag} \notin \boldsymbol{G}_{6},\ \boldsymbol{\alpha}_{6}^{\ddag} \notin \boldsymbol{D}^{**},\ \boldsymbol{\alpha}_{6}^{\ddag \prime} \in \boldsymbol{D}^{**} } }S\left(\mathcal{A}_{\gamma p_{2} p_{3} p_{4} p_{5} p_{6}}, p_{6}\right) \\
=&\ T_{0331} + T_{0332} + T_{0333}.
\end{align}
We discard the whole of $T_{0331}$. For $T_{0332}$ we perform a straightforward decomposition to reach an eight-dimensional sum
\begin{equation}
\sum_{\substack{\boldsymbol{\alpha}_{2} \in C,\ \boldsymbol{\alpha}_{2} \notin \boldsymbol{G}_{2} \\ \nu_0 \leqslant \alpha_{3} < \min\left(\alpha_2, \frac{1}{2}(1-\alpha_1 -\alpha_2)\right) \\ \boldsymbol{\alpha}_{3} \notin \boldsymbol{G}_{3} \\ \boldsymbol{\alpha}_{3} \text{ can be partitioned into } (m, n) \in \boldsymbol{D}_{0} \text{ or } (m, n, h) \in \boldsymbol{D}_{1} \cup \boldsymbol{D}_{2} \\ \nu_0 \leqslant \alpha_{4} < \min\left(\alpha_3, \frac{1}{2}(1-\alpha_1 -\alpha_2 -\alpha_3)\right) \\ \boldsymbol{\alpha}_{4} \notin \boldsymbol{G}_{4},\ \boldsymbol{\alpha}_{4} \notin \boldsymbol{D}^{*},\ \boldsymbol{\alpha}_{4} \in \boldsymbol{D}^{\ddag} \\ \nu_0 \leqslant \alpha_{5} < \min\left(\alpha_4, \frac{1}{2}(1-\alpha_1 -\alpha_2 -\alpha_3 -\alpha_4)\right) \\ \boldsymbol{\alpha}_{5} \notin \boldsymbol{G}_{5} \\ \nu_0 \leqslant \alpha_{6} < \frac{1}{2}\alpha_1 \\ \boldsymbol{\alpha}_{6}^{\ddag} \notin \boldsymbol{G}_{6},\ \boldsymbol{\alpha}_{6}^{\ddag} \in \boldsymbol{D}^{**} \\ \nu_0 \leqslant \alpha_{7} < \min\left(\alpha_6, \frac{1}{2}(\alpha_1 -\alpha_6)\right) \\ \boldsymbol{\alpha}_{7}^{\ddag} \notin \boldsymbol{G}_{7} \\ \nu_0 \leqslant \alpha_{8} < \min\left(\alpha_7, \frac{1}{2}(\alpha_1 -\alpha_6 -\alpha_7)\right) \\ \boldsymbol{\alpha}_{8}^{\ddag} \notin \boldsymbol{G}_{8} } }S\left(\mathcal{A}_{\gamma p_{2} p_{3} p_{4} p_{5} p_{6} p_{7} p_{8}}, p_{8}\right),
\end{equation}
where 
$$
\boldsymbol{\alpha}_{7}^{\ddag} =(1-\alpha_1-\alpha_2-\alpha_3-\alpha_4-\alpha_5,\ \alpha_2,\ \alpha_3,\ \alpha_4,\ \alpha_5,\ \alpha_6,\ \alpha_7)
$$
and
$$
\boldsymbol{\alpha}_{8}^{\ddag} =(1-\alpha_1-\alpha_2-\alpha_3-\alpha_4-\alpha_5,\ \alpha_2,\ \alpha_3,\ \alpha_4,\ \alpha_5,\ \alpha_6,\ \alpha_7,\ \alpha_8).
$$
For $T_{0333}$ we reverse the roles of $\gamma$ and $\gamma_1$ to get
\begin{align}
\nonumber T_{0333} =&\ \sum_{\substack{\boldsymbol{\alpha}_{2} \in C,\ \boldsymbol{\alpha}_{2} \notin \boldsymbol{G}_{2} \\ \nu_0 \leqslant \alpha_{3} < \min\left(\alpha_2, \frac{1}{2}(1-\alpha_1 -\alpha_2)\right) \\ \boldsymbol{\alpha}_{3} \notin \boldsymbol{G}_{3} \\ \boldsymbol{\alpha}_{3} \text{ can be partitioned into } (m, n) \in \boldsymbol{D}_{0} \text{ or } (m, n, h) \in \boldsymbol{D}_{1} \cup \boldsymbol{D}_{2} \\ \nu_0 \leqslant \alpha_{4} < \min\left(\alpha_3, \frac{1}{2}(1-\alpha_1 -\alpha_2 -\alpha_3)\right) \\ \boldsymbol{\alpha}_{4} \notin \boldsymbol{G}_{4},\ \boldsymbol{\alpha}_{4} \notin \boldsymbol{D}^{*},\ \boldsymbol{\alpha}_{4} \in \boldsymbol{D}^{\ddag} \\ \nu_0 \leqslant \alpha_{5} < \min\left(\alpha_4, \frac{1}{2}(1-\alpha_1 -\alpha_2 -\alpha_3 -\alpha_4)\right) \\ \boldsymbol{\alpha}_{5} \notin \boldsymbol{G}_{5} \\ \nu_0 \leqslant \alpha_{6} < \frac{1}{2}\alpha_1 \\ \boldsymbol{\alpha}_{6}^{\ddag} \notin \boldsymbol{G}_{6},\ \boldsymbol{\alpha}_{6}^{\ddag} \notin \boldsymbol{D}^{**},\ \boldsymbol{\alpha}_{6}^{\ddag \prime} \in \boldsymbol{D}^{**} } }S\left(\mathcal{A}_{\gamma p_{2} p_{3} p_{4} p_{5} p_{6}}, p_{6}\right) \\
=&\ \sum_{\substack{\boldsymbol{\alpha}_{2} \in C,\ \boldsymbol{\alpha}_{2} \notin \boldsymbol{G}_{2} \\ \nu_0 \leqslant \alpha_{3} < \min\left(\alpha_2, \frac{1}{2}(1-\alpha_1 -\alpha_2)\right) \\ \boldsymbol{\alpha}_{3} \notin \boldsymbol{G}_{3} \\ \boldsymbol{\alpha}_{3} \text{ can be partitioned into } (m, n) \in \boldsymbol{D}_{0} \text{ or } (m, n, h) \in \boldsymbol{D}_{1} \cup \boldsymbol{D}_{2} \\ \nu_0 \leqslant \alpha_{4} < \min\left(\alpha_3, \frac{1}{2}(1-\alpha_1 -\alpha_2 -\alpha_3)\right) \\ \boldsymbol{\alpha}_{4} \notin \boldsymbol{G}_{4},\ \boldsymbol{\alpha}_{4} \notin \boldsymbol{D}^{*},\ \boldsymbol{\alpha}_{4} \in \boldsymbol{D}^{\ddag} \\ \nu_0 \leqslant \alpha_{5} < \min\left(\alpha_4, \frac{1}{2}(1-\alpha_1 -\alpha_2 -\alpha_3 -\alpha_4)\right) \\ \boldsymbol{\alpha}_{5} \notin \boldsymbol{G}_{5} \\ \nu_0 \leqslant \alpha_{6} < \frac{1}{2}\alpha_1 \\ \boldsymbol{\alpha}_{6}^{\ddag} \notin \boldsymbol{G}_{6},\ \boldsymbol{\alpha}_{6}^{\ddag} \notin \boldsymbol{D}^{**},\ \boldsymbol{\alpha}_{6}^{\ddag \prime} \in \boldsymbol{D}^{**} } }S\left(\mathcal{A}_{\gamma_1 p_{2} p_{3} p_{4} p_{5} p_{6}}, p_{5}\right),
\end{align}
where we can perform a straightforward decomposition on the sum on the right hand side to reach an eight-dimensional sum
\begin{equation}
\sum_{\substack{\boldsymbol{\alpha}_{2} \in C,\ \boldsymbol{\alpha}_{2} \notin \boldsymbol{G}_{2} \\ \nu_0 \leqslant \alpha_{3} < \min\left(\alpha_2, \frac{1}{2}(1-\alpha_1 -\alpha_2)\right) \\ \boldsymbol{\alpha}_{3} \notin \boldsymbol{G}_{3} \\ \boldsymbol{\alpha}_{3} \text{ can be partitioned into } (m, n) \in \boldsymbol{D}_{0} \text{ or } (m, n, h) \in \boldsymbol{D}_{1} \cup \boldsymbol{D}_{2} \\ \nu_0 \leqslant \alpha_{4} < \min\left(\alpha_3, \frac{1}{2}(1-\alpha_1 -\alpha_2 -\alpha_3)\right) \\ \boldsymbol{\alpha}_{4} \notin \boldsymbol{G}_{4},\ \boldsymbol{\alpha}_{4} \notin \boldsymbol{D}^{*},\ \boldsymbol{\alpha}_{4} \in \boldsymbol{D}^{\ddag} \\ \nu_0 \leqslant \alpha_{5} < \min\left(\alpha_4, \frac{1}{2}(1-\alpha_1 -\alpha_2 -\alpha_3 -\alpha_4)\right) \\ \boldsymbol{\alpha}_{5} \notin \boldsymbol{G}_{5} \\ \nu_0 \leqslant \alpha_{6} < \frac{1}{2}\alpha_1 \\ \boldsymbol{\alpha}_{6}^{\ddag} \notin \boldsymbol{G}_{6},\ \boldsymbol{\alpha}_{6}^{\ddag} \notin \boldsymbol{D}^{**},\ \boldsymbol{\alpha}_{6}^{\ddag \prime} \in \boldsymbol{D}^{**} \\ \nu_0 \leqslant \alpha_{7} < \min\left(\alpha_5, \frac{1}{2}(1-\alpha_1 -\alpha_2 -\alpha_3 -\alpha_4 -\alpha_5)\right) \\ \boldsymbol{\alpha}_{7}^{\ddag \prime} \notin \boldsymbol{G}_{7} \\ \nu_0 \leqslant \alpha_{8} < \min\left(\alpha_7, \frac{1}{2}(1-\alpha_1 -\alpha_2 -\alpha_3 -\alpha_4 -\alpha_5 -\alpha_7)\right) \\ \boldsymbol{\alpha}_{8}^{\ddag \prime} \notin \boldsymbol{G}_{8} } }S\left(\mathcal{A}_{\gamma_1 p_{2} p_{3} p_{4} p_{5} p_{6} p_{7} p_{8}}, p_{8}\right),
\end{equation}
where
$$
\boldsymbol{\alpha}_{7}^{\ddag \prime} =(\alpha_1-\alpha_6,\ \alpha_2,\ \alpha_3,\ \alpha_4,\ \alpha_5,\ \alpha_6,\ \alpha_7)
$$
and
$$
\boldsymbol{\alpha}_{8}^{\ddag \prime} =(\alpha_1-\alpha_6,\ \alpha_2,\ \alpha_3,\ \alpha_4,\ \alpha_5,\ \alpha_6,\ \alpha_7,\ \alpha_8).
$$
We can also use Buchstab's identity in reverse on those sums, but the corresponding savings are quite small.

For $T_{04}$ we can also use the devices mentioned earlier to take into account the savings over this sum. Note that there are two almost-prime variables counted by this sum, so the use of straightforward decompositions is just like the case in $T_{033}$. The sum $T_{04}$ counts numbers of the form $\beta p_2 p_3 p_4 \beta_1$, where $\beta \sim x^{1-\alpha_1 -\alpha_2 -\alpha_3}$, $\left(\beta, P(p_{3})\right)=1$, $\beta_1 \sim x^{\alpha_1 -\alpha_4}$ and $\left(\beta_1, P(p_{4})\right)=1$. Here we can decompose either $\beta$ or $\beta_1$, leading to the six-dimensional sums
\begin{equation}
\sum_{\substack{\boldsymbol{\alpha}_{2} \in C,\ \boldsymbol{\alpha}_{2} \notin \boldsymbol{G}_{2} \\ \nu_0 \leqslant \alpha_{3} < \min\left(\alpha_2, \frac{1}{2}(1-\alpha_1 -\alpha_2)\right) \\ \boldsymbol{\alpha}_{3} \notin \boldsymbol{G}_{3} \\ \boldsymbol{\alpha}_{3} \text{ cannot be partitioned into } (m, n) \in \boldsymbol{D}_{0} \text{ or } (m, n, h) \in \boldsymbol{D}_{1} \cup \boldsymbol{D}_{2} \\ \nu_0 \leqslant \alpha_{4} < \frac{1}{2}\alpha_1 \\ \boldsymbol{\alpha}_{4}^{\prime} \notin \boldsymbol{G}_{4},\ \boldsymbol{\alpha}_{4}^{\prime} \in \boldsymbol{D}^{*} \\ \nu_0 \leqslant \alpha_{5} < \min\left(\alpha_4, \frac{1}{2}(\alpha_1 -\alpha_4)\right) \\ \boldsymbol{\alpha}_{5}^{\prime} \notin \boldsymbol{G}_{5} \\ \nu_0 \leqslant \alpha_{6} < \min\left(\alpha_5, \frac{1}{2}(\alpha_1 -\alpha_4 -\alpha_5)\right) \\ \boldsymbol{\alpha}_{6}^{\prime} \notin \boldsymbol{G}_{6} } }S\left(\mathcal{A}_{\beta p_{2} p_{3} p_{4} p_{5} p_{6}}, p_{6}\right)
\end{equation}
and
\begin{equation}
\sum_{\substack{\boldsymbol{\alpha}_{2} \in C,\ \boldsymbol{\alpha}_{2} \notin \boldsymbol{G}_{2} \\ \nu_0 \leqslant \alpha_{3} < \min\left(\alpha_2, \frac{1}{2}(1-\alpha_1 -\alpha_2)\right) \\ \boldsymbol{\alpha}_{3} \notin \boldsymbol{G}_{3} \\ \boldsymbol{\alpha}_{3} \text{ cannot be partitioned into } (m, n) \in \boldsymbol{D}_{0} \text{ or } (m, n, h) \in \boldsymbol{D}_{1} \cup \boldsymbol{D}_{2} \\ \nu_0 \leqslant \alpha_{4} < \frac{1}{2}\alpha_1 \\ \boldsymbol{\alpha}_{4}^{\prime} \notin \boldsymbol{G}_{4},\ \boldsymbol{\alpha}_{4}^{\prime} \notin \boldsymbol{D}^{*},\ \boldsymbol{\alpha}_{4}^{\prime \prime} \in \boldsymbol{D}^{*} \\ \nu_0 \leqslant \alpha_{5} < \min\left(\alpha_3, \frac{1}{2}(1- \alpha_1 -\alpha_2 -\alpha_3)\right) \\ \boldsymbol{\alpha}_{5}^{\prime \prime} \notin \boldsymbol{G}_{5} \\ \nu_0 \leqslant \alpha_{6} < \min\left(\alpha_5, \frac{1}{2}(1- \alpha_1 -\alpha_2 -\alpha_3 -\alpha_5)\right) \\ \boldsymbol{\alpha}_{6}^{\prime \prime} \notin \boldsymbol{G}_{6} } }S\left(\mathcal{A}_{\beta_1 p_{2} p_{3} p_{4} p_{5} p_{6}}, p_{6}\right),
\end{equation}
where 
$$
\boldsymbol{\alpha}_{5}^{\prime} = (1-\alpha_1-\alpha_2-\alpha_3,\ \alpha_2,\ \alpha_3,\ \alpha_4,\ \alpha_5),
$$
$$
\boldsymbol{\alpha}_{6}^{\prime} = (1-\alpha_1-\alpha_2-\alpha_3,\ \alpha_2,\ \alpha_3,\ \alpha_4,\ \alpha_5,\ \alpha_6),
$$
$$
\boldsymbol{\alpha}_{4}^{\prime \prime} = (\alpha_1-\alpha_4,\ \alpha_2,\ \alpha_4,\ \alpha_3),
$$
$$
\boldsymbol{\alpha}_{5}^{\prime \prime} = (\alpha_1-\alpha_4,\ \alpha_2,\ \alpha_3,\ \alpha_4,\ \alpha_5)
$$
and
$$
\boldsymbol{\alpha}_{6}^{\prime \prime} = (\alpha_1-\alpha_4,\ \alpha_2,\ \alpha_3,\ \alpha_4,\ \alpha_5,\ \alpha_6).
$$
On the remaining of $T_{04}$ (with $\boldsymbol{\alpha}_{4}^{\prime} \notin \boldsymbol{D}^{*}$ and $\boldsymbol{\alpha}_{4}^{\prime \prime} \notin \boldsymbol{D}^{*}$) we can use Buchstab's identity in reverse to make savings, and the use of reversed Buchstab's identity over this sum can be seen as the following: the remaining sum counts numbers of the form $\beta p_2 p_3 p_4 \beta_1$, hence we can decompose either $\beta$ or $\beta_1$, leading to the savings of numbers of the forms
$$
\beta p_2 p_3 p_4 (p_5 \beta_2) \quad \text{where } \beta_2 \sim x^{\alpha_1 -\alpha_4 -\alpha_5},\ \left(\beta_2, P(p_{5})\right)=1,\ \boldsymbol{\alpha}_{5}^{\prime} \in \boldsymbol{G}_{5}
$$
and
$$
(\beta_3 p_5) p_2 p_3 p_4 \beta_1 \quad \text{where } \beta_3 \sim x^{1-\alpha_1 -\alpha_2 -\alpha_3 -\alpha_5},\ \left(\beta_3, P(p_{5})\right)=1,\ \boldsymbol{\alpha}_{5}^{\prime \prime} \in \boldsymbol{G}_{5}.
$$
Here, the numbers of the form
$$
(\beta_3 p_6) p_2 p_3 p_4 (p_5 \beta_2) \quad \text{where } \boldsymbol{\alpha}_{5}^{\prime} \in \boldsymbol{G}_{5} \text{ and } \boldsymbol{\alpha}_{5}^{\prime \prime \prime} := (\alpha_1 -\alpha_4, \alpha_2, \alpha_3, \alpha_4, \alpha_6) \in \boldsymbol{G}_{5}
$$
are counted twice, hence we need to subtract them from the savings. For simplicity we omit the sieve iteration process of this sum. One can see the second part of the estimation of $\Phi_7$ in \cite{Jia928} to understand this decomposing procedure.

Altogether we get a loss from $\sum_{C}$ of
\begin{align}
\nonumber & \left( \int_{\boldsymbol{t}_{4} \in U_{C01}} \frac{\omega \left(\frac{1 - t_1 - t_2 - t_3 - t_4}{t_4}\right)}{t_1 t_2 t_3 t_4^2} d t_4 d t_3 d t_2 d t_1 \right) \\
\nonumber -& \left( \int_{\boldsymbol{t}_{5} \in U_{C02}} \frac{\omega \left(\frac{1 - t_1 - t_2 - t_3 - t_4 - t_5}{t_5}\right)}{t_1 t_2 t_3 t_4 t_5^2} d t_5 d t_4 d t_3 d t_2 d t_1 \right) \\
\nonumber +& \left( \int_{\boldsymbol{t}_{6} \in U_{C03}} \frac{\omega \left(\frac{1 - t_1 - t_2 - t_3 - t_4 - t_5 - t_6}{t_6}\right)}{t_1 t_2 t_3 t_4 t_5 t_6^2} d t_6 d t_5 d t_4 d t_3 d t_2 d t_1 \right) \\
\nonumber +& \left( \int_{\boldsymbol{t}_{6} \in U_{C04}} \frac{\omega \left(\frac{t_1 - t_6}{t_6}\right) \omega \left(\frac{1 - t_1 - t_2 - t_3 - t_4 - t_5}{t_5}\right)}{t_2 t_3 t_4 t_5^2 t_6^2} d t_6 d t_5 d t_4 d t_3 d t_2 d t_1 \right) \\
\nonumber +& \left( \int_{\boldsymbol{t}_{8} \in U_{C05}} \frac{\omega \left(\frac{1 - t_1 - t_2 - t_3 - t_4 - t_5 - t_6 - t_7 - t_8}{t_8}\right)}{t_1 t_2 t_3 t_4 t_5 t_6 t_7 t_8^2} d t_8 d t_7 d t_6 d t_5 d t_4 d t_3 d t_2 d t_1 \right) \\
\nonumber +& \left( \int_{\boldsymbol{t}_{8} \in U_{C06}} \frac{\omega \left(\frac{t_1 - t_6 - t_7 - t_8}{t_8}\right) \omega \left(\frac{1 - t_1 - t_2 - t_3 - t_4 - t_5}{t_5}\right)}{t_2 t_3 t_4 t_5^2 t_6 t_7 t_8^2} d t_8 d t_7 d t_6 d t_5 d t_4 d t_3 d t_2 d t_1 \right) \\
\nonumber +& \left( \int_{\boldsymbol{t}_{8} \in U_{C07}} \frac{\omega \left(\frac{t_1 - t_6}{t_6}\right) \omega \left(\frac{1 - t_1 - t_2 - t_3 - t_4 - t_5  - t_7 - t_8}{t_8}\right)}{t_2 t_3 t_4 t_5 t_6^2 t_7 t_8^2} d t_8 d t_7 d t_6 d t_5 d t_4 d t_3 d t_2 d t_1 \right) \\
\nonumber +& \left( \int_{\boldsymbol{t}_{4} \in U_{C08}} \frac{\omega \left(\frac{t_1 - t_4}{t_4}\right) \omega \left(\frac{1 - t_1 - t_2 - t_3}{t_3}\right)}{t_2 t_3^2 t_4^2} d t_4 d t_3 d t_2 d t_1 \right) \\
\nonumber -& \left( \int_{\boldsymbol{t}_{5} \in U_{C09}} \frac{\omega \left(\frac{t_1 - t_4 - t_5}{t_5}\right) \omega \left(\frac{1 - t_1 - t_2 - t_3}{t_3}\right)}{t_2 t_3^2 t_4 t_5^2} d t_5 d t_4 d t_3 d t_2 d t_1 \right) \\
\nonumber -& \left( \int_{\boldsymbol{t}_{5} \in U_{C10}} \frac{\omega \left(\frac{t_1 - t_4}{t_4}\right) \omega \left(\frac{1 - t_1 - t_2 - t_3 - t_5}{t_5}\right)}{t_2 t_3 t_4^2 t_5^2} d t_5 d t_4 d t_3 d t_2 d t_1 \right) \\
\nonumber +& \left( \int_{\boldsymbol{t}_{6} \in U_{C11}} \frac{\omega \left(\frac{t_1 - t_4 - t_5}{t_5}\right) \omega \left(\frac{1 - t_1 - t_2 - t_3 - t_6}{t_6}\right)}{t_2 t_3 t_4 t_5^2 t_6^2} d t_6 d t_5 d t_4 d t_3 d t_2 d t_1 \right) \\
\nonumber +& \left( \int_{\boldsymbol{t}_{6} \in U_{C12}} \frac{\omega \left(\frac{t_1 - t_4 - t_5 - t_6}{t_6}\right) \omega \left(\frac{1 - t_1 - t_2 - t_3}{t_3}\right)}{t_2 t_3^2 t_4 t_5 t_6^2} d t_6 d t_5 d t_4 d t_3 d t_2 d t_1 \right) \\
\nonumber +& \left( \int_{\boldsymbol{t}_{6} \in U_{C13}} \frac{\omega \left(\frac{t_1 - t_4}{t_4}\right) \omega \left(\frac{1 - t_1 - t_2 - t_3 - t_5 - t_6}{t_6}\right)}{t_2 t_3 t_4^2 t_5 t_6^2} d t_6 d t_5 d t_4 d t_3 d t_2 d t_1 \right) \\
\nonumber \leqslant& \left( \int_{\boldsymbol{t}_{4} \in U_{C01}} \frac{\omega_1 \left(\frac{1 - t_1 - t_2 - t_3 - t_4}{t_4}\right)}{t_1 t_2 t_3 t_4^2} d t_4 d t_3 d t_2 d t_1 \right) \\
\nonumber -& \left( \int_{\boldsymbol{t}_{5} \in U_{C02}} \frac{\omega_0 \left(\frac{1 - t_1 - t_2 - t_3 - t_4 - t_5}{t_5}\right)}{t_1 t_2 t_3 t_4 t_5^2} d t_5 d t_4 d t_3 d t_2 d t_1 \right) \\
\nonumber +& \left( \int_{\boldsymbol{t}_{6} \in U_{C03}} \frac{\omega_1 \left(\frac{1 - t_1 - t_2 - t_3 - t_4 - t_5 - t_6}{t_6}\right)}{t_1 t_2 t_3 t_4 t_5 t_6^2} d t_6 d t_5 d t_4 d t_3 d t_2 d t_1 \right) \\
\nonumber +& \left( \int_{\boldsymbol{t}_{6} \in U_{C04}} \frac{\omega_1 \left(\frac{t_1 - t_6}{t_6}\right) \omega_1 \left(\frac{1 - t_1 - t_2 - t_3 - t_4 - t_5}{t_5}\right)}{t_2 t_3 t_4 t_5^2 t_6^2} d t_6 d t_5 d t_4 d t_3 d t_2 d t_1 \right) \\
\nonumber +& \left( \int_{\boldsymbol{t}_{8} \in U_{C05}} \frac{\max \left(\frac{t_8}{1 - t_1 - t_2 - t_3 - t_4 - t_5 - t_6 - t_7 - t_8}, 0.5672\right)}{t_1 t_2 t_3 t_4 t_5 t_6 t_7 t_8^2} d t_8 d t_7 d t_6 d t_5 d t_4 d t_3 d t_2 d t_1 \right) \\
\nonumber +& \left( \int_{\boldsymbol{t}_{8} \in U_{C06}} \frac{\max \left(\frac{t_8}{t_1 - t_6 - t_7 - t_8}, 0.5672\right) \max \left(\frac{t_5}{1 - t_1 - t_2 - t_3 - t_4 - t_5}, 0.5672\right)}{t_2 t_3 t_4 t_5^2 t_6 t_7 t_8^2} d t_8 d t_7 d t_6 d t_5 d t_4 d t_3 d t_2 d t_1 \right) \\
\nonumber +& \left( \int_{\boldsymbol{t}_{8} \in U_{C07}} \frac{\max \left(\frac{t_6}{t_1 - t_6}, 0.5672\right) \max \left(\frac{t_8}{1 - t_1 - t_2 - t_3 - t_4 - t_5 - t_7 - t_8}, 0.5672\right)}{t_2 t_3 t_4 t_5 t_6^2 t_7 t_8^2} d t_8 d t_7 d t_6 d t_5 d t_4 d t_3 d t_2 d t_1 \right) \\
\nonumber +& \left( \int_{\boldsymbol{t}_{4} \in U_{C08}} \frac{\omega_1 \left(\frac{t_1 - t_4}{t_4}\right) \omega_1 \left(\frac{1 - t_1 - t_2 - t_3}{t_3}\right)}{t_2 t_3^2 t_4^2} d t_4 d t_3 d t_2 d t_1 \right) \\
\nonumber -& \left( \int_{\boldsymbol{t}_{5} \in U_{C09}} \frac{\omega_0 \left(\frac{t_1 - t_4 - t_5}{t_5}\right) \omega_0 \left(\frac{1 - t_1 - t_2 - t_3}{t_3}\right)}{t_2 t_3^2 t_4 t_5^2} d t_5 d t_4 d t_3 d t_2 d t_1 \right) \\
\nonumber -& \left( \int_{\boldsymbol{t}_{5} \in U_{C10}} \frac{\omega_0 \left(\frac{t_1 - t_4}{t_4}\right) \omega_0 \left(\frac{1 - t_1 - t_2 - t_3 - t_5}{t_5}\right)}{t_2 t_3 t_4^2 t_5^2} d t_5 d t_4 d t_3 d t_2 d t_1 \right) \\
\nonumber +& \left( \int_{\boldsymbol{t}_{6} \in U_{C11}} \frac{\omega_1 \left(\frac{t_1 - t_4 - t_5}{t_5}\right) \omega_1 \left(\frac{1 - t_1 - t_2 - t_3 - t_6}{t_6}\right)}{t_2 t_3 t_4 t_5^2 t_6^2} d t_6 d t_5 d t_4 d t_3 d t_2 d t_1 \right) \\
\nonumber +& \left( \int_{\boldsymbol{t}_{6} \in U_{C12}} \frac{\omega_1 \left(\frac{t_1 - t_4 - t_5 - t_6}{t_6}\right) \omega_1 \left(\frac{1 - t_1 - t_2 - t_3}{t_3}\right)}{t_2 t_3^2 t_4 t_5 t_6^2} d t_6 d t_5 d t_4 d t_3 d t_2 d t_1 \right) \\
\nonumber +& \left( \int_{\boldsymbol{t}_{6} \in U_{C13}} \frac{\omega_1 \left(\frac{t_1 - t_4}{t_4}\right) \omega_1 \left(\frac{1 - t_1 - t_2 - t_3 - t_5 - t_6}{t_6}\right)}{t_2 t_3 t_4^2 t_5 t_6^2} d t_6 d t_5 d t_4 d t_3 d t_2 d t_1 \right) \\
\nonumber \leqslant&\ 0.21 - 0 + 0.015 + 0.05 + 0.001 + 0.001 + 0.001 + 0.22 - (0 + 0 - 0) + 0.015 + 0.001 \\
=&\ 0.514,
\end{align}
where
\begin{align}
\nonumber U_{C01}(\boldsymbol{\alpha}_{4}) :=&\ \left\{ \boldsymbol{\alpha}_{2} \in C,\ \boldsymbol{\alpha}_{2} \notin \boldsymbol{G}_{2},\ \nu_0 \leqslant \alpha_{3} < \min\left(\alpha_2, \frac{1}{2}(1-\alpha_1 -\alpha_2)\right),\ \boldsymbol{\alpha}_{3} \notin \boldsymbol{G}_{3}, \right. \\
\nonumber & \quad \boldsymbol{\alpha}_{3} \text{ can be partitioned into } (m, n) \in \boldsymbol{D}_{0} \text{ or } (m, n, h) \in \boldsymbol{D}_{1} \cup \boldsymbol{D}_{2}, \\
\nonumber & \quad \nu_0 \leqslant \alpha_{4} < \min\left(\alpha_3, \frac{1}{2}(1-\alpha_1 -\alpha_2 -\alpha_3)\right),\ \boldsymbol{\alpha}_{4} \notin \boldsymbol{G}_{4},\ \boldsymbol{\alpha}_{4} \notin \boldsymbol{D}^{*},\ \boldsymbol{\alpha}_{4} \notin \boldsymbol{D}^{\ddag}, \\
\nonumber & \left. \quad \nu_0 \leqslant \alpha_1 < \frac{1}{2},\ \nu\left(\alpha_1\right) \leqslant \alpha_2 < \min\left(\alpha_1, \frac{1}{2}(1-\alpha_1) \right) \right\}, \\
\nonumber U_{C02}(\boldsymbol{\alpha}_{5}) :=&\ \left\{ \boldsymbol{\alpha}_{2} \in C,\ \boldsymbol{\alpha}_{2} \notin \boldsymbol{G}_{2},\ \nu_0 \leqslant \alpha_{3} < \min\left(\alpha_2, \frac{1}{2}(1-\alpha_1 -\alpha_2)\right),\ \boldsymbol{\alpha}_{3} \notin \boldsymbol{G}_{3}, \right. \\
\nonumber & \quad \boldsymbol{\alpha}_{3} \text{ can be partitioned into } (m, n) \in \boldsymbol{D}_{0} \text{ or } (m, n, h) \in \boldsymbol{D}_{1} \cup \boldsymbol{D}_{2}, \\
\nonumber & \quad \nu_0 \leqslant \alpha_{4} < \min\left(\alpha_3, \frac{1}{2}(1-\alpha_1 -\alpha_2 -\alpha_3)\right),\ \boldsymbol{\alpha}_{4} \notin \boldsymbol{G}_{4},\ \boldsymbol{\alpha}_{4} \notin \boldsymbol{D}^{*},\ \boldsymbol{\alpha}_{4} \notin \boldsymbol{D}^{\ddag}, \\
\nonumber & \quad \alpha_4 < \alpha_5 < \frac{1}{2}(1-\alpha_1 -\alpha_2 -\alpha_3 -\alpha_4),\ \boldsymbol{\alpha}_{5} \in \boldsymbol{G}_{5}, \\
\nonumber & \left. \quad \nu_0 \leqslant \alpha_1 < \frac{1}{2},\ \nu\left(\alpha_1\right) \leqslant \alpha_2 < \min\left(\alpha_1, \frac{1}{2}(1-\alpha_1) \right) \right\}, \\
\nonumber U_{C03}(\boldsymbol{\alpha}_{6}) :=&\ \left\{ \boldsymbol{\alpha}_{2} \in C,\ \boldsymbol{\alpha}_{2} \notin \boldsymbol{G}_{2},\ \nu_0 \leqslant \alpha_{3} < \min\left(\alpha_2, \frac{1}{2}(1-\alpha_1 -\alpha_2)\right),\ \boldsymbol{\alpha}_{3} \notin \boldsymbol{G}_{3}, \right. \\
\nonumber & \quad \boldsymbol{\alpha}_{3} \text{ can be partitioned into } (m, n) \in \boldsymbol{D}_{0} \text{ or } (m, n, h) \in \boldsymbol{D}_{1} \cup \boldsymbol{D}_{2}, \\
\nonumber & \quad \nu_0 \leqslant \alpha_{4} < \min\left(\alpha_3, \frac{1}{2}(1-\alpha_1 -\alpha_2 -\alpha_3)\right),\ \boldsymbol{\alpha}_{4} \notin \boldsymbol{G}_{4},\ \boldsymbol{\alpha}_{4} \in \boldsymbol{D}^{*}, \\
\nonumber & \quad \nu_0 \leqslant \alpha_{5} < \min\left(\alpha_4, \frac{1}{2}(1-\alpha_1 -\alpha_2 -\alpha_3 -\alpha_4)\right),\ \boldsymbol{\alpha}_{5} \notin \boldsymbol{G}_{5}, \\
\nonumber & \quad \nu_0 \leqslant \alpha_{6} < \min\left(\alpha_5, \frac{1}{2}(1-\alpha_1 -\alpha_2 -\alpha_3 -\alpha_4 -\alpha_5)\right),\ \boldsymbol{\alpha}_{6} \notin \boldsymbol{G}_{6},\ \boldsymbol{\alpha}_{6} \notin \boldsymbol{D}^{**}, \\
\nonumber & \left. \quad \nu_0 \leqslant \alpha_1 < \frac{1}{2},\ \nu\left(\alpha_1\right) \leqslant \alpha_2 < \min\left(\alpha_1, \frac{1}{2}(1-\alpha_1) \right) \right\}, \\
\nonumber U_{C04}(\boldsymbol{\alpha}_{6}) :=&\ \left\{ \boldsymbol{\alpha}_{2} \in C,\ \boldsymbol{\alpha}_{2} \notin \boldsymbol{G}_{2},\ \nu_0 \leqslant \alpha_{3} < \min\left(\alpha_2, \frac{1}{2}(1-\alpha_1 -\alpha_2)\right),\ \boldsymbol{\alpha}_{3} \notin \boldsymbol{G}_{3}, \right. \\
\nonumber & \quad \boldsymbol{\alpha}_{3} \text{ can be partitioned into } (m, n) \in \boldsymbol{D}_{0} \text{ or } (m, n, h) \in \boldsymbol{D}_{1} \cup \boldsymbol{D}_{2}, \\
\nonumber & \quad \nu_0 \leqslant \alpha_{4} < \min\left(\alpha_3, \frac{1}{2}(1-\alpha_1 -\alpha_2 -\alpha_3)\right),\ \boldsymbol{\alpha}_{4} \notin \boldsymbol{G}_{4},\ \boldsymbol{\alpha}_{4} \notin \boldsymbol{D}^{*},\ \boldsymbol{\alpha}_{4} \in \boldsymbol{D}^{\ddag}, \\
\nonumber & \quad \nu_0 \leqslant \alpha_{5} < \min\left(\alpha_4, \frac{1}{2}(1-\alpha_1 -\alpha_2 -\alpha_3 -\alpha_4)\right),\ \boldsymbol{\alpha}_{5} \notin \boldsymbol{G}_{5}, \\
\nonumber & \quad \nu_0 \leqslant \alpha_{6} < \frac{1}{2}\alpha_1 ,\ \boldsymbol{\alpha}_{6}^{\ddag} \notin \boldsymbol{G}_{6},\ \boldsymbol{\alpha}_{6}^{\ddag} \notin \boldsymbol{D}^{**},\ \boldsymbol{\alpha}_{6}^{\ddag \prime} \notin \boldsymbol{D}^{**}, \\
\nonumber & \left. \quad \nu_0 \leqslant \alpha_1 < \frac{1}{2},\ \nu\left(\alpha_1\right) \leqslant \alpha_2 < \min\left(\alpha_1, \frac{1}{2}(1-\alpha_1) \right) \right\}, \\
\nonumber U_{C05}(\boldsymbol{\alpha}_{8}) :=&\ \left\{ \boldsymbol{\alpha}_{2} \in C,\ \boldsymbol{\alpha}_{2} \notin \boldsymbol{G}_{2},\ \nu_0 \leqslant \alpha_{3} < \min\left(\alpha_2, \frac{1}{2}(1-\alpha_1 -\alpha_2)\right),\ \boldsymbol{\alpha}_{3} \notin \boldsymbol{G}_{3}, \right. \\
\nonumber & \quad \boldsymbol{\alpha}_{3} \text{ can be partitioned into } (m, n) \in \boldsymbol{D}_{0} \text{ or } (m, n, h) \in \boldsymbol{D}_{1} \cup \boldsymbol{D}_{2}, \\
\nonumber & \quad \nu_0 \leqslant \alpha_{4} < \min\left(\alpha_3, \frac{1}{2}(1-\alpha_1 -\alpha_2 -\alpha_3)\right),\ \boldsymbol{\alpha}_{4} \notin \boldsymbol{G}_{4},\ \boldsymbol{\alpha}_{4} \in \boldsymbol{D}^{*}, \\
\nonumber & \quad \nu_0 \leqslant \alpha_{5} < \min\left(\alpha_4, \frac{1}{2}(1-\alpha_1 -\alpha_2 -\alpha_3 -\alpha_4)\right),\ \boldsymbol{\alpha}_{5} \notin \boldsymbol{G}_{5}, \\
\nonumber & \quad \nu_0 \leqslant \alpha_{6} < \min\left(\alpha_5, \frac{1}{2}(1-\alpha_1 -\alpha_2 -\alpha_3 -\alpha_4 -\alpha_5)\right),\ \boldsymbol{\alpha}_{6} \notin \boldsymbol{G}_{6},\ \boldsymbol{\alpha}_{6} \in \boldsymbol{D}^{**}, \\
\nonumber & \quad \nu_0 \leqslant \alpha_{7} < \min\left(\alpha_6, \frac{1}{2}(1-\alpha_1 -\alpha_2 -\alpha_3 -\alpha_4-\alpha_5-\alpha_6)\right),\ \boldsymbol{\alpha}_{7} \notin \boldsymbol{G}_{7}, \\
\nonumber & \quad \nu_0 \leqslant \alpha_{8} < \min\left(\alpha_7, \frac{1}{2}(1-\alpha_1 -\alpha_2 -\alpha_3 -\alpha_4-\alpha_5-\alpha_6-\alpha_7)\right),\ \boldsymbol{\alpha}_{8} \notin \boldsymbol{G}_{8}, \\
\nonumber & \left. \quad \nu_0 \leqslant \alpha_1 < \frac{1}{2},\ \nu\left(\alpha_1\right) \leqslant \alpha_2 < \min\left(\alpha_1, \frac{1}{2}(1-\alpha_1) \right) \right\}, \\
\nonumber U_{C06}(\boldsymbol{\alpha}_{8}) :=&\ \left\{ \boldsymbol{\alpha}_{2} \in C,\ \boldsymbol{\alpha}_{2} \notin \boldsymbol{G}_{2},\ \nu_0 \leqslant \alpha_{3} < \min\left(\alpha_2, \frac{1}{2}(1-\alpha_1 -\alpha_2)\right),\ \boldsymbol{\alpha}_{3} \notin \boldsymbol{G}_{3}, \right. \\
\nonumber & \quad \boldsymbol{\alpha}_{3} \text{ can be partitioned into } (m, n) \in \boldsymbol{D}_{0} \text{ or } (m, n, h) \in \boldsymbol{D}_{1} \cup \boldsymbol{D}_{2}, \\
\nonumber & \quad \nu_0 \leqslant \alpha_{4} < \min\left(\alpha_3, \frac{1}{2}(1-\alpha_1 -\alpha_2 -\alpha_3)\right),\ \boldsymbol{\alpha}_{4} \notin \boldsymbol{G}_{4},\ \boldsymbol{\alpha}_{4} \notin \boldsymbol{D}^{*},\ \boldsymbol{\alpha}_{4} \in \boldsymbol{D}^{\ddag}, \\
\nonumber & \quad \nu_0 \leqslant \alpha_{5} < \min\left(\alpha_4, \frac{1}{2}(1-\alpha_1 -\alpha_2 -\alpha_3 -\alpha_4)\right),\ \boldsymbol{\alpha}_{5} \notin \boldsymbol{G}_{5}, \\
\nonumber & \quad \nu_0 \leqslant \alpha_{6} < \frac{1}{2}\alpha_1 ,\ \boldsymbol{\alpha}_{6}^{\ddag} \notin \boldsymbol{G}_{6},\ \boldsymbol{\alpha}_{6}^{\ddag} \in \boldsymbol{D}^{**}, \\
\nonumber & \quad \nu_0 \leqslant \alpha_{7} < \min\left(\alpha_6, \frac{1}{2}(\alpha_1 -\alpha_6)\right),\ \boldsymbol{\alpha}_{7}^{\ddag} \notin \boldsymbol{G}_{7}, \\
\nonumber & \quad \nu_0 \leqslant \alpha_{8} < \min\left(\alpha_7, \frac{1}{2}(\alpha_1 -\alpha_6 -\alpha_7)\right),\ \boldsymbol{\alpha}_{8}^{\ddag} \notin \boldsymbol{G}_{8}, \\
\nonumber & \left. \quad \nu_0 \leqslant \alpha_1 < \frac{1}{2},\ \nu\left(\alpha_1\right) \leqslant \alpha_2 < \min\left(\alpha_1, \frac{1}{2}(1-\alpha_1) \right) \right\}, \\
\nonumber U_{C07}(\boldsymbol{\alpha}_{8}) :=&\ \left\{ \boldsymbol{\alpha}_{2} \in C,\ \boldsymbol{\alpha}_{2} \notin \boldsymbol{G}_{2},\ \nu_0 \leqslant \alpha_{3} < \min\left(\alpha_2, \frac{1}{2}(1-\alpha_1 -\alpha_2)\right),\ \boldsymbol{\alpha}_{3} \notin \boldsymbol{G}_{3}, \right. \\
\nonumber & \quad \boldsymbol{\alpha}_{3} \text{ can be partitioned into } (m, n) \in \boldsymbol{D}_{0} \text{ or } (m, n, h) \in \boldsymbol{D}_{1} \cup \boldsymbol{D}_{2}, \\
\nonumber & \quad \nu_0 \leqslant \alpha_{4} < \min\left(\alpha_3, \frac{1}{2}(1-\alpha_1 -\alpha_2 -\alpha_3)\right),\ \boldsymbol{\alpha}_{4} \notin \boldsymbol{G}_{4},\ \boldsymbol{\alpha}_{4} \notin \boldsymbol{D}^{*},\ \boldsymbol{\alpha}_{4} \in \boldsymbol{D}^{\ddag}, \\
\nonumber & \quad \nu_0 \leqslant \alpha_{5} < \min\left(\alpha_4, \frac{1}{2}(1-\alpha_1 -\alpha_2 -\alpha_3 -\alpha_4)\right),\ \boldsymbol{\alpha}_{5} \notin \boldsymbol{G}_{5}, \\
\nonumber & \quad \nu_0 \leqslant \alpha_{6} < \frac{1}{2}\alpha_1 ,\ \boldsymbol{\alpha}_{6}^{\ddag} \notin \boldsymbol{G}_{6},\ \boldsymbol{\alpha}_{6}^{\ddag} \notin \boldsymbol{D}^{**},\ \boldsymbol{\alpha}_{6}^{\ddag \prime} \in \boldsymbol{D}^{**}, \\
\nonumber & \quad \nu_0 \leqslant \alpha_{7} < \min\left(\alpha_5, \frac{1}{2}(1-\alpha_1 -\alpha_2 -\alpha_3 -\alpha_4 -\alpha_5)\right),\ \boldsymbol{\alpha}_{7}^{\ddag \prime} \notin \boldsymbol{G}_{7}, \\
\nonumber & \quad \nu_0 \leqslant \alpha_{8} < \min\left(\alpha_7, \frac{1}{2}(1-\alpha_1 -\alpha_2 -\alpha_3 -\alpha_4 -\alpha_5 -\alpha_7)\right),\ \boldsymbol{\alpha}_{8}^{\ddag \prime} \notin \boldsymbol{G}_{8}, \\
\nonumber & \left. \quad \nu_0 \leqslant \alpha_1 < \frac{1}{2},\ \nu\left(\alpha_1\right) \leqslant \alpha_2 < \min\left(\alpha_1, \frac{1}{2}(1-\alpha_1) \right) \right\}, \\
\nonumber U_{C08}(\boldsymbol{\alpha}_{4}) :=&\ \left\{ \boldsymbol{\alpha}_{2} \in C,\ \boldsymbol{\alpha}_{2} \notin \boldsymbol{G}_{2},\ \nu_0 \leqslant \alpha_{3} < \min\left(\alpha_2, \frac{1}{2}(1-\alpha_1 -\alpha_2)\right),\ \boldsymbol{\alpha}_{3} \notin \boldsymbol{G}_{3}, \right. \\
\nonumber & \quad \boldsymbol{\alpha}_{3} \text{ cannot be partitioned into } (m, n) \in \boldsymbol{D}_{0} \text{ or } (m, n, h) \in \boldsymbol{D}_{1} \cup \boldsymbol{D}_{2}, \\
\nonumber & \quad \nu_0 \leqslant \alpha_{4} < \frac{1}{2}\alpha_1,\ \boldsymbol{\alpha}_{4}^{\prime} \notin \boldsymbol{G}_{4},\ \boldsymbol{\alpha}_{4}^{\prime} \notin \boldsymbol{D}^{*},\ \boldsymbol{\alpha}_{4}^{\prime \prime} \notin \boldsymbol{D}^{*},\\
\nonumber & \left. \quad \nu_0 \leqslant \alpha_1 < \frac{1}{2},\ \nu\left(\alpha_1\right) \leqslant \alpha_2 < \min\left(\alpha_1, \frac{1}{2}(1-\alpha_1) \right) \right\}, \\
\nonumber U_{C09}(\boldsymbol{\alpha}_{5}) :=&\ \left\{ \boldsymbol{\alpha}_{2} \in C,\ \boldsymbol{\alpha}_{2} \notin \boldsymbol{G}_{2},\ \nu_0 \leqslant \alpha_{3} < \min\left(\alpha_2, \frac{1}{2}(1-\alpha_1 -\alpha_2)\right),\ \boldsymbol{\alpha}_{3} \notin \boldsymbol{G}_{3}, \right. \\
\nonumber & \quad \boldsymbol{\alpha}_{3} \text{ cannot be partitioned into } (m, n) \in \boldsymbol{D}_{0} \text{ or } (m, n, h) \in \boldsymbol{D}_{1} \cup \boldsymbol{D}_{2}, \\
\nonumber & \quad \nu_0 \leqslant \alpha_{4} < \frac{1}{2}\alpha_1,\ \boldsymbol{\alpha}_{4}^{\prime} \notin \boldsymbol{G}_{4},\ \boldsymbol{\alpha}_{4}^{\prime} \notin \boldsymbol{D}^{*},\ \boldsymbol{\alpha}_{4}^{\prime \prime} \notin \boldsymbol{D}^{*},\\
\nonumber & \quad \alpha_4 < \alpha_5 < \frac{1}{2}(\alpha_1 -\alpha_4),\ \boldsymbol{\alpha}_{5}^{\prime} \in \boldsymbol{G}_{5}, \\
\nonumber & \left. \quad \nu_0 \leqslant \alpha_1 < \frac{1}{2},\ \nu\left(\alpha_1\right) \leqslant \alpha_2 < \min\left(\alpha_1, \frac{1}{2}(1-\alpha_1) \right) \right\}, \\
\nonumber U_{C10}(\boldsymbol{\alpha}_{5}) :=&\ \left\{ \boldsymbol{\alpha}_{2} \in C,\ \boldsymbol{\alpha}_{2} \notin \boldsymbol{G}_{2},\ \nu_0 \leqslant \alpha_{3} < \min\left(\alpha_2, \frac{1}{2}(1-\alpha_1 -\alpha_2)\right),\ \boldsymbol{\alpha}_{3} \notin \boldsymbol{G}_{3}, \right. \\
\nonumber & \quad \boldsymbol{\alpha}_{3} \text{ cannot be partitioned into } (m, n) \in \boldsymbol{D}_{0} \text{ or } (m, n, h) \in \boldsymbol{D}_{1} \cup \boldsymbol{D}_{2}, \\
\nonumber & \quad \nu_0 \leqslant \alpha_{4} < \frac{1}{2}\alpha_1,\ \boldsymbol{\alpha}_{4}^{\prime} \notin \boldsymbol{G}_{4},\ \boldsymbol{\alpha}_{4}^{\prime} \notin \boldsymbol{D}^{*},\ \boldsymbol{\alpha}_{4}^{\prime \prime} \notin \boldsymbol{D}^{*},\\
\nonumber & \quad \alpha_3 < \alpha_5 < \frac{1}{2}(1 - \alpha_1 - \alpha_2 - \alpha_3),\ \boldsymbol{\alpha}_{5}^{\prime \prime} \in \boldsymbol{G}_{5}, \\
\nonumber & \left. \quad \nu_0 \leqslant \alpha_1 < \frac{1}{2},\ \nu\left(\alpha_1\right) \leqslant \alpha_2 < \min\left(\alpha_1, \frac{1}{2}(1-\alpha_1) \right) \right\}, \\
\nonumber U_{C11}(\boldsymbol{\alpha}_{6}) :=&\ \left\{ \boldsymbol{\alpha}_{2} \in C,\ \boldsymbol{\alpha}_{2} \notin \boldsymbol{G}_{2},\ \nu_0 \leqslant \alpha_{3} < \min\left(\alpha_2, \frac{1}{2}(1-\alpha_1 -\alpha_2)\right),\ \boldsymbol{\alpha}_{3} \notin \boldsymbol{G}_{3}, \right. \\
\nonumber & \quad \boldsymbol{\alpha}_{3} \text{ cannot be partitioned into } (m, n) \in \boldsymbol{D}_{0} \text{ or } (m, n, h) \in \boldsymbol{D}_{1} \cup \boldsymbol{D}_{2}, \\
\nonumber & \quad \nu_0 \leqslant \alpha_{4} < \frac{1}{2}\alpha_1,\ \boldsymbol{\alpha}_{4}^{\prime} \notin \boldsymbol{G}_{4},\ \boldsymbol{\alpha}_{4}^{\prime} \notin \boldsymbol{D}^{*},\ \boldsymbol{\alpha}_{4}^{\prime \prime} \notin \boldsymbol{D}^{*},\\
\nonumber & \quad \alpha_4 < \alpha_5 < \frac{1}{2}(\alpha_1 -\alpha_4),\ \boldsymbol{\alpha}_{5}^{\prime} \in \boldsymbol{G}_{5}, \\
\nonumber & \quad \alpha_3 < \alpha_6 < \frac{1}{2}(1 - \alpha_1 - \alpha_2 - \alpha_3),\ \boldsymbol{\alpha}_{5}^{\prime \prime \prime} \in \boldsymbol{G}_{5}, \\
\nonumber & \left. \quad \nu_0 \leqslant \alpha_1 < \frac{1}{2},\ \nu\left(\alpha_1\right) \leqslant \alpha_2 < \min\left(\alpha_1, \frac{1}{2}(1-\alpha_1) \right) \right\}, \\
\nonumber U_{C12}(\boldsymbol{\alpha}_{6}) :=&\ \left\{ \boldsymbol{\alpha}_{2} \in C,\ \boldsymbol{\alpha}_{2} \notin \boldsymbol{G}_{2},\ \nu_0 \leqslant \alpha_{3} < \min\left(\alpha_2, \frac{1}{2}(1-\alpha_1 -\alpha_2)\right),\ \boldsymbol{\alpha}_{3} \notin \boldsymbol{G}_{3}, \right. \\
\nonumber & \quad \boldsymbol{\alpha}_{3} \text{ cannot be partitioned into } (m, n) \in \boldsymbol{D}_{0} \text{ or } (m, n, h) \in \boldsymbol{D}_{1} \cup \boldsymbol{D}_{2}, \\
\nonumber & \quad \nu_0 \leqslant \alpha_{4} < \frac{1}{2}\alpha_1,\ \boldsymbol{\alpha}_{4}^{\prime} \notin \boldsymbol{G}_{4},\ \boldsymbol{\alpha}_{4}^{\prime} \in \boldsymbol{D}^{*},\\
\nonumber & \quad \nu_0 \leqslant \alpha_{5} < \min\left(\alpha_4, \frac{1}{2}(\alpha_1 -\alpha_4)\right),\ \boldsymbol{\alpha}_{5}^{\prime} \notin \boldsymbol{G}_{5}, \\
\nonumber & \quad \nu_0 \leqslant \alpha_{6} < \min\left(\alpha_5, \frac{1}{2}(\alpha_1 -\alpha_4 -\alpha_5)\right),\ \boldsymbol{\alpha}_{6}^{\prime} \notin \boldsymbol{G}_{6}, \\
\nonumber & \left. \quad \nu_0 \leqslant \alpha_1 < \frac{1}{2},\ \nu\left(\alpha_1\right) \leqslant \alpha_2 < \min\left(\alpha_1, \frac{1}{2}(1-\alpha_1) \right) \right\}, \\
\nonumber U_{C13}(\boldsymbol{\alpha}_{6}) :=&\ \left\{ \boldsymbol{\alpha}_{2} \in C,\ \boldsymbol{\alpha}_{2} \notin \boldsymbol{G}_{2},\ \nu_0 \leqslant \alpha_{3} < \min\left(\alpha_2, \frac{1}{2}(1-\alpha_1 -\alpha_2)\right),\ \boldsymbol{\alpha}_{3} \notin \boldsymbol{G}_{3}, \right. \\
\nonumber & \quad \boldsymbol{\alpha}_{3} \text{ cannot be partitioned into } (m, n) \in \boldsymbol{D}_{0} \text{ or } (m, n, h) \in \boldsymbol{D}_{1} \cup \boldsymbol{D}_{2}, \\
\nonumber & \quad \nu_0 \leqslant \alpha_{4} < \frac{1}{2}\alpha_1,\ \boldsymbol{\alpha}_{4}^{\prime} \notin \boldsymbol{G}_{4},\ \boldsymbol{\alpha}_{4}^{\prime} \notin \boldsymbol{D}^{*},\ \boldsymbol{\alpha}_{4}^{\prime \prime} \in \boldsymbol{D}^{*},\\
\nonumber & \quad \nu_0 \leqslant \alpha_{5} < \min\left(\alpha_3, \frac{1}{2}(1 -\alpha_1 -\alpha_2 -\alpha_3)\right),\ \boldsymbol{\alpha}_{5}^{\prime \prime} \notin \boldsymbol{G}_{5}, \\
\nonumber & \quad \nu_0 \leqslant \alpha_{6} < \min\left(\alpha_5, \frac{1}{2}(1 -\alpha_1 -\alpha_2 -\alpha_3 -\alpha_5)\right),\ \boldsymbol{\alpha}_{6}^{\prime \prime} \notin \boldsymbol{G}_{6}, \\
\nonumber & \left. \quad \nu_0 \leqslant \alpha_1 < \frac{1}{2},\ \nu\left(\alpha_1\right) \leqslant \alpha_2 < \min\left(\alpha_1, \frac{1}{2}(1-\alpha_1) \right) \right\}.
\end{align}
One can also see the integrals corresponding to $U_{C08}$--$U_{C11}$ as an simple, explicit expression of the function $w^{*}(\boldsymbol{\alpha}_{4})$ defined in [\cite{HarmanBOOK}, Chapter 7.9]. We remark that a small part of $C$ is actually covered by $\boldsymbol{G}_{2}$. If we discard the whole of $\sum_{C}$, we would have a loss larger than $1$ which leads to a trivial lower bound.

Finally, by (11), (12) and (30), the total loss from $\sum_{3}$ is less than
$$
2 \times 0.241 + 0.514 < 0.996
$$
and we conclude that
$$
\pi(x)-\pi(x-x^{0.52})=S\left(\mathcal{A}, x^{\frac{1}{2}}\right) \geqslant 0.004 \frac{x^{0.52}}{\log x}.
$$
The lower constant $0.004$ can be slightly improved by more careful decompositions and accurate calculations. The lower bounds for other values of $\theta$ between $0.52$ and $0.525$ can be proved in the same way, so we omit the calculation details. One can check our code for them to verify the numerical calculations.

\section{The final decomposition II: Upper Bound}
In this section, we ignore the presence of $\varepsilon$ for clarity. Let $\omega(u)$, $\omega_0(u)$ and $\omega_1(u)$ denote the same functions as in Section 5. We still use the idea of decomposing $S\left(\mathcal{A}, z\right)$ explained in Section 2 with (6)--(8) to prove our upper bound results. We note that for the upper bound problem, we can only drop negative parts. Fix $\theta =0.52$, $\nu_0 = \nu_{\min} = 2 \theta -1 = 0.04$ and let $p_{j}=x^{\alpha_{j}}$. By Buchstab's identity, we have
\begin{align}
\nonumber S\left(\mathcal{A}, x^{\frac{1}{2}}\right) =&\ S\left(\mathcal{A}, x^{\nu(0)}\right)-\sum_{\nu(0) \leqslant \alpha_{1}< \frac{1}{2}} S\left(\mathcal{A}_{p_{1}}, p_{1}\right) \\
=&\ \sum_{1} - \sum_{2}^{\prime}.
\end{align}
We can give asymptotic formulas for $\sum_{1}$. For $\sum_{2}^{\prime}$, We need to split the whole summation range over $p_{1}$ into different ranges and consider further decompositions in each range because we can only drop negative parts on the upper bound problem. The sets $\boldsymbol{G}$ and $\boldsymbol{D}$ with same superscripts and subscripts as in Section 5 represent the same asymptotic regions. We shall define some new sets and subsets using Lemma~\ref{l23} and partition technique.
Put
\begin{align}
\nonumber \boldsymbol{D}_{3} =&\ \left\{\boldsymbol{\alpha}_{2}: \alpha_{2} \leqslant \alpha_{1},\ 2 \alpha_{1} + \alpha_{2} < 1,\ \alpha_2 < \frac{7}{2} \theta - \frac{3}{2} \right\}, \\
\nonumber \boldsymbol{D}^{+} =&\ \left\{\boldsymbol{\alpha}_{3}: \left(\alpha_{1}, \alpha_{2}, \alpha_{3}, \alpha_{3}\right) \text { can be partitioned into } (m, n) \in \boldsymbol{D}_{0}^{\prime} \text { or }(m, n, h) \in \boldsymbol{D}_{1}^{\prime} \cup \boldsymbol{D}_{2}^{\prime} \right\}, \\
\nonumber \boldsymbol{D}^{+ +} =&\ \left\{\boldsymbol{\alpha}_{5}: \left(\alpha_{1}, \alpha_{2}, \alpha_{3}, \alpha_{4}, \alpha_{5}, \alpha_{5}\right) \text { can be partitioned into } (m, n) \in \boldsymbol{D}_{0}^{\prime} \text { or }(m, n, h) \in \boldsymbol{D}_{1}^{\prime} \cup \boldsymbol{D}_{2}^{\prime} \right\}, \\
\nonumber \boldsymbol{D}^{\#} =&\ \left\{\boldsymbol{\alpha}_{3}: \text{both } \boldsymbol{\alpha}_{3} \text{ and } \left(1- \alpha_{1} - \alpha_2 - \alpha_3 , \alpha_{2}, \alpha_{3} \right) \text { can be partitioned into } (m, n) \in \boldsymbol{D}_{0} \text { or }(m, n, h) \in \boldsymbol{D}_{1} \cup \boldsymbol{D}_{2} \right\}, \\
\nonumber H =&\ \left\{\boldsymbol{\alpha}_{1}: \frac{7}{2} \theta - \frac{3}{2} \leqslant \alpha_1 \leqslant 4 - 7 \theta \right\},
\end{align}
where $\boldsymbol{D}_3$ correspond to conditions on variables that allow a further decomposition, $\boldsymbol{D}^{+}$ and $\boldsymbol{D}^{+ +}$ allow two and three further decompositions respectively, and $\boldsymbol{D}^{\#}$ allows two further decompositions with a role-reversal. In the sum corresponding to region $H$, we need to discard the whole of it because we cannot use Lemmas~\ref{l21}--\ref{l23} to give an asymptotic formula for the two-dimensional sum with $\alpha_2 \geqslant \frac{7}{2} \theta - \frac{3}{2}$ after a Buchstab iteration. We remark that $H$ is empty when $\theta > \frac{11}{21} \approx 0.5238$.

Next, we shall define some subregions of $A$ and $B$ defined in Section 5. The plot of these regions can also be found in Appendix 1.
\begin{align}
\nonumber A_1 =&\ \left\{\boldsymbol{\alpha}_{2}: \boldsymbol{\alpha}_{2} \in A,\ \alpha_2 < \min \left(\frac{3 \theta -1}{2}, \frac{1+\theta}{5}\right) \right\};\\
\nonumber A_2 =&\ \left\{\boldsymbol{\alpha}_{2}: \boldsymbol{\alpha}_{2} \in A,\ \alpha_2 \geqslant \min \left(\frac{3 \theta -1}{2}, \frac{1+\theta}{5}\right) \right\};\\
\nonumber B_1 =&\ \left\{\boldsymbol{\alpha}_{2}: \boldsymbol{\alpha}_{2} \in B,\ \alpha_2 < \min \left(\frac{3 \theta -1}{2}, \frac{1+\theta}{5}\right) \right\};\\
\nonumber B_2 =&\ \left\{\boldsymbol{\alpha}_{2}: \boldsymbol{\alpha}_{2} \in B,\ \alpha_2 \geqslant \min \left(\frac{3 \theta -1}{2}, \frac{1+\theta}{5}\right) \right\};\\
\nonumber A_1^{\prime} =&\ \left\{\boldsymbol{\alpha}_{2}: (1-\alpha_1 -\alpha_2, \alpha_2) \in B_1,\ (1-\alpha_1 -\alpha_2) \notin H \right\};\\
\nonumber A_2^{\prime} =&\ \left\{\boldsymbol{\alpha}_{2}: (1-\alpha_1 -\alpha_2, \alpha_2) \in B_2,\ (1-\alpha_1 -\alpha_2) \notin H \right\}.
\end{align}
Hence, by Buchstab's identity, we have
\begin{align}
\nonumber - \sum_{2}^{\prime} = - \sum_{\nu(0) \leqslant \alpha_{1}< \frac{1}{2}} S\left(\mathcal{A}_{p_{1}}, p_{1}\right) =&\ - \sum_{\substack{\nu(0) \leqslant \alpha_{1}< \frac{1}{2} \\ \boldsymbol{\alpha}_{1} \in H }} S\left(\mathcal{A}_{p_{1}}, p_{1}\right) - \sum_{\substack{\nu(0) \leqslant \alpha_{1}< \frac{1}{2} \\ \boldsymbol{\alpha}_{1} \notin H }} S\left(\mathcal{A}_{p_{1}}, p_{1}\right) \\
\nonumber =&\ - \sum_{\substack{\nu(0) \leqslant \alpha_{1}< \frac{1}{2} \\ \boldsymbol{\alpha}_{1} \in H }} S\left(\mathcal{A}_{p_{1}}, p_{1}\right) - \sum_{\substack{\nu(0) \leqslant \alpha_{1}< \frac{1}{2} \\ \boldsymbol{\alpha}_{1} \notin H }} S\left(\mathcal{A}_{p_{1}}, x^{\nu\left(\alpha_{1}\right)}\right) \\
\nonumber &+ \sum_{\substack{\nu(0) \leqslant \alpha_{1}<\frac{1}{2} \\ \boldsymbol{\alpha}_{1} \notin H \\
\nu\left(\alpha_{1}\right) \leqslant \alpha_{2}<\min \left(\alpha_{1},\frac{1}{2}\left(1-\alpha_{1}\right)\right)}} S\left(\mathcal{A}_{p_{1} p_{2}}, p_{2}\right)  \\
\nonumber =&\ - \sum_{\substack{\nu(0) \leqslant \alpha_{1}< \frac{1}{2} \\ \boldsymbol{\alpha}_{1} \in H }} S\left(\mathcal{A}_{p_{1}}, p_{1}\right) - \sum_{\substack{\nu(0) \leqslant \alpha_{1}< \frac{1}{2} \\ \boldsymbol{\alpha}_{1} \notin H }} S\left(\mathcal{A}_{p_{1}}, x^{\nu\left(\alpha_{1}\right)}\right) \\
\nonumber &+ \sum_{\substack{\boldsymbol{\alpha}_{1} \notin H \\ \boldsymbol{\alpha}_{2} \in A_1 }} S\left(\mathcal{A}_{p_{1} p_{2}}, p_{2}\right)
+ \sum_{\substack{\boldsymbol{\alpha}_{1} \notin H \\ \boldsymbol{\alpha}_{2} \in A_2 }} S\left(\mathcal{A}_{p_{1} p_{2}}, p_{2}\right)  \\
\nonumber &+ \sum_{\substack{\boldsymbol{\alpha}_{1} \notin H \\ \boldsymbol{\alpha}_{2} \in B_1 }} S\left(\mathcal{A}_{p_{1} p_{2}}, p_{2}\right)
+ \sum_{\substack{\boldsymbol{\alpha}_{1} \notin H \\ \boldsymbol{\alpha}_{2} \in B_2 }} S\left(\mathcal{A}_{p_{1} p_{2}}, p_{2}\right)
+ \sum_{\substack{\boldsymbol{\alpha}_{1} \notin H \\ \boldsymbol{\alpha}_{2} \in C }} S\left(\mathcal{A}_{p_{1} p_{2}}, p_{2}\right)  \\
=&\ - \sum_{H}^{\prime} - \sum_{3}^{\prime} + \sum_{A_1}^{\prime} + \sum_{A_2}^{\prime} + \sum_{B_1}^{\prime} + \sum_{B_2}^{\prime} + \sum_{C}^{\prime}.
\end{align}
By a similar discussion as in Section 5, we know that
\begin{equation}
\sum_{B_1}^{\prime} = \sum_{A_1^{\prime}}^{\prime} := \sum_{ \boldsymbol{\alpha}_{2} \in A_1^{\prime} } S\left(\mathcal{A}_{p_{1} p_{2}}, p_{2}\right)
\end{equation}
and
\begin{equation}
\sum_{B_2}^{\prime} = \sum_{A_2^{\prime}}^{\prime} := \sum_{ \boldsymbol{\alpha}_{2} \in A_2^{\prime} } S\left(\mathcal{A}_{p_{1} p_{2}}, p_{2}\right),
\end{equation}
hence
\begin{equation}
- \sum_{2}^{\prime} = - \sum_{H}^{\prime} - \sum_{3}^{\prime} + \sum_{A_1}^{\prime} + \sum_{A_2}^{\prime} + \sum_{A_1^{\prime}}^{\prime} + \sum_{A_2^{\prime}}^{\prime} + \sum_{C}^{\prime}.
\end{equation}
We have an asymptotic formula for $\sum_{3}^{\prime}$. For $\sum_{H}^{\prime}$ which cannot be decomposed anymore, we discard the whole of the sum leading to a loss of
\begin{equation}
\int_{\frac{7}{2} \theta - \frac{3}{2}}^{4 - 7 \theta} \frac{\omega\left(\frac{1 - t_1}{t_1}\right)}{t_1^2} d t_1 < 0.183.
\end{equation}

For $\sum_{A_1}^{\prime}$ we can use Buchstab's identity to reach
\begin{align}
\nonumber \sum_{A_1}^{\prime} =&\ \sum_{\substack{\boldsymbol{\alpha}_{1} \notin H \\ \boldsymbol{\alpha}_{2} \in A_1 }}  S\left(\mathcal{A}_{p_{1} p_{2}}, p_{2}\right) \\
=&\ \sum_{\substack{\boldsymbol{\alpha}_{1} \notin H \\ \boldsymbol{\alpha}_{2} \in A_1 }} S\left(\mathcal{A}_{p_{1} p_{2}}, x^{\nu_0}\right) - \sum_{\substack{\boldsymbol{\alpha}_{1} \notin H,\ \boldsymbol{\alpha}_{2} \in A_1 \\ \nu_0 \leqslant \alpha_{3} < \min\left(\alpha_2, \frac{1}{2}(1-\alpha_1 -\alpha_2)\right) }} S\left(\mathcal{A}_{p_{1} p_{2} p_{3}}, p_{3}\right).
\end{align}
By Lemma~\ref{l21}, we can give an asymptotic formula for the first sum on the right-hand side. For the second sum, we can perform a straightforward decomposition if we have $\boldsymbol{\alpha}_{3} \in \boldsymbol{D}^{+}$, and we can perform a role-reversal if we have $\boldsymbol{\alpha}_{3} \in \boldsymbol{D}^{\#}$. We can also use Buchstab's identity in reverse to gain some four-dimensional savings. Altogether, we have the following expression after this decomposition procedure:
\begin{align}
\nonumber &- \sum_{\substack{\boldsymbol{\alpha}_{1} \notin H,\ \boldsymbol{\alpha}_{2} \in A_1 \\ \nu_0 \leqslant \alpha_{3} < \min\left(\alpha_2, \frac{1}{2}(1-\alpha_1 -\alpha_2)\right) }} S\left(\mathcal{A}_{p_{1} p_{2} p_{3}}, p_{3}\right) \\
\nonumber =&\ - \sum_{\substack{\boldsymbol{\alpha}_{1} \notin H,\ \boldsymbol{\alpha}_{2} \in A_1 \\ \nu_0 \leqslant \alpha_{3} < \min\left(\alpha_2, \frac{1}{2}(1-\alpha_1 -\alpha_2)\right) \\ \boldsymbol{\alpha}_{3} \in \boldsymbol{G}_{3} }} S\left(\mathcal{A}_{p_{1} p_{2} p_{3}}, p_{3}\right) - \sum_{\substack{\boldsymbol{\alpha}_{1} \notin H,\ \boldsymbol{\alpha}_{2} \in A_1 \\ \nu_0 \leqslant \alpha_{3} < \min\left(\alpha_2, \frac{1}{2}(1-\alpha_1 -\alpha_2)\right) \\ \boldsymbol{\alpha}_{3} \notin \boldsymbol{G}_{3},\ \boldsymbol{\alpha}_{3} \notin \boldsymbol{D}^{+} }} S\left(\mathcal{A}_{p_{1} p_{2} p_{3}}, p_{3}\right) \\
\nonumber &- \sum_{\substack{\boldsymbol{\alpha}_{1} \notin H,\ \boldsymbol{\alpha}_{2} \in A_1 \\ \nu_0 \leqslant \alpha_{3} < \min\left(\alpha_2, \frac{1}{2}(1-\alpha_1 -\alpha_2)\right) \\ \boldsymbol{\alpha}_{3} \notin \boldsymbol{G}_{3},\ \boldsymbol{\alpha}_{3} \in \boldsymbol{D}^{+} }} S\left(\mathcal{A}_{p_{1} p_{2} p_{3}}, p_{3}\right) \\
\nonumber =&\ - \sum_{\substack{\boldsymbol{\alpha}_{1} \notin H,\ \boldsymbol{\alpha}_{2} \in A_1 \\ \nu_0 \leqslant \alpha_{3} < \min\left(\alpha_2, \frac{1}{2}(1-\alpha_1 -\alpha_2)\right) \\ \boldsymbol{\alpha}_{3} \in \boldsymbol{G}_{3} }} S\left(\mathcal{A}_{p_{1} p_{2} p_{3}}, p_{3}\right) 
- \sum_{\substack{\boldsymbol{\alpha}_{1} \notin H,\ \boldsymbol{\alpha}_{2} \in A_1 \\ \nu_0 \leqslant \alpha_{3} < \min\left(\alpha_2, \frac{1}{2}(1-\alpha_1 -\alpha_2)\right) \\ \boldsymbol{\alpha}_{3} \notin \boldsymbol{G}_{3},\ \boldsymbol{\alpha}_{3} \notin \boldsymbol{D}^{+} }} S\left(\mathcal{A}_{p_{1} p_{2} p_{3}}, p_{3}\right) \\
\nonumber &- \sum_{\substack{\boldsymbol{\alpha}_{1} \notin H,\ \boldsymbol{\alpha}_{2} \in A_1 \\ \nu_0 \leqslant \alpha_{3} < \min\left(\alpha_2, \frac{1}{2}(1-\alpha_1 -\alpha_2)\right) \\ \boldsymbol{\alpha}_{3} \notin \boldsymbol{G}_{3},\ \boldsymbol{\alpha}_{3} \in \boldsymbol{D}^{+} }} S\left(\mathcal{A}_{p_{1} p_{2} p_{3}}, x^{\nu_0}\right)
+ \sum_{\substack{\boldsymbol{\alpha}_{1} \notin H,\ \boldsymbol{\alpha}_{2} \in A_1 \\ \nu_0 \leqslant \alpha_{3} < \min\left(\alpha_2, \frac{1}{2}(1-\alpha_1 -\alpha_2)\right) \\ \boldsymbol{\alpha}_{3} \notin \boldsymbol{G}_{3},\ \boldsymbol{\alpha}_{3} \in \boldsymbol{D}^{+} \\ \nu_0 \leqslant \alpha_{4} < \min\left(\alpha_3, \frac{1}{2}(1-\alpha_1 -\alpha_2 -\alpha_3)\right) \\ \boldsymbol{\alpha}_{4} \in \boldsymbol{G}_{4} }} S\left(\mathcal{A}_{p_{1} p_{2} p_{3} p_{4}}, p_{4}\right) \\
\nonumber &+ \sum_{\substack{\boldsymbol{\alpha}_{1} \notin H,\ \boldsymbol{\alpha}_{2} \in A_1 \\ \nu_0 \leqslant \alpha_{3} < \min\left(\alpha_2, \frac{1}{2}(1-\alpha_1 -\alpha_2)\right) \\ \boldsymbol{\alpha}_{3} \notin \boldsymbol{G}_{3},\ \boldsymbol{\alpha}_{3} \in \boldsymbol{D}^{+} \\ \nu_0 \leqslant \alpha_{4} < \min\left(\alpha_3, \frac{1}{2}(1-\alpha_1 -\alpha_2 -\alpha_3)\right) \\ \boldsymbol{\alpha}_{4} \notin \boldsymbol{G}_{4} }} S\left(\mathcal{A}_{p_{1} p_{2} p_{3} p_{4}}, x^{\nu_0}\right) 
- \sum_{\substack{\boldsymbol{\alpha}_{1} \notin H,\ \boldsymbol{\alpha}_{2} \in A_1 \\ \nu_0 \leqslant \alpha_{3} < \min\left(\alpha_2, \frac{1}{2}(1-\alpha_1 -\alpha_2)\right) \\ \boldsymbol{\alpha}_{3} \notin \boldsymbol{G}_{3},\ \boldsymbol{\alpha}_{3} \in \boldsymbol{D}^{+} \\ \nu_0 \leqslant \alpha_{4} < \min\left(\alpha_3, \frac{1}{2}(1-\alpha_1 -\alpha_2 -\alpha_3)\right) \\ \boldsymbol{\alpha}_{4} \notin \boldsymbol{G}_{4} \\ \nu_0 \leqslant \alpha_{5} < \min\left(\alpha_4, \frac{1}{2}(1-\alpha_1 -\alpha_2 -\alpha_3 -\alpha_4)\right) \\ \boldsymbol{\alpha}_{5} \in \boldsymbol{G}_{5} }} S\left(\mathcal{A}_{p_{1} p_{2} p_{3} p_{4} p_{5}}, p_{5}\right) \\
\nonumber &- \sum_{\substack{\boldsymbol{\alpha}_{1} \notin H,\ \boldsymbol{\alpha}_{2} \in A_1 \\ \nu_0 \leqslant \alpha_{3} < \min\left(\alpha_2, \frac{1}{2}(1-\alpha_1 -\alpha_2)\right) \\ \boldsymbol{\alpha}_{3} \notin \boldsymbol{G}_{3},\ \boldsymbol{\alpha}_{3} \in \boldsymbol{D}^{+} \\ \nu_0 \leqslant \alpha_{4} < \min\left(\alpha_3, \frac{1}{2}(1-\alpha_1 -\alpha_2 -\alpha_3)\right) \\ \boldsymbol{\alpha}_{4} \notin \boldsymbol{G}_{4} \\ \nu_0 \leqslant \alpha_{5} < \min\left(\alpha_4, \frac{1}{2}(1-\alpha_1 -\alpha_2 -\alpha_3 -\alpha_4)\right) \\ \boldsymbol{\alpha}_{5} \notin \boldsymbol{G}_{5} }} S\left(\mathcal{A}_{p_{1} p_{2} p_{3} p_{4} p_{5}}, p_{5}\right) \\
=&\ - S_{01}^{\prime} - T_{01}^{\prime} - S_{02}^{\prime} + S_{03}^{\prime} + S_{04}^{\prime} - S_{05}^{\prime} - T_{02}^{\prime}.
\end{align}
We can give asymptotic formulas for $S_{01}^{\prime}$--$S_{05}^{\prime}$, and we can subtract the contribution of the sum
\begin{equation}
\sum_{\substack{\boldsymbol{\alpha}_{1} \notin H,\ \boldsymbol{\alpha}_{2} \in A_1 \\ \nu_0 \leqslant \alpha_{3} < \min\left(\alpha_2, \frac{1}{2}(1-\alpha_1 -\alpha_2)\right) \\ \boldsymbol{\alpha}_{3} \notin \boldsymbol{G}_{3},\ \boldsymbol{\alpha}_{3} \notin \boldsymbol{D}^{+} \\ \alpha_3 < \alpha_4 < \frac{1}{2}(1-\alpha_1 -\alpha_2 -\alpha_3) \\ \boldsymbol{\alpha}_{4} \in \boldsymbol{G}_{4}}} S\left(\mathcal{A}_{p_{1} p_{2} p_{3} p_{4}}, p_{4}\right)
\end{equation}
from the loss from $T_{01}^{\prime}$ by using Buchstab's identity in reverse.

To sum up, the loss from $\sum_{A_1}^{\prime}$ can be bounded by
\begin{align}
\nonumber & \left( \int_{\boldsymbol{t}_{3} \in V_{A1}} \frac{\omega \left(\frac{1 - t_1 - t_2 - t_3}{t_3}\right)}{t_1 t_2 t_3^2} d t_3 d t_2 d t_1 \right) \\ 
\nonumber -& \left( \int_{\boldsymbol{t}_{4} \in V_{A2}} \frac{\omega \left(\frac{1 - t_1 - t_2 - t_3 - t_4}{t_4}\right)}{t_1 t_2 t_3 t_4^2} d t_4 d t_3 d t_2 d t_1 \right) \\
\nonumber +& \left( \int_{\boldsymbol{t}_{5} \in V_{A3}} \frac{\omega \left(\frac{1 - t_1 - t_2 - t_3 - t_4 - t_5}{t_5}\right)}{t_1 t_2 t_3 t_4 t_5^2} d t_5 d t_4 d t_3 d t_2 d t_1 \right) \\
\nonumber \leqslant & \left( \int_{\boldsymbol{t}_{3} \in V_{A1}} \frac{\omega_1 \left(\frac{1 - t_1 - t_2 - t_3}{t_3}\right)}{t_1 t_2 t_3^2} d t_3 d t_2 d t_1 \right) \\ 
\nonumber -& \left( \int_{\boldsymbol{t}_{4} \in V_{A2}} \frac{\omega_0 \left(\frac{1 - t_1 - t_2 - t_3 - t_4}{t_4}\right)}{t_1 t_2 t_3 t_4^2} d t_4 d t_3 d t_2 d t_1 \right) \\
\nonumber +& \left( \int_{\boldsymbol{t}_{5} \in V_{A3}} \frac{\omega_1 \left(\frac{1 - t_1 - t_2 - t_3 - t_4 - t_5}{t_5}\right)}{t_1 t_2 t_3 t_4 t_5^2} d t_5 d t_4 d t_3 d t_2 d t_1 \right) \\
\leqslant &\ (0.19 - 0.005 + 0.07) < 0.255,
\end{align}
where
\begin{align}
\nonumber V_{A1}(\boldsymbol{\alpha}_{3}) :=&\ \left\{ \boldsymbol{\alpha}_{1} \notin H,\ \boldsymbol{\alpha}_{2} \in A_{1},\ \nu_0 \leqslant \alpha_{3} < \min\left(\alpha_2, \frac{1}{2}(1-\alpha_1 -\alpha_2)\right), \right. \\
\nonumber & \quad \boldsymbol{\alpha}_{3} \notin \boldsymbol{G}_{3},\ \boldsymbol{\alpha}_{3} \notin \boldsymbol{D}^{+}, \\
\nonumber & \left. \quad \nu_0 \leqslant \alpha_1 < \frac{1}{2},\ \nu\left(\alpha_1\right) \leqslant \alpha_2 < \min\left(\alpha_1, \frac{1}{2}(1-\alpha_1) \right) \right\}, \\
\nonumber V_{A2}(\boldsymbol{\alpha}_{4}) :=&\ \left\{ \boldsymbol{\alpha}_{1} \notin H,\ \boldsymbol{\alpha}_{2} \in A_{1},\ \nu_0 \leqslant \alpha_{3} < \min\left(\alpha_2, \frac{1}{2}(1-\alpha_1 -\alpha_2)\right), \right. \\
\nonumber & \quad \boldsymbol{\alpha}_{3} \notin \boldsymbol{G}_{3},\ \boldsymbol{\alpha}_{3} \notin \boldsymbol{D}^{+}, \\
\nonumber & \quad \alpha_3 < \alpha_4 < \frac{1}{2}(1-\alpha_1 -\alpha_2 -\alpha_3),\ \boldsymbol{\alpha}_{4} \in \boldsymbol{G}_{4}, \\
\nonumber & \left. \quad \nu_0 \leqslant \alpha_1 < \frac{1}{2},\ \nu\left(\alpha_1\right) \leqslant \alpha_2 < \min\left(\alpha_1, \frac{1}{2}(1-\alpha_1) \right) \right\}, \\
\nonumber V_{A3}(\boldsymbol{\alpha}_{5}) :=&\ \left\{ \boldsymbol{\alpha}_{1} \notin H,\ \boldsymbol{\alpha}_{2} \in A_{1},\ \nu_0 \leqslant \alpha_{3} < \min\left(\alpha_2, \frac{1}{2}(1-\alpha_1 -\alpha_2)\right), \right. \\
\nonumber & \quad \boldsymbol{\alpha}_{3} \notin \boldsymbol{G}_{3},\ \boldsymbol{\alpha}_{3} \in \boldsymbol{D}^{+}, \\
\nonumber & \quad \nu_0 \leqslant \alpha_{4} < \min\left(\alpha_3, \frac{1}{2}(1-\alpha_1 -\alpha_2 -\alpha_3)\right),\ \boldsymbol{\alpha}_{4} \notin \boldsymbol{G}_{4}, \\
\nonumber & \quad \nu_0 \leqslant \alpha_{5} < \min\left(\alpha_4, \frac{1}{2}(1-\alpha_1 -\alpha_2 -\alpha_3 -\alpha_4)\right),\ \boldsymbol{\alpha}_{5} \notin \boldsymbol{G}_{5}, \\
\nonumber & \left. \quad \nu_0 \leqslant \alpha_1 < \frac{1}{2},\ \nu\left(\alpha_1\right) \leqslant \alpha_2 < \min\left(\alpha_1, \frac{1}{2}(1-\alpha_1) \right) \right\}.
\end{align}

By the essentially identical decomposing process, the loss from $\sum_{A_1^{\prime}}^{\prime}$ is less than
\begin{align}
\nonumber & \left( \int_{\boldsymbol{t}_{3} \in V_{A4}} \frac{\omega \left(\frac{1 - t_1 - t_2 - t_3}{t_3}\right)}{t_1 t_2 t_3^2} d t_3 d t_2 d t_1 \right) \\ 
\nonumber -& \left( \int_{\boldsymbol{t}_{4} \in V_{A5}} \frac{\omega \left(\frac{1 - t_1 - t_2 - t_3 - t_4}{t_4}\right)}{t_1 t_2 t_3 t_4^2} d t_4 d t_3 d t_2 d t_1 \right) \\
\nonumber +& \left( \int_{\boldsymbol{t}_{5} \in V_{A6}} \frac{\omega \left(\frac{1 - t_1 - t_2 - t_3 - t_4 - t_5}{t_5}\right)}{t_1 t_2 t_3 t_4 t_5^2} d t_5 d t_4 d t_3 d t_2 d t_1 \right) \\
\nonumber \leqslant & \left( \int_{\boldsymbol{t}_{3} \in V_{A4}} \frac{\omega_1 \left(\frac{1 - t_1 - t_2 - t_3}{t_3}\right)}{t_1 t_2 t_3^2} d t_3 d t_2 d t_1 \right) \\ 
\nonumber -& \left( \int_{\boldsymbol{t}_{4} \in V_{A5}} \frac{\omega_0 \left(\frac{1 - t_1 - t_2 - t_3 - t_4}{t_4}\right)}{t_1 t_2 t_3 t_4^2} d t_4 d t_3 d t_2 d t_1 \right) \\
\nonumber +& \left( \int_{\boldsymbol{t}_{5} \in V_{A6}} \frac{\omega_1 \left(\frac{1 - t_1 - t_2 - t_3 - t_4 - t_5}{t_5}\right)}{t_1 t_2 t_3 t_4 t_5^2} d t_5 d t_4 d t_3 d t_2 d t_1 \right) \\
\leqslant &\ (0.32 - 0.01 + 0.07) < 0.38,
\end{align}
where
\begin{align}
\nonumber V_{A4}(\boldsymbol{\alpha}_{3}) :=&\ \left\{ \boldsymbol{\alpha}_{2} \in A_{1}^{\prime},\ \nu_0 \leqslant \alpha_{3} < \min\left(\alpha_2, \frac{1}{2}(1-\alpha_1 -\alpha_2)\right), \right. \\
\nonumber & \quad \boldsymbol{\alpha}_{3} \notin \boldsymbol{G}_{3},\ \boldsymbol{\alpha}_{3} \notin \boldsymbol{D}^{+}, \\
\nonumber & \left. \quad \nu_0 \leqslant \alpha_1 < \frac{1}{2},\ \nu\left(\alpha_1\right) \leqslant \alpha_2 < \min\left(\alpha_1, \frac{1}{2}(1-\alpha_1) \right) \right\}, \\
\nonumber V_{A5}(\boldsymbol{\alpha}_{4}) :=&\ \left\{\boldsymbol{\alpha}_{2} \in A_{1}^{\prime},\ \nu_0 \leqslant \alpha_{3} < \min\left(\alpha_2, \frac{1}{2}(1-\alpha_1 -\alpha_2)\right), \right. \\
\nonumber & \quad \boldsymbol{\alpha}_{3} \notin \boldsymbol{G}_{3},\ \boldsymbol{\alpha}_{3} \notin \boldsymbol{D}^{+}, \\
\nonumber & \quad \alpha_3 < \alpha_4 < \frac{1}{2}(1-\alpha_1 -\alpha_2 -\alpha_3),\ \boldsymbol{\alpha}_{4} \in \boldsymbol{G}_{4}, \\
\nonumber & \left. \quad \nu_0 \leqslant \alpha_1 < \frac{1}{2},\ \nu\left(\alpha_1\right) \leqslant \alpha_2 < \min\left(\alpha_1, \frac{1}{2}(1-\alpha_1) \right) \right\}, \\
\nonumber V_{A6}(\boldsymbol{\alpha}_{5}) :=&\ \left\{ \boldsymbol{\alpha}_{2} \in A_{1}^{\prime},\ \nu_0 \leqslant \alpha_{3} < \min\left(\alpha_2, \frac{1}{2}(1-\alpha_1 -\alpha_2)\right), \right. \\
\nonumber & \quad \boldsymbol{\alpha}_{3} \notin \boldsymbol{G}_{3},\ \boldsymbol{\alpha}_{3} \in \boldsymbol{D}^{+}, \\
\nonumber & \quad \nu_0 \leqslant \alpha_{4} < \min\left(\alpha_3, \frac{1}{2}(1-\alpha_1 -\alpha_2 -\alpha_3)\right),\ \boldsymbol{\alpha}_{4} \notin \boldsymbol{G}_{4}, \\
\nonumber & \quad \nu_0 \leqslant \alpha_{5} < \min\left(\alpha_4, \frac{1}{2}(1-\alpha_1 -\alpha_2 -\alpha_3 -\alpha_4)\right),\ \boldsymbol{\alpha}_{5} \notin \boldsymbol{G}_{5}, \\
\nonumber & \left. \quad \nu_0 \leqslant \alpha_1 < \frac{1}{2},\ \nu\left(\alpha_1\right) \leqslant \alpha_2 < \min\left(\alpha_1, \frac{1}{2}(1-\alpha_1) \right) \right\}.
\end{align}

For $\sum_{A_2}^{\prime}$, we apply Buchstab's identity as in (37) to get
\begin{align}
\nonumber \sum_{A_2}^{\prime} =&\ \sum_{\substack{\boldsymbol{\alpha}_{1} \notin H \\ \boldsymbol{\alpha}_{2} \in A_2 }} S\left(\mathcal{A}_{p_{1} p_{2}}, p_{2}\right) \\
=&\ \sum_{\substack{\boldsymbol{\alpha}_{1} \notin H \\ \boldsymbol{\alpha}_{2} \in A_2 }} S\left(\mathcal{A}_{p_{1} p_{2}}, x^{\nu_0}\right) - \sum_{\substack{\boldsymbol{\alpha}_{1} \notin H,\ \boldsymbol{\alpha}_{2} \in A_2 \\ \nu_0 \leqslant \alpha_{3} < \min\left(\alpha_2, \frac{1}{2}(1-\alpha_1 -\alpha_2)\right) }} S\left(\mathcal{A}_{p_{1} p_{2} p_{3}}, p_{3}\right).
\end{align}
Although Lemma~\ref{l21} is not applicable in this case, we can use Lemma~\ref{l23} to give an asymptotic formula for the first sum on the right-hand side. For the second sum, we cannot perform any further decompositions because we cannot give asymptotic formula for the four-dimensional sum after applying Buchstab's identity twice. Thus, the loss from $\sum_{A_2}^{\prime}$ is just
\begin{align}
\nonumber & \int_{\boldsymbol{t}_{3} \in V_{A7}} \frac{\omega \left(\frac{1 - t_1 - t_2 - t_3}{t_3}\right)}{t_1 t_2 t_3^2} d t_3 d t_2 d t_1 \\
\leqslant & \int_{\boldsymbol{t}_{3} \in V_{A7}} \frac{\omega_1 \left(\frac{1 - t_1 - t_2 - t_3}{t_3}\right)}{t_1 t_2 t_3^2} d t_3 d t_2 d t_1 < 0.12,
\end{align}
where
\begin{align}
\nonumber V_{A7}(\boldsymbol{\alpha}_{3}) :=&\ \left\{ \boldsymbol{\alpha}_{1} \notin H,\ \boldsymbol{\alpha}_{2} \in A_{2},\ \nu_0 \leqslant \alpha_{3} < \min\left(\alpha_2, \frac{1}{2}(1-\alpha_1 -\alpha_2)\right),\ \boldsymbol{\alpha}_{3} \notin \boldsymbol{G}_{3}, \right. \\
\nonumber & \left. \quad \nu_0 \leqslant \alpha_1 < \frac{1}{2},\ \nu\left(\alpha_1\right) \leqslant \alpha_2 < \min\left(\alpha_1, \frac{1}{2}(1-\alpha_1) \right) \right\}.
\end{align}
Similarly, the loss from $\sum_{A_2^{\prime}}^{\prime}$ is
\begin{align}
\nonumber & \int_{\boldsymbol{t}_{3} \in V_{A8}} \frac{\omega \left(\frac{1 - t_1 - t_2 - t_3}{t_3}\right)}{t_1 t_2 t_3^2} d t_3 d t_2 d t_1 \\
\leqslant & \int_{\boldsymbol{t}_{3} \in V_{A8}} \frac{\omega_1 \left(\frac{1 - t_1 - t_2 - t_3}{t_3}\right)}{t_1 t_2 t_3^2} d t_3 d t_2 d t_1 < 0.22,
\end{align}
where
\begin{align}
\nonumber V_{A8}(\boldsymbol{\alpha}_{3}) :=&\ \left\{ \boldsymbol{\alpha}_{2} \in A_{2}^{\prime},\ \nu_0 \leqslant \alpha_{3} < \min\left(\alpha_2, \frac{1}{2}(1-\alpha_1 -\alpha_2)\right),\ \boldsymbol{\alpha}_{3} \notin \boldsymbol{G}_{3}, \right. \\
\nonumber & \left. \quad \nu_0 \leqslant \alpha_1 < \frac{1}{2},\ \nu\left(\alpha_1\right) \leqslant \alpha_2 < \min\left(\alpha_1, \frac{1}{2}(1-\alpha_1) \right) \right\}.
\end{align}

For the remaining $\sum_{C}^{\prime}$, we have
\begin{align}
\nonumber \sum_{C}^{\prime} =&\ \sum_{\substack{\boldsymbol{\alpha}_{1} \notin H \\ \boldsymbol{\alpha}_{2} \in C }}  S\left(\mathcal{A}_{p_{1} p_{2}}, p_{2}\right) \\
=&\ \sum_{\substack{\boldsymbol{\alpha}_{1} \notin H \\ \boldsymbol{\alpha}_{2} \in C }} S\left(\mathcal{A}_{p_{1} p_{2}}, x^{\nu_0}\right) - \sum_{\substack{\boldsymbol{\alpha}_{1} \notin H,\ \boldsymbol{\alpha}_{2} \in C \\ \nu_0 \leqslant \alpha_{3} < \min\left(\alpha_2, \frac{1}{2}(1-\alpha_1 -\alpha_2)\right) }} S\left(\mathcal{A}_{p_{1} p_{2} p_{3}}, p_{3}\right).
\end{align}
We can give an asymptotic formula for the first sum on the right-hand side. For the second sum, we need to consider role-reversals because we may have a large $\alpha_1$ in this case. We can perform a straightforward decomposition if we have $\boldsymbol{\alpha}_{3} \in \boldsymbol{D}^{+}$, and we can perform a role-reversal if we have $\boldsymbol{\alpha}_{3} \in \boldsymbol{D}^{\#}$. Using Buchstab's identity, we write
\begin{align}
\nonumber &- \sum_{\substack{\boldsymbol{\alpha}_{1} \notin H,\ \boldsymbol{\alpha}_{2} \in C \\ \nu_0 \leqslant \alpha_{3} < \min\left(\alpha_2, \frac{1}{2}(1-\alpha_1 -\alpha_2)\right) }} S\left(\mathcal{A}_{p_{1} p_{2} p_{3}}, p_{3}\right) \\
\nonumber =&\ - \sum_{\substack{\boldsymbol{\alpha}_{1} \notin H,\ \boldsymbol{\alpha}_{2} \in C \\ \nu_0 \leqslant \alpha_{3} < \min\left(\alpha_2, \frac{1}{2}(1-\alpha_1 -\alpha_2)\right) \\ \boldsymbol{\alpha}_{3} \in \boldsymbol{G}_{3} }} S\left(\mathcal{A}_{p_{1} p_{2} p_{3}}, p_{3}\right) - \sum_{\substack{\boldsymbol{\alpha}_{1} \notin H,\ \boldsymbol{\alpha}_{2} \in C \\ \nu_0 \leqslant \alpha_{3} < \min\left(\alpha_2, \frac{1}{2}(1-\alpha_1 -\alpha_2)\right) \\ \boldsymbol{\alpha}_{3} \notin \boldsymbol{G}_{3},\ \boldsymbol{\alpha}_{3} \notin \boldsymbol{D}^{+},\ \boldsymbol{\alpha}_{3} \notin \boldsymbol{D}^{\#} }} S\left(\mathcal{A}_{p_{1} p_{2} p_{3}}, p_{3}\right) \\
\nonumber &- \sum_{\substack{\boldsymbol{\alpha}_{1} \notin H,\ \boldsymbol{\alpha}_{2} \in C \\ \nu_0 \leqslant \alpha_{3} < \min\left(\alpha_2, \frac{1}{2}(1-\alpha_1 -\alpha_2)\right) \\ \boldsymbol{\alpha}_{3} \notin \boldsymbol{G}_{3},\ \boldsymbol{\alpha}_{3} \in \boldsymbol{D}^{+} }} S\left(\mathcal{A}_{p_{1} p_{2} p_{3}}, p_{3}\right) - \sum_{\substack{\boldsymbol{\alpha}_{1} \notin H,\ \boldsymbol{\alpha}_{2} \in C \\ \nu_0 \leqslant \alpha_{3} < \min\left(\alpha_2, \frac{1}{2}(1-\alpha_1 -\alpha_2)\right) \\ \boldsymbol{\alpha}_{3} \notin \boldsymbol{G}_{3},\ \boldsymbol{\alpha}_{3} \notin \boldsymbol{D}^{+},\ \boldsymbol{\alpha}_{3} \in \boldsymbol{D}^{\#} }} S\left(\mathcal{A}_{p_{1} p_{2} p_{3}}, p_{3}\right) \\
\nonumber =&\ - \sum_{\substack{\boldsymbol{\alpha}_{1} \notin H,\ \boldsymbol{\alpha}_{2} \in C \\ \nu_0 \leqslant \alpha_{3} < \min\left(\alpha_2, \frac{1}{2}(1-\alpha_1 -\alpha_2)\right) \\ \boldsymbol{\alpha}_{3} \in \boldsymbol{G}_{3} }} S\left(\mathcal{A}_{p_{1} p_{2} p_{3}}, p_{3}\right) - \sum_{\substack{\boldsymbol{\alpha}_{1} \notin H,\ \boldsymbol{\alpha}_{2} \in C \\ \nu_0 \leqslant \alpha_{3} < \min\left(\alpha_2, \frac{1}{2}(1-\alpha_1 -\alpha_2)\right) \\ \boldsymbol{\alpha}_{3} \notin \boldsymbol{G}_{3},\ \boldsymbol{\alpha}_{3} \notin \boldsymbol{D}^{+},\ \boldsymbol{\alpha}_{3} \notin \boldsymbol{D}^{\#} }} S\left(\mathcal{A}_{p_{1} p_{2} p_{3}}, p_{3}\right) \\
\nonumber &- \sum_{\substack{\boldsymbol{\alpha}_{1} \notin H,\ \boldsymbol{\alpha}_{2} \in C \\ \nu_0 \leqslant \alpha_{3} < \min\left(\alpha_2, \frac{1}{2}(1-\alpha_1 -\alpha_2)\right) \\ \boldsymbol{\alpha}_{3} \notin \boldsymbol{G}_{3},\ \boldsymbol{\alpha}_{3} \in \boldsymbol{D}^{+} }} S\left(\mathcal{A}_{p_{1} p_{2} p_{3}}, x^{\nu_0}\right)
+ \sum_{\substack{\boldsymbol{\alpha}_{1} \notin H,\ \boldsymbol{\alpha}_{2} \in C \\ \nu_0 \leqslant \alpha_{3} < \min\left(\alpha_2, \frac{1}{2}(1-\alpha_1 -\alpha_2)\right) \\ \boldsymbol{\alpha}_{3} \notin \boldsymbol{G}_{3},\ \boldsymbol{\alpha}_{3} \in \boldsymbol{D}^{+} \\ \nu_0 \leqslant \alpha_{4} < \min\left(\alpha_3, \frac{1}{2}(1-\alpha_1 -\alpha_2 -\alpha_3)\right) \\ \boldsymbol{\alpha}_{4} \in \boldsymbol{G}_{4} }} S\left(\mathcal{A}_{p_{1} p_{2} p_{3} p_{4}}, p_{4}\right) \\
\nonumber &+ \sum_{\substack{\boldsymbol{\alpha}_{1} \notin H,\ \boldsymbol{\alpha}_{2} \in C \\ \nu_0 \leqslant \alpha_{3} < \min\left(\alpha_2, \frac{1}{2}(1-\alpha_1 -\alpha_2)\right) \\ \boldsymbol{\alpha}_{3} \notin \boldsymbol{G}_{3},\ \boldsymbol{\alpha}_{3} \in \boldsymbol{D}^{+} \\ \nu_0 \leqslant \alpha_{4} < \min\left(\alpha_3, \frac{1}{2}(1-\alpha_1 -\alpha_2 -\alpha_3)\right) \\ \boldsymbol{\alpha}_{4} \notin \boldsymbol{G}_{4} }} S\left(\mathcal{A}_{p_{1} p_{2} p_{3} p_{4}}, x^{\nu_0}\right) 
- \sum_{\substack{\boldsymbol{\alpha}_{1} \notin H,\ \boldsymbol{\alpha}_{2} \in C \\ \nu_0 \leqslant \alpha_{3} < \min\left(\alpha_2, \frac{1}{2}(1-\alpha_1 -\alpha_2)\right) \\ \boldsymbol{\alpha}_{3} \notin \boldsymbol{G}_{3},\ \boldsymbol{\alpha}_{3} \in \boldsymbol{D}^{+} \\ \nu_0 \leqslant \alpha_{4} < \min\left(\alpha_3, \frac{1}{2}(1-\alpha_1 -\alpha_2 -\alpha_3)\right) \\ \boldsymbol{\alpha}_{4} \notin \boldsymbol{G}_{4} \\ \nu_0 \leqslant \alpha_{5} < \min\left(\alpha_4, \frac{1}{2}(1-\alpha_1 -\alpha_2 -\alpha_3 -\alpha_4)\right) \\ \boldsymbol{\alpha}_{5} \in \boldsymbol{G}_{5} }} S\left(\mathcal{A}_{p_{1} p_{2} p_{3} p_{4} p_{5}}, p_{5}\right) \\
\nonumber &- \sum_{\substack{\boldsymbol{\alpha}_{1} \notin H,\ \boldsymbol{\alpha}_{2} \in C \\ \nu_0 \leqslant \alpha_{3} < \min\left(\alpha_2, \frac{1}{2}(1-\alpha_1 -\alpha_2)\right) \\ \boldsymbol{\alpha}_{3} \notin \boldsymbol{G}_{3},\ \boldsymbol{\alpha}_{3} \in \boldsymbol{D}^{+} \\ \nu_0 \leqslant \alpha_{4} < \min\left(\alpha_3, \frac{1}{2}(1-\alpha_1 -\alpha_2 -\alpha_3)\right) \\ \boldsymbol{\alpha}_{4} \notin \boldsymbol{G}_{4} \\ \nu_0 \leqslant \alpha_{5} < \min\left(\alpha_4, \frac{1}{2}(1-\alpha_1 -\alpha_2 -\alpha_3 -\alpha_4)\right) \\ \boldsymbol{\alpha}_{5} \notin \boldsymbol{G}_{5} }} S\left(\mathcal{A}_{p_{1} p_{2} p_{3} p_{4} p_{5}}, p_{5}\right)
- \sum_{\substack{\boldsymbol{\alpha}_{1} \notin H,\ \boldsymbol{\alpha}_{2} \in C \\ \nu_0 \leqslant \alpha_{3} < \min\left(\alpha_2, \frac{1}{2}(1-\alpha_1 -\alpha_2)\right) \\ \boldsymbol{\alpha}_{3} \notin \boldsymbol{G}_{3},\ \boldsymbol{\alpha}_{3} \notin \boldsymbol{D}^{+},\ \boldsymbol{\alpha}_{3} \in \boldsymbol{D}^{\#} }} S\left(\mathcal{A}_{p_{1} p_{2} p_{3}}, x^{\nu_0}\right) \\
\nonumber &+ \sum_{\substack{\boldsymbol{\alpha}_{1} \notin H,\ \boldsymbol{\alpha}_{2} \in C \\ \nu_0 \leqslant \alpha_{3} < \min\left(\alpha_2, \frac{1}{2}(1-\alpha_1 -\alpha_2)\right) \\ \boldsymbol{\alpha}_{3} \notin \boldsymbol{G}_{3},\ \boldsymbol{\alpha}_{3} \notin \boldsymbol{D}^{+},\ \boldsymbol{\alpha}_{3} \in \boldsymbol{D}^{\#} \\ \nu_0 \leqslant \alpha_{4} < \min\left(\alpha_3, \frac{1}{2}(1-\alpha_1 -\alpha_2 -\alpha_3)\right) \\ \boldsymbol{\alpha}_{4} \in \boldsymbol{G}_{4} }} S\left(\mathcal{A}_{p_{1} p_{2} p_{3} p_{4}}, p_{4}\right) 
+ \sum_{\substack{\boldsymbol{\alpha}_{1} \notin H,\ \boldsymbol{\alpha}_{2} \in C \\ \nu_0 \leqslant \alpha_{3} < \min\left(\alpha_2, \frac{1}{2}(1-\alpha_1 -\alpha_2)\right) \\ \boldsymbol{\alpha}_{3} \notin \boldsymbol{G}_{3},\ \boldsymbol{\alpha}_{3} \notin \boldsymbol{D}^{+},\ \boldsymbol{\alpha}_{3} \in \boldsymbol{D}^{\#} \\ \nu_0 \leqslant \alpha_{4} < \min\left(\alpha_3, \frac{1}{2}(1-\alpha_1 -\alpha_2 -\alpha_3)\right) \\ \boldsymbol{\alpha}_{4} \notin \boldsymbol{G}_{4} }} S\left(\mathcal{A}_{\eta p_{2} p_{3} p_{4}}, x^{\nu_0}\right) \\
\nonumber &- \sum_{\substack{\boldsymbol{\alpha}_{1} \notin H,\ \boldsymbol{\alpha}_{2} \in C \\ \nu_0 \leqslant \alpha_{3} < \min\left(\alpha_2, \frac{1}{2}(1-\alpha_1 -\alpha_2)\right) \\ \boldsymbol{\alpha}_{3} \notin \boldsymbol{G}_{3},\ \boldsymbol{\alpha}_{3} \notin \boldsymbol{D}^{+},\ \boldsymbol{\alpha}_{3} \in \boldsymbol{D}^{\#} \\ \nu_0 \leqslant \alpha_{4} < \min\left(\alpha_3, \frac{1}{2}(1-\alpha_1 -\alpha_2 -\alpha_3)\right) \\ \boldsymbol{\alpha}_{4} \notin \boldsymbol{G}_{4} \\ \nu_0 \leqslant \alpha_{5} < \frac{1}{2}\alpha_1 \\ \boldsymbol{\alpha}_{5}^{\#} \in \boldsymbol{G}_{5}}} S\left(\mathcal{A}_{\eta p_{2} p_{3} p_{4} p_{5}}, p_{5}\right)
- \sum_{\substack{\boldsymbol{\alpha}_{1} \notin H,\ \boldsymbol{\alpha}_{2} \in C \\ \nu_0 \leqslant \alpha_{3} < \min\left(\alpha_2, \frac{1}{2}(1-\alpha_1 -\alpha_2)\right) \\ \boldsymbol{\alpha}_{3} \notin \boldsymbol{G}_{3},\ \boldsymbol{\alpha}_{3} \notin \boldsymbol{D}^{+},\ \boldsymbol{\alpha}_{3} \in \boldsymbol{D}^{\#} \\ \nu_0 \leqslant \alpha_{4} < \min\left(\alpha_3, \frac{1}{2}(1-\alpha_1 -\alpha_2 -\alpha_3)\right) \\ \boldsymbol{\alpha}_{4} \notin \boldsymbol{G}_{4} \\ \nu_0 \leqslant \alpha_{5} < \frac{1}{2}\alpha_1 \\ \boldsymbol{\alpha}_{5}^{\#} \notin \boldsymbol{G}_{5}}} S\left(\mathcal{A}_{\eta p_{2} p_{3} p_{4} p_{5}}, p_{5}\right) \\
=&\ - S_{06}^{\prime} - T_{03}^{\prime} - S_{07}^{\prime} + S_{08}^{\prime} + S_{09}^{\prime} - S_{10}^{\prime} - T_{04}^{\prime} - S_{11}^{\prime} + S_{12}^{\prime} + S_{13}^{\prime} - S_{14}^{\prime} - T_{05}^{\prime},
\end{align}
where $\eta \sim x^{1-\alpha_1 -\alpha_2 -\alpha_3 -\alpha_4}$, $\left(\eta, P(p_{4})\right)=1$ and
$$
\boldsymbol{\alpha}_{5}^{\#} =(1-\alpha_1-\alpha_2-\alpha_3-\alpha_4,\ \alpha_2,\ \alpha_3,\ \alpha_4,\ \alpha_5).
$$
We can give asymptotic formulas for $S_{06}^{\prime}$--$S_{14}^{\prime}$. For $T_{03}^{\prime}$ we can use Buchstab's identity in reverse to subtract the contribution of the sum
\begin{equation}
\sum_{\substack{\boldsymbol{\alpha}_{1} \notin H,\ \boldsymbol{\alpha}_{2} \in C \\ \nu_0 \leqslant \alpha_{3} < \min\left(\alpha_2, \frac{1}{2}(1-\alpha_1 -\alpha_2)\right) \\ \boldsymbol{\alpha}_{3} \notin \boldsymbol{G}_{3},\ \boldsymbol{\alpha}_{3} \notin \boldsymbol{D}^{+} \\ \alpha_3 < \alpha_4 < \frac{1}{2}(1-\alpha_1 -\alpha_2 -\alpha_3) \\ \boldsymbol{\alpha}_{4} \in \boldsymbol{G}_{4}}} S\left(\mathcal{A}_{p_{1} p_{2} p_{3} p_{4}}, p_{4}\right)
\end{equation}
from the loss. For the remaining
\begin{equation}
T_{04}^{\prime} = \sum_{\substack{\boldsymbol{\alpha}_{1} \notin H,\ \boldsymbol{\alpha}_{2} \in C \\ \nu_0 \leqslant \alpha_{3} < \min\left(\alpha_2, \frac{1}{2}(1-\alpha_1 -\alpha_2)\right) \\ \boldsymbol{\alpha}_{3} \notin \boldsymbol{G}_{3},\ \boldsymbol{\alpha}_{3} \in \boldsymbol{D}^{+} \\ \nu_0 \leqslant \alpha_{4} < \min\left(\alpha_3, \frac{1}{2}(1-\alpha_1 -\alpha_2 -\alpha_3)\right) \\ \boldsymbol{\alpha}_{4} \notin \boldsymbol{G}_{4} \\ \nu_0 \leqslant \alpha_{5} < \min\left(\alpha_4, \frac{1}{2}(1-\alpha_1 -\alpha_2 -\alpha_3 -\alpha_4)\right) \\ \boldsymbol{\alpha}_{5} \notin \boldsymbol{G}_{5} }} S\left(\mathcal{A}_{p_{1} p_{2} p_{3} p_{4} p_{5}}, p_{5}\right)
\end{equation}
and
\begin{equation}
T_{05}^{\prime} = \sum_{\substack{\boldsymbol{\alpha}_{1} \notin H,\ \boldsymbol{\alpha}_{2} \in C \\ \nu_0 \leqslant \alpha_{3} < \min\left(\alpha_2, \frac{1}{2}(1-\alpha_1 -\alpha_2)\right) \\ \boldsymbol{\alpha}_{3} \notin \boldsymbol{G}_{3},\ \boldsymbol{\alpha}_{3} \notin \boldsymbol{D}^{+},\ \boldsymbol{\alpha}_{3} \in \boldsymbol{D}^{\#} \\ \nu_0 \leqslant \alpha_{4} < \min\left(\alpha_3, \frac{1}{2}(1-\alpha_1 -\alpha_2 -\alpha_3)\right) \\ \boldsymbol{\alpha}_{4} \notin \boldsymbol{G}_{4} \\ \nu_0 \leqslant \alpha_{5} < \frac{1}{2}\alpha_1 \\ \boldsymbol{\alpha}_{5}^{\#} \notin \boldsymbol{G}_{5}}} S\left(\mathcal{A}_{\eta p_{2} p_{3} p_{4} p_{5}}, p_{5}\right),
\end{equation}
we can perform a further straightforward decomposition on $T_{04}^{\prime}$ if $\boldsymbol{\alpha}_{5} \in \boldsymbol{D}^{++}$,  and on $T_{05}^{\prime}$ if either $\boldsymbol{\alpha}_{5}^{\#} \in \boldsymbol{D}^{++}$ or $(\alpha_1-\alpha_5,\ \alpha_2,\ \alpha_3,\ \alpha_5,\ \alpha_4) \in \boldsymbol{D}^{++}$. Note that for $T_{032}$ and $T_{033}$ in Section 5 we use similar discussion. This leads to the loss of three seven-dimensional sums
\begin{equation}
\sum_{\substack{\boldsymbol{\alpha}_{1} \notin H,\ \boldsymbol{\alpha}_{2} \in C \\ \nu_0 \leqslant \alpha_{3} < \min\left(\alpha_2, \frac{1}{2}(1-\alpha_1 -\alpha_2)\right) \\ \boldsymbol{\alpha}_{3} \notin \boldsymbol{G}_{3},\ \boldsymbol{\alpha}_{3} \in \boldsymbol{D}^{+} \\ \nu_0 \leqslant \alpha_{4} < \min\left(\alpha_3, \frac{1}{2}(1-\alpha_1 -\alpha_2 -\alpha_3)\right) \\ \boldsymbol{\alpha}_{4} \notin \boldsymbol{G}_{4} \\ \nu_0 \leqslant \alpha_{5} < \min\left(\alpha_4, \frac{1}{2}(1-\alpha_1 -\alpha_2 -\alpha_3 -\alpha_4)\right) \\ \boldsymbol{\alpha}_{5} \notin \boldsymbol{G}_{5},\ \boldsymbol{\alpha}_{5} \in \boldsymbol{D}^{++} \\ \nu_0 \leqslant \alpha_{6} < \min\left(\alpha_5, \frac{1}{2}(1-\alpha_1 -\alpha_2 -\alpha_3 -\alpha_4 -\alpha_5)\right) \\ \boldsymbol{\alpha}_{6} \notin \boldsymbol{G}_{6} \\ \nu_0 \leqslant \alpha_{7} < \min\left(\alpha_6, \frac{1}{2}(1-\alpha_1 -\alpha_2 -\alpha_3 -\alpha_4 -\alpha_5 -\alpha_6)\right) \\ \boldsymbol{\alpha}_{7} \notin \boldsymbol{G}_{7} }} S\left(\mathcal{A}_{p_{1} p_{2} p_{3} p_{4} p_{5} p_{6} p_{7}}, p_{7}\right),
\end{equation}
\begin{equation}
\sum_{\substack{\boldsymbol{\alpha}_{1} \notin H,\ \boldsymbol{\alpha}_{2} \in C \\ \nu_0 \leqslant \alpha_{3} < \min\left(\alpha_2, \frac{1}{2}(1-\alpha_1 -\alpha_2)\right) \\ \boldsymbol{\alpha}_{3} \notin \boldsymbol{G}_{3},\ \boldsymbol{\alpha}_{3} \notin \boldsymbol{D}^{+},\ \boldsymbol{\alpha}_{3} \in \boldsymbol{D}^{\#} \\ \nu_0 \leqslant \alpha_{4} < \min\left(\alpha_3, \frac{1}{2}(1-\alpha_1 -\alpha_2 -\alpha_3)\right) \\ \boldsymbol{\alpha}_{4} \notin \boldsymbol{G}_{4} \\ \nu_0 \leqslant \alpha_{5} < \frac{1}{2}\alpha_1 \\ \boldsymbol{\alpha}_{5}^{\#} \notin \boldsymbol{G}_{5},\ \boldsymbol{\alpha}_{5}^{\#} \in \boldsymbol{D}^{++} \\ \nu_0 \leqslant \alpha_{6} < \min\left(\alpha_5, \frac{1}{2}(\alpha_1 -\alpha_5)\right) \\ \boldsymbol{\alpha}_{6}^{\#} \notin \boldsymbol{G}_{6} \\ \nu_0 \leqslant \alpha_{7} < \min\left(\alpha_6, \frac{1}{2}(\alpha_1 -\alpha_5 -\alpha_6)\right) \\ \boldsymbol{\alpha}_{7}^{\#} \notin \boldsymbol{G}_{7}  }} S\left(\mathcal{A}_{\eta p_{2} p_{3} p_{4} p_{5} p_{6} p_{7}}, p_{7}\right),
\end{equation}
and
\begin{equation}
\sum_{\substack{\boldsymbol{\alpha}_{1} \notin H,\ \boldsymbol{\alpha}_{2} \in C \\ \nu_0 \leqslant \alpha_{3} < \min\left(\alpha_2, \frac{1}{2}(1-\alpha_1 -\alpha_2)\right) \\ \boldsymbol{\alpha}_{3} \notin \boldsymbol{G}_{3},\ \boldsymbol{\alpha}_{3} \notin \boldsymbol{D}^{+},\ \boldsymbol{\alpha}_{3} \in \boldsymbol{D}^{\#} \\ \nu_0 \leqslant \alpha_{4} < \min\left(\alpha_3, \frac{1}{2}(1-\alpha_1 -\alpha_2 -\alpha_3)\right) \\ \boldsymbol{\alpha}_{4} \notin \boldsymbol{G}_{4} \\ \nu_0 \leqslant \alpha_{5} < \frac{1}{2}\alpha_1 \\ \boldsymbol{\alpha}_{5}^{\#} \notin \boldsymbol{G}_{5},\ \boldsymbol{\alpha}_{5}^{\#} \notin \boldsymbol{D}^{++},\ \boldsymbol{\alpha}_{5}^{\# \prime} \in \boldsymbol{D}^{++} \\ \nu_0 \leqslant \alpha_{6} < \min\left(\alpha_4, \frac{1}{2}(1-\alpha_1-\alpha_2-\alpha_3-\alpha_4)\right) \\ \boldsymbol{\alpha}_{6}^{\# \prime} \notin \boldsymbol{G}_{6} \\ \nu_0 \leqslant \alpha_{7} < \min\left(\alpha_6, \frac{1}{2}(1-\alpha_1-\alpha_2-\alpha_3-\alpha_4 -\alpha_6)\right) \\ \boldsymbol{\alpha}_{7}^{\# \prime} \notin \boldsymbol{G}_{7}  }} S\left(\mathcal{A}_{\eta_1 p_{2} p_{3} p_{4} p_{5} p_{6} p_{7}}, p_{7}\right),
\end{equation}
where 
$\eta_1 \sim x^{\alpha_1-\alpha_5}$, $\left(\eta_1, P(p_{5})\right)=1$,
$$
\boldsymbol{\alpha}_{6}^{\#} =(1-\alpha_1-\alpha_2-\alpha_3-\alpha_4,\ \alpha_2,\ \alpha_3,\ \alpha_4,\ \alpha_5,\ \alpha_6),
$$
$$
\boldsymbol{\alpha}_{7}^{\#} =(1-\alpha_1-\alpha_2-\alpha_3-\alpha_4,\ \alpha_2,\ \alpha_3,\ \alpha_4,\ \alpha_5,\ \alpha_6,\ \alpha_7),
$$
$$
\boldsymbol{\alpha}_{5}^{\# \prime} =(\alpha_1-\alpha_5,\ \alpha_2,\ \alpha_3,\ \alpha_5,\ \alpha_4),
$$
$$
\boldsymbol{\alpha}_{6}^{\# \prime} =(\alpha_1-\alpha_5,\ \alpha_2,\ \alpha_3,\ \alpha_4,\ \alpha_5,\ \alpha_6)
$$
and 
$$
\boldsymbol{\alpha}_{7}^{\# \prime} =(\alpha_1-\alpha_5,\ \alpha_2,\ \alpha_3,\ \alpha_4,\ \alpha_5,\ \alpha_6,\ \alpha_7).
$$

Again, the loss from $\sum_{C}^{\prime}$ is no more than
\begin{align}
\nonumber & \left( \int_{\boldsymbol{t}_{3} \in V_{C1}} \frac{\omega \left(\frac{1 - t_1 - t_2 - t_3}{t_3}\right)}{t_1 t_2 t_3^2} d t_3 d t_2 d t_1 \right) \\ 
\nonumber -& \left( \int_{\boldsymbol{t}_{4} \in V_{C2}} \frac{\omega \left(\frac{1 - t_1 - t_2 - t_3 - t_4}{t_4}\right)}{t_1 t_2 t_3 t_4^2} d t_4 d t_3 d t_2 d t_1 \right) \\
\nonumber +& \left( \int_{\boldsymbol{t}_{5} \in V_{C3}} \frac{\omega \left(\frac{1 - t_1 - t_2 - t_3 - t_4 - t_5}{t_5}\right)}{t_1 t_2 t_3 t_4 t_5^2} d t_5 d t_4 d t_3 d t_2 d t_1 \right) \\
\nonumber +& \left( \int_{\boldsymbol{t}_{5} \in V_{C4}} \frac{\omega \left(\frac{t_1 - t_5}{t_5}\right) \omega \left(\frac{1 - t_1 - t_2 - t_3 - t_4}{t_4}\right)}{t_2 t_3 t_4^2 t_5^2} d t_5 d t_4 d t_3 d t_2 d t_1 \right) \\
\nonumber +& \left( \int_{\boldsymbol{t}_{7} \in V_{C5}} \frac{\omega \left(\frac{1 - t_1 - t_2 - t_3 - t_4 - t_5 - t_6 - t_7}{t_7}\right)}{t_1 t_2 t_3 t_4 t_5 t_6 t_7^2} d t_7 d t_6 d t_5 d t_4 d t_3 d t_2 d t_1 \right) \\
\nonumber +& \left( \int_{\boldsymbol{t}_{7} \in V_{C6}} \frac{\omega \left(\frac{t_1 - t_5 - t_6 - t_7}{t_7}\right) \omega \left(\frac{1 - t_1 - t_2 - t_3 - t_4}{t_4}\right)}{t_2 t_3 t_4^2 t_5 t_6 t_7^2} d t_7 d t_6 d t_5 d t_4 d t_3 d t_2 d t_1 \right) \\
\nonumber +& \left( \int_{\boldsymbol{t}_{7} \in V_{C7}} \frac{\omega \left(\frac{t_1 - t_5}{t_5}\right) \omega \left(\frac{1 - t_1 - t_2 - t_3 - t_4 - t_6 - t_7}{t_7}\right)}{t_2 t_3 t_4 t_5^2 t_6 t_7^2} d t_7 d t_6 d t_5 d t_4 d t_3 d t_2 d t_1 \right) \\
\nonumber \leqslant & \left( \int_{\boldsymbol{t}_{3} \in V_{C1}} \frac{\omega_1 \left(\frac{1 - t_1 - t_2 - t_3}{t_3}\right)}{t_1 t_2 t_3^2} d t_3 d t_2 d t_1 \right) \\ 
\nonumber -& \left( \int_{\boldsymbol{t}_{4} \in V_{C2}} \frac{\omega_0 \left(\frac{1 - t_1 - t_2 - t_3 - t_4}{t_4}\right)}{t_1 t_2 t_3 t_4^2} d t_4 d t_3 d t_2 d t_1 \right) \\
\nonumber +& \left( \int_{\boldsymbol{t}_{5} \in V_{C3}} \frac{\omega_1 \left(\frac{1 - t_1 - t_2 - t_3 - t_4 - t_5}{t_5}\right)}{t_1 t_2 t_3 t_4 t_5^2} d t_5 d t_4 d t_3 d t_2 d t_1 \right) \\
\nonumber +& \left( \int_{\boldsymbol{t}_{5} \in V_{C4}} \frac{\omega_1 \left(\frac{t_1 - t_5}{t_5}\right) \omega_1 \left(\frac{1 - t_1 - t_2 - t_3 - t_4}{t_4}\right)}{t_2 t_3 t_4^2 t_5^2} d t_5 d t_4 d t_3 d t_2 d t_1 \right) \\
\nonumber +& \left( \int_{\boldsymbol{t}_{7} \in V_{C5}} \frac{\max \left(\frac{t_7}{1 - t_1 - t_2 - t_3 - t_4 - t_5 - t_6 - t_7}, 0.5672\right)}{t_1 t_2 t_3 t_4 t_5 t_6 t_7^2} d t_7 d t_6 d t_5 d t_4 d t_3 d t_2 d t_1 \right) \\
\nonumber +& \left( \int_{\boldsymbol{t}_{7} \in V_{C6}} \frac{\max \left(\frac{t_7}{t_1 - t_5 - t_6 - t_7}, 0.5672\right) \max \left(\frac{t_4}{1 - t_1 - t_2 - t_3 - t_4}, 0.5672\right)}{t_2 t_3 t_4^2 t_5 t_6 t_7^2} d t_7 d t_6 d t_5 d t_4 d t_3 d t_2 d t_1 \right) \\
\nonumber +& \left( \int_{\boldsymbol{t}_{7} \in V_{C7}} \frac{\max \left(\frac{t_5}{t_1 - t_5}, 0.5672\right) \max \left(\frac{t_7}{1 - t_1 - t_2 - t_3 - t_4 - t_6 - t_7}, 0.5672\right)}{t_2 t_3 t_4 t_5^2 t_6 t_7^2} d t_7 d t_6 d t_5 d t_4 d t_3 d t_2 d t_1 \right) \\
\leqslant &\ (0.31 - 0 + 0.13 + 0.25 + 0.02 + 0.005 + 0.001) < 0.716,
\end{align}
where
\begin{align}
\nonumber V_{C1}(\boldsymbol{\alpha}_{3}) :=&\ \left\{ \boldsymbol{\alpha}_{1} \notin H,\ \boldsymbol{\alpha}_{2} \in C,\ \nu_0 \leqslant \alpha_{3} < \min\left(\alpha_2, \frac{1}{2}(1-\alpha_1 -\alpha_2)\right), \right. \\
\nonumber & \quad \boldsymbol{\alpha}_{3} \notin \boldsymbol{G}_{3},\ \boldsymbol{\alpha}_{3} \notin \boldsymbol{D}^{+},\ \boldsymbol{\alpha}_{3} \notin \boldsymbol{D}^{\#}, \\
\nonumber & \left. \quad \nu_0 \leqslant \alpha_1 < \frac{1}{2},\ \nu\left(\alpha_1\right) \leqslant \alpha_2 < \min\left(\alpha_1, \frac{1}{2}(1-\alpha_1) \right) \right\}, \\
\nonumber V_{C2}(\boldsymbol{\alpha}_{4}) :=&\ \left\{ \boldsymbol{\alpha}_{1} \notin H,\ \boldsymbol{\alpha}_{2} \in C,\ \nu_0 \leqslant \alpha_{3} < \min\left(\alpha_2, \frac{1}{2}(1-\alpha_1 -\alpha_2)\right), \right. \\
\nonumber & \quad \boldsymbol{\alpha}_{3} \notin \boldsymbol{G}_{3},\ \boldsymbol{\alpha}_{3} \notin \boldsymbol{D}^{+},\ \boldsymbol{\alpha}_{3} \notin \boldsymbol{D}^{\#}, \\
\nonumber & \quad \alpha_3 < \alpha_4 < \frac{1}{2}(1-\alpha_1 -\alpha_2 -\alpha_3),\ \boldsymbol{\alpha}_{4} \in \boldsymbol{G}_{4}, \\
\nonumber & \left. \quad \nu_0 \leqslant \alpha_1 < \frac{1}{2},\ \nu\left(\alpha_1\right) \leqslant \alpha_2 < \min\left(\alpha_1, \frac{1}{2}(1-\alpha_1) \right) \right\}, \\
\nonumber V_{C3}(\boldsymbol{\alpha}_{5}) :=&\ \left\{ \boldsymbol{\alpha}_{1} \notin H,\ \boldsymbol{\alpha}_{2} \in C,\ \nu_0 \leqslant \alpha_{3} < \min\left(\alpha_2, \frac{1}{2}(1-\alpha_1 -\alpha_2)\right), \right. \\
\nonumber & \quad \boldsymbol{\alpha}_{3} \notin \boldsymbol{G}_{3},\ \boldsymbol{\alpha}_{3} \in \boldsymbol{D}^{+}, \\
\nonumber & \quad \nu_0 \leqslant \alpha_{4} < \min\left(\alpha_3, \frac{1}{2}(1-\alpha_1 -\alpha_2 -\alpha_3)\right),\ \boldsymbol{\alpha}_{4} \notin \boldsymbol{G}_{4}, \\
\nonumber & \quad \nu_0 \leqslant \alpha_{5} < \min\left(\alpha_4, \frac{1}{2}(1-\alpha_1 -\alpha_2 -\alpha_3 -\alpha_4)\right),\ \boldsymbol{\alpha}_{5} \notin \boldsymbol{G}_{5},\ \boldsymbol{\alpha}_{5} \notin \boldsymbol{D}^{++}, \\
\nonumber & \left. \quad \nu_0 \leqslant \alpha_1 < \frac{1}{2},\ \nu\left(\alpha_1\right) \leqslant \alpha_2 < \min\left(\alpha_1, \frac{1}{2}(1-\alpha_1) \right) \right\}, \\
\nonumber V_{C4}(\boldsymbol{\alpha}_{5}) :=&\ \left\{ \boldsymbol{\alpha}_{1} \notin H,\ \boldsymbol{\alpha}_{2} \in C,\ \nu_0 \leqslant \alpha_{3} < \min\left(\alpha_2, \frac{1}{2}(1-\alpha_1 -\alpha_2)\right), \right. \\
\nonumber & \quad \boldsymbol{\alpha}_{3} \notin \boldsymbol{G}_{3},\ \boldsymbol{\alpha}_{3} \notin \boldsymbol{D}^{+},\ \boldsymbol{\alpha}_{3} \in \boldsymbol{D}^{\#}, \\
\nonumber & \quad \nu_0 \leqslant \alpha_{4} < \min\left(\alpha_3, \frac{1}{2}(1-\alpha_1 -\alpha_2 -\alpha_3)\right),\ \boldsymbol{\alpha}_{4} \notin \boldsymbol{G}_{4}, \\
\nonumber & \quad \nu_0 \leqslant \alpha_{5} < \frac{1}{2}\alpha_1,\ \boldsymbol{\alpha}_{5}^{\#} \notin \boldsymbol{G}_{5},\ \boldsymbol{\alpha}_{5}^{\#} \notin \boldsymbol{D}^{++},\ \boldsymbol{\alpha}_{5}^{\# \prime} \notin \boldsymbol{D}^{++} \\
\nonumber & \left. \quad \nu_0 \leqslant \alpha_1 < \frac{1}{2},\ \nu\left(\alpha_1\right) \leqslant \alpha_2 < \min\left(\alpha_1, \frac{1}{2}(1-\alpha_1) \right) \right\}, \\
\nonumber V_{C5}(\boldsymbol{\alpha}_{7}) :=&\ \left\{ \boldsymbol{\alpha}_{1} \notin H,\ \boldsymbol{\alpha}_{2} \in C,\ \nu_0 \leqslant \alpha_{3} < \min\left(\alpha_2, \frac{1}{2}(1-\alpha_1 -\alpha_2)\right), \right. \\
\nonumber & \quad \boldsymbol{\alpha}_{3} \notin \boldsymbol{G}_{3},\ \boldsymbol{\alpha}_{3} \in \boldsymbol{D}^{+}, \\
\nonumber & \quad \nu_0 \leqslant \alpha_{4} < \min\left(\alpha_3, \frac{1}{2}(1-\alpha_1 -\alpha_2 -\alpha_3)\right),\ \boldsymbol{\alpha}_{4} \notin \boldsymbol{G}_{4}, \\
\nonumber & \quad \nu_0 \leqslant \alpha_{5} < \min\left(\alpha_4, \frac{1}{2}(1-\alpha_1 -\alpha_2 -\alpha_3 -\alpha_4)\right),\ \boldsymbol{\alpha}_{5} \notin \boldsymbol{G}_{5},\ \boldsymbol{\alpha}_{5} \in \boldsymbol{D}^{++}, \\
\nonumber & \quad \nu_0 \leqslant \alpha_{6} < \min\left(\alpha_5, \frac{1}{2}(1-\alpha_1 -\alpha_2 -\alpha_3 -\alpha_4 -\alpha_5)\right),\ \boldsymbol{\alpha}_{6} \notin \boldsymbol{G}_{6}, \\
\nonumber & \quad \nu_0 \leqslant \alpha_{7} < \min\left(\alpha_6, \frac{1}{2}(1-\alpha_1 -\alpha_2 -\alpha_3 -\alpha_4 -\alpha_5 -\alpha_6)\right),\ \boldsymbol{\alpha}_{7} \notin \boldsymbol{G}_{7}, \\
\nonumber & \left. \quad \nu_0 \leqslant \alpha_1 < \frac{1}{2},\ \nu\left(\alpha_1\right) \leqslant \alpha_2 < \min\left(\alpha_1, \frac{1}{2}(1-\alpha_1) \right) \right\}, \\
\nonumber V_{C6}(\boldsymbol{\alpha}_{7}) :=&\ \left\{ \boldsymbol{\alpha}_{1} \notin H,\ \boldsymbol{\alpha}_{2} \in C,\ \nu_0 \leqslant \alpha_{3} < \min\left(\alpha_2, \frac{1}{2}(1-\alpha_1 -\alpha_2)\right), \right. \\
\nonumber & \quad \boldsymbol{\alpha}_{3} \notin \boldsymbol{G}_{3},\ \boldsymbol{\alpha}_{3} \notin \boldsymbol{D}^{+},\ \boldsymbol{\alpha}_{3} \in \boldsymbol{D}^{\#}, \\
\nonumber & \quad \nu_0 \leqslant \alpha_{4} < \min\left(\alpha_3, \frac{1}{2}(1-\alpha_1 -\alpha_2 -\alpha_3)\right),\ \boldsymbol{\alpha}_{4} \notin \boldsymbol{G}_{4}, \\
\nonumber & \quad \nu_0 \leqslant \alpha_{5} < \frac{1}{2}\alpha_1,\ \boldsymbol{\alpha}_{5}^{\#} \notin \boldsymbol{G}_{5},\ \boldsymbol{\alpha}_{5}^{\#} \in \boldsymbol{D}^{++} \\
\nonumber & \quad \nu_0 \leqslant \alpha_{6} < \min\left(\alpha_5, \frac{1}{2}(\alpha_1 -\alpha_5)\right),\ \boldsymbol{\alpha}_{6}^{\#} \notin \boldsymbol{G}_{6}, \\
\nonumber & \quad \nu_0 \leqslant \alpha_{7} < \min\left(\alpha_6, \frac{1}{2}(\alpha_1 -\alpha_5 -\alpha_6)\right),\ \boldsymbol{\alpha}_{7}^{\#} \notin \boldsymbol{G}_{7}, \\
\nonumber & \left. \quad \nu_0 \leqslant \alpha_1 < \frac{1}{2},\ \nu\left(\alpha_1\right) \leqslant \alpha_2 < \min\left(\alpha_1, \frac{1}{2}(1-\alpha_1) \right) \right\}, \\
\nonumber V_{C7}(\boldsymbol{\alpha}_{7}) :=&\ \left\{ \boldsymbol{\alpha}_{1} \notin H,\ \boldsymbol{\alpha}_{2} \in C,\ \nu_0 \leqslant \alpha_{3} < \min\left(\alpha_2, \frac{1}{2}(1-\alpha_1 -\alpha_2)\right), \right. \\
\nonumber & \quad \boldsymbol{\alpha}_{3} \notin \boldsymbol{G}_{3},\ \boldsymbol{\alpha}_{3} \notin \boldsymbol{D}^{+},\ \boldsymbol{\alpha}_{3} \in \boldsymbol{D}^{\#}, \\
\nonumber & \quad \nu_0 \leqslant \alpha_{4} < \min\left(\alpha_3, \frac{1}{2}(1-\alpha_1 -\alpha_2 -\alpha_3)\right),\ \boldsymbol{\alpha}_{4} \notin \boldsymbol{G}_{4}, \\
\nonumber & \quad \nu_0 \leqslant \alpha_{5} < \frac{1}{2}\alpha_1,\ \boldsymbol{\alpha}_{5}^{\#} \notin \boldsymbol{G}_{5},\ \boldsymbol{\alpha}_{5}^{\#} \notin \boldsymbol{D}^{++},\ \boldsymbol{\alpha}_{5}^{\# \prime} \in \boldsymbol{D}^{++} \\
\nonumber & \quad \nu_0 \leqslant \alpha_{6} < \min\left(\alpha_4, \frac{1}{2}(1 -\alpha_1 -\alpha_2 -\alpha_3 -\alpha_4)\right),\ \boldsymbol{\alpha}_{6}^{\# \prime} \notin \boldsymbol{G}_{6}, \\
\nonumber & \quad \nu_0 \leqslant \alpha_{7} < \min\left(\alpha_6, \frac{1}{2}(1 -\alpha_1 -\alpha_2 -\alpha_3 -\alpha_4 -\alpha_6)\right),\ \boldsymbol{\alpha}_{7}^{\# \prime} \notin \boldsymbol{G}_{7}, \\
\nonumber & \left. \quad \nu_0 \leqslant \alpha_1 < \frac{1}{2},\ \nu\left(\alpha_1\right) \leqslant \alpha_2 < \min\left(\alpha_1, \frac{1}{2}(1-\alpha_1) \right) \right\}.
\end{align}

Finally, by (35), (36), (40), (41), (43), (44) and (53), the total loss from $\sum_{2}^{\prime}$ is less than
$$
0.183 + 0.255 + 0.38 + 0.12 + 0.22 + 0.716 < 1.874
$$
and we conclude that
$$
\pi(x)-\pi(x-x^{0.52})=S\left(\mathcal{A}, x^{\frac{1}{2}}\right) \leqslant 2.874 \frac{x^{0.52}}{\log x}.
$$
The upper constant $2.874$ can be slightly improved by more careful decompositions and accurate calculations. We remark that we can also use Lemma~\ref{l23} together with a variant of [\cite{BHP}, Lemma 17] on the lower bound problem, but the new four-dimensional loss after using them on $\sum_{A}$ exceeds the original two-dimensional loss when $\theta = 0.52$. Our upper bound result is weaker than Iwaniec's upper constant $\frac{4}{1+\theta} \approx 2.6316$ when $\theta = 0.52$, but our sieve approach gives better results for slightly longer intervals (of length $x^{0.522}$, $x^{0.523}$ and so on, see the values in following table). In fact, the upper constant rises rapidly as $\theta$ increases. The upper bounds for other values of $\theta$ between $0.52$ and $0.525$ can be proved in the same way, so we omit the calculation details. One can check our code for them to verify the numerical calculations.

\begin{center}
\begin{tabular}{|c|c|c|}
\hline \boldmath{$\theta$} & \textbf{New} \boldmath{$\operatorname{UB}(\theta)$} & \textbf{Iwaniec's} \boldmath{$\operatorname{UB}(\theta)$} \\
\hline \boldmath{$0.520$} & $<2.874$ & $\frac{4}{1+0.52} <2.6316$  \\
\hline \boldmath{$0.521$} & $<2.700$ & $\frac{4}{1+0.521} <2.6299$  \\
\hline \boldmath{$0.522$} & \boldmath{$<2.583$} & $\frac{4}{1+0.522} <2.6282$  \\
\hline \boldmath{$0.523$} & \boldmath{$<2.536$} & $\frac{4}{1+0.523} <2.6264$  \\
\hline \boldmath{$0.524$} & \boldmath{$<2.437$} & $\frac{4}{1+0.524} <2.6247$  \\
\hline \boldmath{$0.525$} & \boldmath{$<2.347$} & $\frac{4}{1+0.525} <2.6230$  \\
\hline
\end{tabular}
\end{center}

\section{Applications}
Clearly our Theorem~\ref{t1} has many interesting applications (just like the previous BHP's result), and we state some of them in this section. Note that we ignore the presence of $\varepsilon$ because we can use the same method to prove Theorem~\ref{t1} with a slightly smaller $\theta$, such as $0.52 - 10^{-100}$. The first application is about primes in arithmetic progressions in short intervals, which improves upon the result of Harman [\cite{HarmanBOOK}, Theorem 10.8].
\begin{theorem1number}\label{t2}
For all $q \leqslant (\log x)^{K}$ and any $a$ coprime to $q$, we have
$$
\pi(x;q,a)-\pi(x-x^{0.52};q,a) \geqslant 0.004 \frac{x^{0.52}}{\varphi(q) \log x}.
$$
\end{theorem1number}

Another application is about bounded gaps between primes in short intervals, which improves the result of Alweiss and Luo [\cite{AlweissLuo}, Corollary 1.2].
\begin{theorem1number}\label{t3}
There exist positive integers $k, d$ such that the interval $[x - x^{0.52}, x]$ contains $\gg x^{0.52}(\log x)^{-k}$ pairs of consecutive primes differing by at most $d$.
\end{theorem1number}

The third application is about primes with prime subscripts (or prime-primes) in short intervals. By using the numerical bound in Theorem~\ref{t1B}, we can derive the following theorem.
\begin{theorem1number}\label{PrimePrime}
We have
$$
\pi(\pi(x))-\pi(\pi(x-x^{0.52})) \geqslant 0.004 \frac{x^{0.52}}{(\log x)^2}.
$$
\end{theorem1number}
\noindent The bound for the number of prime-primes in interval of length $x^{0.525}$ was obtained by Broughan and Barnett \cite{BroughanBarnett}, where they also proved the analogs of Prime Number Theorem and weak Dirichlet's Theorem for prime-primes.

The next two applications focus on Goldbach numbers (sum of two primes) in short intervals. By replacing [\cite{HarmanBOOK}, Theorem 10.8] by our Theorem~\ref{t2} in the proof of the main theorem in \cite{Grimmelt}, we can easily deduce the following result.
\begin{theorem1number}\label{t4}
Almost all even numbers in the interval $[x, x+x^{\frac{13}{225}}]$ are Goldbach numbers.
\end{theorem1number}
\noindent By combining our Theorem~\ref{t1} with the main theorem proved in \cite{LiRunbo1215}, we can easily show the following result.
\begin{theorem1number}\label{t5}
The interval $[x, x+x^{\frac{26}{1075}}]$ contains Goldbach numbers.
\end{theorem1number}
\noindent Note that $\frac{13}{225} \approx 0.0578$ and $\frac{26}{1075} \approx 0.0242$. Previous exponents $\frac{7}{120} \approx 0.0583$ [\cite{Grimmelt}, Theorem 1.1] and $\frac{21}{860} \approx 0.0244$ [\cite{LiRunbo1215}, Theorem 8.1] come from BHP's $0.525$. We remark that if we focus on Maillet numbers (difference of two primes) instead of Goldbach numbers in short intervals, Pintz \cite{PintzM} improved the exponent in Theorem~\ref{t5} to any $\varepsilon > 0$.

The next four applications of Theorem~\ref{t1} are not direct corollaries of Theorem~\ref{t1} and Theorem~\ref{t1B}. However, since the arithmetic information inputs are quite similar, we can easily get these results with a slight modification of our calculation. We remark that we can only get an exponent $0.5248$ in these applications since the corresponding arithmetic information is weaker than that in Section 3 and 4. The first one is about the distribution of prime ideals of imaginary quadratic fields.
\begin{theorem1number}\label{t6}
Let $d < 0$ be the discriminant of an imaginary quadratic field $K = \mathbb{Q}(\sqrt{d})$, and let $Q(x,y) \in \mathbb{Z}[x,y]$ be a positive definite quadratic form with discriminant $d$. Then, for every pair $(s,t) \in \mathbb{R}^2$, there is another pair $(m,n) \in \mathbb{Z}^2$ for which $Q(m,n)$ is prime and
$$
Q(s-m,t-n) \ll Q(s,t)^{0.5248} +1.
$$
Specially, For every $z \in \mathbb{C}$, one can find a Gaussian prime $\mathfrak{p} \neq z$ satisfying
$$
\left| z - \mathfrak{p} \right| \ll |z|^{0.5248} +1.
$$
\end{theorem1number}
\noindent Theorem~\ref{t6} improves previous results of Lewis \cite{LewisThesis} and Harman, Kumchev and Lewis \cite{HarmanKumchevLewis}, who got exponents $0.528$ and $0.53$ respectively.

The second one focuses on primes in arithmetic progressions valid except for a small set of exceptional moduli.
\begin{theorem1number}\label{t7}
There exists a $C > 0$ such that if $q$ is large, all prime factors of $q$ is less than $q^{C}$, and we have
$$
L(s, \chi) \neq 0 \text{ for } \mathrm{Re}s > 1- \frac{1}{(\log q)^{\frac{3}{4}}}, \qquad |t| \leqslant \exp\left(\varepsilon (\log q)^{\frac{3}{4}} \right)
$$
for every $d \mid q$ with $\chi$ a primitive character mod $d$, then for any $a$ coprime to $q$, we have
$$
\pi(x;q,a) \gg \frac{x}{\varphi(q)\log x}
$$
whenever $q < x^{0.4752}$.
\end{theorem1number}

The third one is a corollary of Theorem~\ref{t7}, which concerns the number of Carmichael numbers less than $x$. By combining our exponent $0.5248$ with [\cite{Lichtman2}, Theorem 1.1], we know that
\begin{theorem1number}\label{t8}
Let $Carm(x)$ denote the number of Carmichael numbers less than $x$. Then we have
$$
Carm(x) > x^{(1-0.2844)(1-0.5248)} > x^{0.34}.
$$
\end{theorem1number}
\noindent Theorem~\ref{t8} improves previous results of Lichtman \cite{Lichtman2} and Harman \cite{HarmanCarmichael2} \cite{HarmanCarmichael1}, who got exponents $0.3389$, $\frac{1}{3}$ and $0.33$ respectively.

The fourth one is another corollary of Theorem~\ref{t7}, which focuses on Linnik's constant for prime power moduli with a fixed prime and improves the result of Banks and Shparlinski \cite{BanksShparlinski}.
\begin{theorem1number}\label{t9}
For any $q = p^{R}$ with a large integer $R$ and any $a$ coprime to $p$, we have
$$
\sum_{\substack{n \leqslant x \\ n \equiv a (\bmod q) }} \Lambda(n) \gg \frac{x}{\varphi(q)}
$$
for any $x > q^{\frac{1}{0.4752}}$. Specially, we can bound Linnik's constant by $\frac{1}{0.4752} < 2.1044$ if $q$ is a power of a fixed prime.
\end{theorem1number}

The tenth application of Theorem~\ref{t1} concerns the work of Erdős and Rényi \cite{ErdosRenyi} on Turán’s problem 10. By applying the methods in \cite{Andersson2006} together with our Theorem~\ref{t1}, we can obtain the following upper bound of the power sum of complex $z$:
\begin{theorem1number}\label{t10}
We have
$$
\inf_{|z_k| \geqslant 1} \max_{v = 1, \ldots, n^2} \left|\sum_{1 \leqslant k \leqslant n}z^v_k \right| = \sqrt{n} + O \left(n^{0.26}\right).
$$
\end{theorem1number}
\noindent Theorem~\ref{t10} improves upon the result of Andersson \cite{Andersson2006}, which has an error of $O \left(n^{0.2625}\right)$.

Next application of Theorem~\ref{t1} gives a better lower bound for the pairs of ``symmetric primes'', which was first considered by Tang and Wu \cite{TangWu}. By applying our Theorem~\ref{t1} directly, we can get the following bound.
\begin{theorem1number}\label{t11}
We have
$$
\sum_{\substack{p \leqslant x \\ \exists p^{\prime} \text{ such that } [x/p^{\prime}] = p}} 1 \gg \frac{x^{\frac{12}{37}}}{\log x}.
$$
\end{theorem1number}
\noindent Note that $\frac{12}{37}=\frac{1-0.52}{2-0.52}$. Theorem~\ref{t11} improves upon the result of Tang and Wu \cite{TangWu}, which has a lower bound $x^{\frac{19}{59}} (\log x)^{-1}$ comes from BHP's $0.525$.

The twelfth application of Theorem~\ref{t1} focuses on the size of a Sidon set and the sum of elements in it. By applying Theorem~\ref{t1} together with the methods in \cite{DingYC}, we can get the following result.
\begin{theorem1number}\label{t12}
Let $S$ be a Sidon set in $\{1,2,\ldots,n\}$ with $|S| = S_n$, then we have
$$
S_n = n^{\frac{1}{2}} + O\left(n^{\frac{13}{50}}\right)
$$
for positive integers $n$, and
$$
\sum_{a \in S}a = \frac{1}{2} n^{\frac{3}{2}} + O\left(n^{\frac{69}{50}}\right).
$$
\end{theorem1number}
\noindent The second part of Theorem~\ref{t12} improves upon the result of Ding [\cite{DingYC}, Corollary 1.3], where he proved an error of $O\left(n^{\frac{221}{160}}\right)$. The first part of Theorem~\ref{t12} is a direct corollary of Theorem~\ref{t1}. One can see [\cite{DingYC}, Lemma 2.3] for a proof with $\theta = 0.525$.

The last application of Theorem~\ref{t1} is Waring--Goldbach problem in short intervals, which improves the previous result of Wang [\cite{WangMengdi}, Corollary 2]. Using our new Theorem~\ref{t2} together with [\cite{WangMengdi}, Theorem 1], we can get the following result.
\begin{theorem1number}\label{t13}
Let $v = v(p, k)$ denote the integer such that $p^{v} \mid k$ but $p^{v+1} \nmid k$. Define
\begin{align*}
y = y(p, k) = 
\begin{cases}
v + 2, & \text{if } p = 2 \text{ and } v > 0, \\
v + 1, & \text{otherwise}
\end{cases}
\end{align*}
and
$$
R_k = \prod_{(p-1) \mid k}p^y.
$$
Then, when $k \geqslant 2$, $\theta > 0.52$ and $s > \max\left(\frac{12500}{13}, k(k+1)\right)$, for all sufficiently large $n \equiv s (\bmod R_k)$, there are primes 
$$
p_1, \ldots, p_s \in \left[ \left(\frac{n}{s}\right)^{\frac{1}{k}} - n^{\frac{\theta}{k}} , \left(\frac{n}{s}\right)^{\frac{1}{k}} + n^{\frac{\theta}{k}} \right]
$$
such that
$$
n = p_1^k + \cdots + p_s^k.
$$
Note that $\frac{12500}{13} = \frac{2}{0.004 \times 0.52} \approx 961.5$.
\end{theorem1number}

\newpage
\section*{Appendix 1: Plot of Regions}
\begin{center}
\includegraphics[height=8cm]{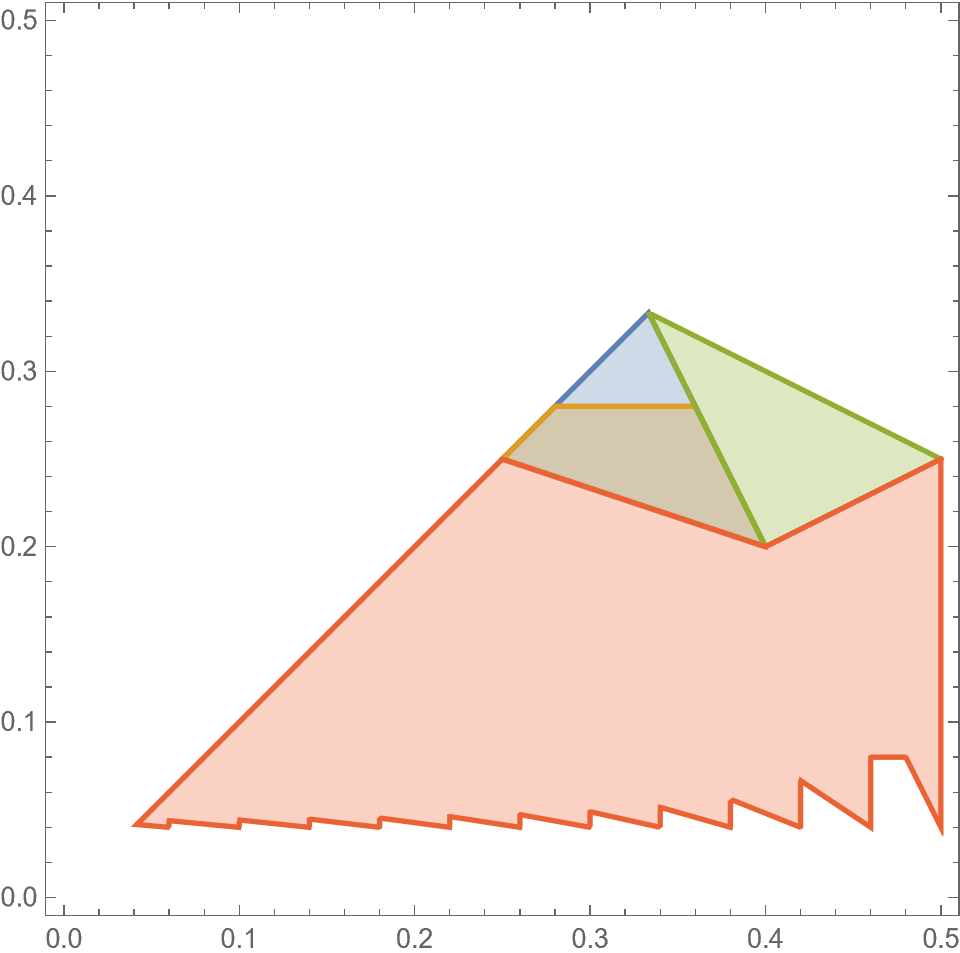}
\includegraphics[height=8cm]{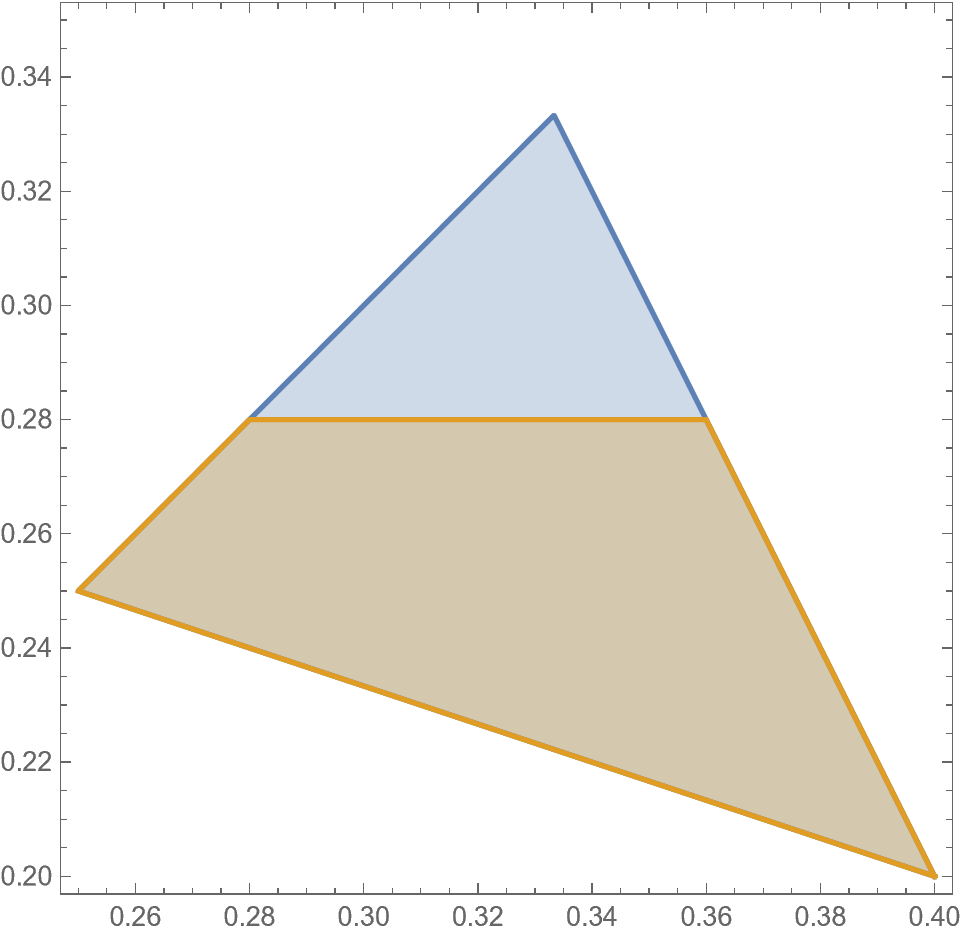}
\includegraphics[height=8cm]{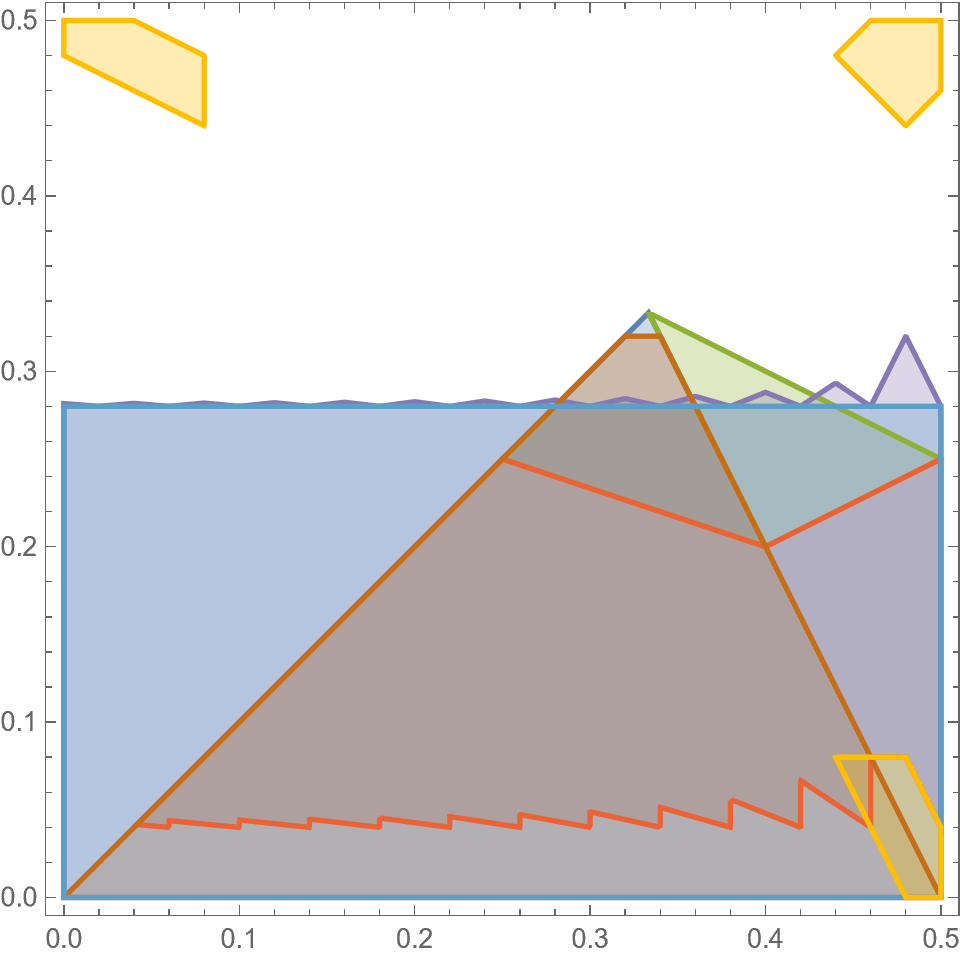}
\includegraphics[height=8cm]{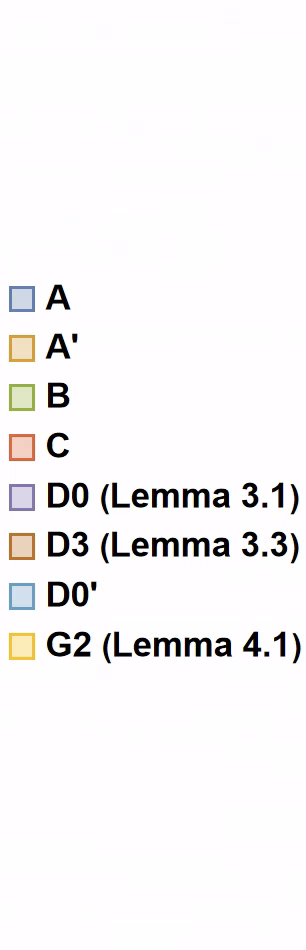}
\\
\textbf{Figure 1: Plot for \boldmath{$\operatorname{LB}(0.52)$}}
\end{center}

\newpage
\begin{center}
\includegraphics[height=8cm]{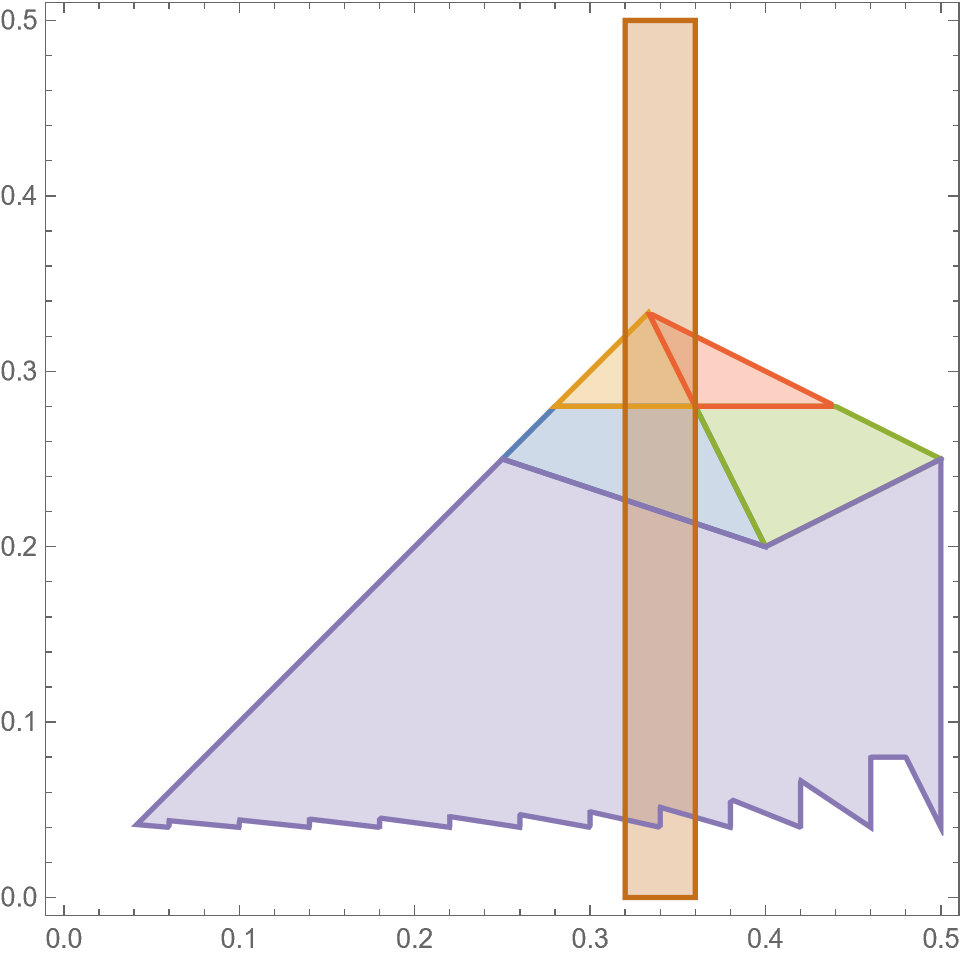}
\includegraphics[height=8cm]{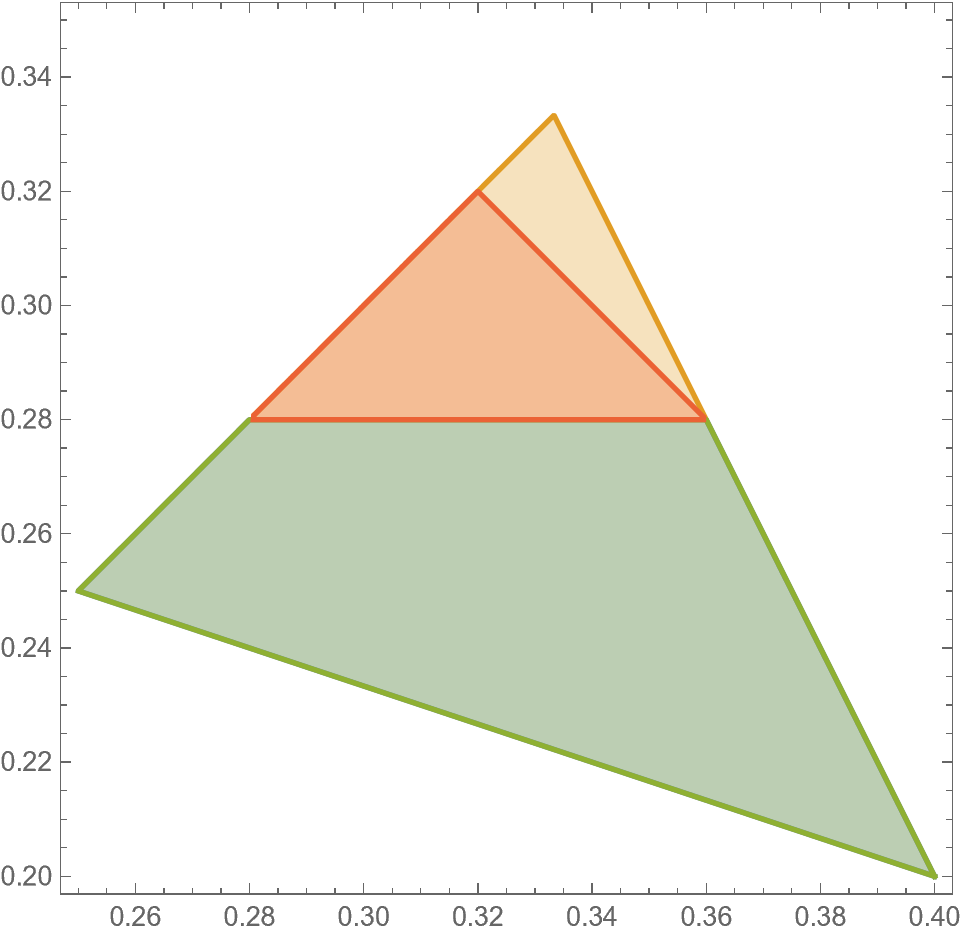}
\includegraphics[height=8cm]{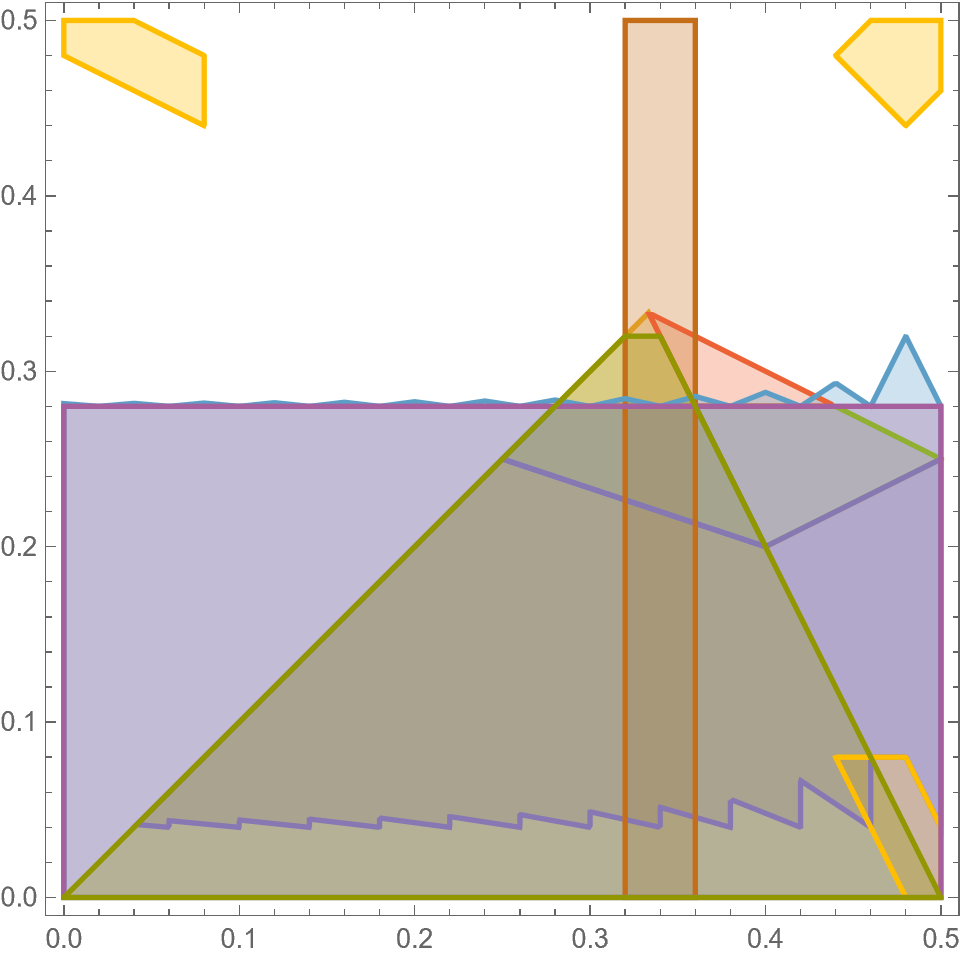}
\includegraphics[height=8cm]{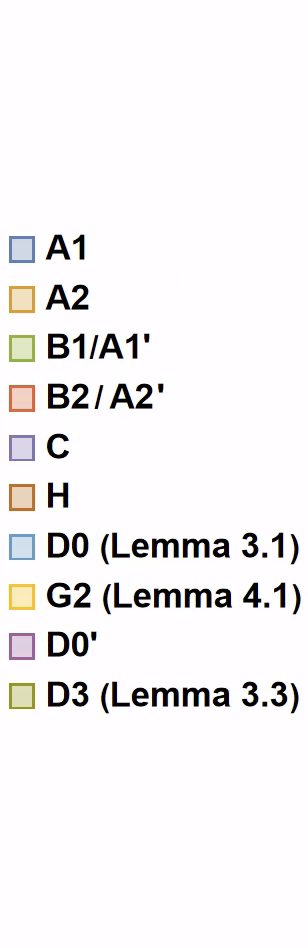}
\\

\textbf{Figure 2: Plot for \boldmath{$\operatorname{UB}(0.52)$}}
\end{center}

\newpage
\section*{Appendix 2: Values of Integrals}

\begin{center}
\begin{tabular}{|c|c|c|c|c|c|c|}
\hline \boldmath{$\theta$} & \boldmath{$0.52$} & \boldmath{$0.521$} & \boldmath{$0.522$} & \boldmath{$0.523$} & \boldmath{$0.524$} & \boldmath{$0.525$} \\
\hline \textbf{Loss from} \boldmath{$A$} & $< 0.241$ & $< 0.241$ & $< 0.241$ & $< 0.241$ & $< 0.241$ & $< 0.241$ \\
\hline $U_{C01}$ & $< 0.21$ & $< 0.19$ & $< 0.17$ & $< 0.155$ & $< 0.135$ & $< 0.12$ \\
\hline $U_{C02}$ & $0$ & $0$ & $0$ & $0$ & $0$ & $0$ \\
\hline $U_{C03}$ & $< 0.015$ & $< 0.01$ & $< 0.005$ & $< 0.005$ & $< 0.005$ & $< 0.005$ \\
\hline $U_{C04}$ & $< 0.05$ & $< 0.03$ & $< 0.02$ & $< 0.015$ & $< 0.01$ & $< 0.005$ \\
\hline $U_{C05}$ & $< 0.001$ & $< 0.001$ & $< 0.001$ & $< 0.001$ & $< 0.001$ & $< 0.001$ \\
\hline $U_{C06}$ & $< 0.001$ & $< 0.001$ & $< 0.001$ & $< 0.001$ & $< 0.001$ & $< 0.001$ \\
\hline $U_{C07}$ & $< 0.001$ & $< 0.001$ & $< 0.001$ & $< 0.001$ & $< 0.001$ & $< 0.001$ \\
\hline $U_{C08}$ & $< 0.22$ & $< 0.2$ & $< 0.18$ & $< 0.165$ & $< 0.15$ & $< 0.13$ \\
\hline $U_{C09}$ & $0$ & $0$ & $0$ & $0$ & $0$ & $0$ \\
\hline $U_{C10}$ & $0$ & $0$ & $0$ & $0$ & $0$ & $0$ \\
\hline $U_{C11}$ & $0$ & $0$ & $0$ & $0$ & $0$ & $0$ \\
\hline $U_{C12}$ & $< 0.015$ & $< 0.01$ & $< 0.005$ & $< 0.005$ & $< 0.005$ & $< 0.005$ \\
\hline $U_{C13}$ & $< 0.001$ & $< 0.001$ & $< 0.001$ & $< 0.001$ & $< 0.001$ & $< 0.001$ \\
\hline \textbf{Loss from} \boldmath{$C$} & $< 0.514$ & $< 0.443$ & $< 0.384$ & $< 0.349$ & $< 0.309$ & $< 0.269$\\
\hline \textbf{Total Loss} & $< 0.996$ & $< 0.925$ & $< 0.866$ & $< 0.831$ & $< 0.791$ & $< 0.751$\\
\hline \textbf{Lower Bound} & $> 0.004$ & $> 0.075$ & $> 0.134$ & $> 0.169$ & $> 0.209$ & $> 0.249$\\
\hline
\end{tabular}
\\
\textbf{Table 1: Values for \boldmath{$\operatorname{LB}(\theta)$} (C\boldmath{$++$})}
\end{center}

\hspace*{\fill}

\begin{center}
\begin{tabular}{|c|c|c|c|c|c|c|}
\hline \boldmath{$\theta$} & \boldmath{$0.52$} & \boldmath{$0.521$} & \boldmath{$0.522$} & \boldmath{$0.523$} & \boldmath{$0.524$} & \boldmath{$0.525$} \\
\hline \textbf{Loss from} \boldmath{$H$} & $< 0.183$ & $< 0.134$ & $< 0.086$ & $< 0.039$ & $0$ & $0$ \\
\hline $V_{A1}$ & $< 0.19$ & $< 0.23$ & $< 0.26$ & $< 0.3$ & $< 0.33$ & $< 0.33$ \\
\hline $V_{A2}$ & $> 0.005$ & $> 0.01$ & $> 0.01$ & $> 0.015$ & $> 0.03$ & $> 0.035$ \\
\hline $V_{A3}$ & $< 0.07$ & $< 0.05$ & $< 0.035$ & $< 0.03$ & $< 0.02$ & $< 0.015$ \\
\hline \textbf{Loss from} \boldmath{$A_1$} & $< 0.255$ & $< 0.27$ & $< 0.285$ & $< 0.315$ & $< 0.32$ & $< 0.31$ \\
\hline $V_{A4}$ & $< 0.32$ & $< 0.32$ & $< 0.32$ & $< 0.33$ & $< 0.33$ & $< 0.33$ \\
\hline $V_{A5}$ & $> 0.01$ & $> 0.01$ & $> 0.01$ & $> 0.015$ & $> 0.03$ & $> 0.035$ \\
\hline $V_{A6}$ & $< 0.07$ & $< 0.05$ & $< 0.035$ & $< 0.03$ & $< 0.02$ & $< 0.015$ \\
\hline \textbf{Loss from} \boldmath{$A_1^{\prime}$} & $< 0.38$ & $< 0.36$ & $< 0.345$ & $< 0.345$ & $< 0.32$ & $< 0.31$ \\
\hline \textbf{Loss from} \boldmath{$A_2$} & $< 0.12$ & $< 0.11$ & $< 0.13$ & $< 0.14$ & $< 0.17$ & $< 0.16$ \\
\hline \textbf{Loss from} \boldmath{$A_2^{\prime}$} & $< 0.22$ & $< 0.2$ & $< 0.19$ & $< 0.18$ & $< 0.17$ & $< 0.16$ \\
\hline $V_{C1}$ & $< 0.31$ & $< 0.31$ & $< 0.3$ & $< 0.3$ & $< 0.3$ & $< 0.29$ \\
\hline $V_{C2}$ & $0$ & $0$ & $0$ & $0$ & $0$ & $0$ \\
\hline $V_{C3}$ & $< 0.13$ & $< 0.09$ & $< 0.07$ & $< 0.06$ & $< 0.04$ & $< 0.03$ \\
\hline $V_{C4}$ & $< 0.25$ & $< 0.21$ & $< 0.17$ & $< 0.15$ & $< 0.11$ & $< 0.08$ \\
\hline $V_{C5}$ & $< 0.02$ & $< 0.01$ & $< 0.005$ & $< 0.005$ & $< 0.005$ & $< 0.005$ \\
\hline $V_{C6}$ & $< 0.005$ & $< 0.005$ & $< 0.001$ & $< 0.001$ & $< 0.001$ & $< 0.001$ \\
\hline $V_{C7}$ & $< 0.001$ & $< 0.001$ & $< 0.001$ & $< 0.001$ & $< 0.001$ & $< 0.001$ \\
\hline \textbf{Loss from} \boldmath{$C$} & $< 0.716$ & $< 0.626$ & $< 0.547$ & $< 0.517$ & $< 0.457$ & $< 0.407$ \\
\hline \textbf{Total Loss} & $< 1.874$ & $< 1.7$ & $< 1.583$ & $< 1.536$ & $< 1.437$ & $< 1.347$ \\
\hline \textbf{Upper Bound} & $< 2.874$ & $< 2.7$ & $< 2.583$ & $< 2.536$ & $< 2.437$ & $< 2.347$ \\
\hline
\end{tabular}
\\
\textbf{Table 2: Values for \boldmath{$\operatorname{UB}(\theta)$} (C\boldmath{$++$})}
\end{center}

\newpage

\begin{center}
\begin{tabular}{|c|c|c|c|c|c|}
\hline \boldmath{$\theta$} & \boldmath{$0.52$} & \boldmath{$0.521$} & \boldmath{$0.522$} & \boldmath{$0.523$} & \boldmath{$0.524$} \\
\hline \textbf{Loss from} \boldmath{$A$} & $< 0.240227$ & $< 0.240227$ & $< 0.240227$ & $< 0.240227$ & $< 0.240227$ \\
\hline $U_{C01}$ & $< 0.197907$ & $< 0.178493$ & $< 0.158194$ & $< 0.136616$ & $< 0.119466$ \\
\hline $U_{C02}$ & $> 0.001607$ & $> 0.001789$ & $> 0.001688$ & $> 0.001790$ & $> 0.001787$ \\
\hline $U_{C03}$ & $< 0.020936$ & $< 0.014611$ & $< 0.010405$ & $< 0.005868$ & $< 0.003499$ \\
\hline $U_{C04}$ & $< 0.065033$ & $< 0.043988$ & $< 0.037108$ & $< 0.014831$ & $< 0.007705$ \\
\hline $U_{C05}$ & $< 0.000101$ & $< 0.000043$ & $< 0.000018$ & $< 0.000006$ & $< 0.000002$ \\
\hline $U_{C06}$ & $< 0.000131$ & $< 0.000106$ & $< 0.000073$ & $0$ & $0$ \\
\hline $U_{C07}$ & $0$ & $0$ & $0$ & $0$ & $0$ \\
\hline $U_{C08}$ & $< 0.201090$ & $< 0.181571$ & $< 0.165022$ & $< 0.143845$ & $< 0.128240$ \\
\hline $U_{C09}$ & $> 0.000693$ & $> 0.001054$ & $> 0.001193$ & $> 0.001259$ & $> 0.001338$ \\
\hline $U_{C10}$ & $> 0.000222$ & $> 0.000251$ & $> 0.000286$ & $> 0.000295$ & $> 0.000293$ \\
\hline $U_{C11}$ & $< 0.000048$ & $< 0.000040$ & $< 0.000029$ & $< 0.000028$ & $< 0.000265$ \\
\hline $U_{C12}$ & $< 0.008809$ & $< 0.005143$ & $< 0.004523$ & $< 0.001541$ & $< 0.000727$ \\
\hline $U_{C13}$ & $0$ & $0$ & $0$ & $0$ & $0$ \\
\hline \textbf{Loss from} \boldmath{$C$} & $< 0.491533$ & $< 0.420901$ & $< 0.372205$ & $< 0.299391$ & $< 0.256486$ \\
\hline \textbf{Total Loss} & $< 0.971987$ & $< 0.901355$ & $< 0.852659$ & $< 0.779845$ & $< 0.736940$ \\
\hline \textbf{Lower Bound} & $> 0.028013$ & $> 0.098645$ & $> 0.147341$ & $> 0.220155$ & $> 0.263060$ \\
\hline
\end{tabular}
\\
\textbf{Table 3: Values for \boldmath{$\operatorname{LB}(\theta)$} (Mathematica, Error Added)}
\end{center}

\hspace*{\fill}

\begin{center}
\begin{tabular}{|c|c|c|c|c|c|}
\hline \boldmath{$\theta$} & \boldmath{$0.52$} & \boldmath{$0.521$} & \boldmath{$0.522$} & \boldmath{$0.523$} & \boldmath{$0.524$} \\
\hline \textbf{Loss from} \boldmath{$H$} & $< 0.182012$ & $< 0.133815$ & $< 0.085930$ & $< 0.038334$ & $0$ \\
\hline $V_{A1}$ & $< 0.179773$ & $< 0.217159$ & $< 0.254821$ & $< 0.292355$ & $< 0.323686$ \\
\hline $V_{A2}$ & $> 0.004874$ & $> 0.007200$ & $> 0.010359$ & $> 0.017561$ & $> 0.023389$ \\
\hline $V_{A3}$ & $< 0.043475$ & $< 0.035114$ & $< 0.027426$ & $< 0.020820$ & $< 0.015243$ \\
\hline \textbf{Loss from} \boldmath{$A_1$} & $< 0.218374$ & $< 0.245073$ & $< 0.271888$ & $< 0.295614$ & $< 0.315540$ \\
\hline $V_{A4}$ & $< 0.310609$ & $< 0.313652$ & $< 0.316896$ & $< 0.320119$ & $< 0.323686$ \\
\hline $V_{A5}$ & $> 0.008299$ & $> 0.010006$ & $> 0.012635$ & $> 0.019583$ & $> 0.023389$ \\
\hline $V_{A6}$ & $< 0.051108$ & $< 0.038581$ & $< 0.028772$ & $< 0.021186$ & $< 0.015243$ \\
\hline \textbf{Loss from} \boldmath{$A_1^{\prime}$} & $< 0.353418$ & $< 0.342227$ & $< 0.333033$ & $< 0.321722$ & $< 0.315540$ \\
\hline \textbf{Loss from} \boldmath{$A_2$} & $< 0.102865$ & $< 0.109021$ & $< 0.122256$ & $< 0.140969$ & $< 0.155383$ \\
\hline \textbf{Loss from} \boldmath{$A_2^{\prime}$} & $< 0.201264$ & $< 0.195899$ & $< 0.187831$ & $< 0.173941$ & $< 0.155383$ \\
\hline $V_{C1}$ & $< 0.261034$ & $< 0.260555$ & $< 0.257913$ & $< 0.254700$ & $< 0.249854$ \\
\hline $V_{C2}$ & $> 0.000575$ & $> 0.000787$ & $> 0.000850$ & $> 0.000815$ & $> 0.000795$ \\
\hline $V_{C3}$ & $< 0.128160$ & $< 0.107541$ & $< 0.092325$ & $< 0.070907$ & $< 0.055342$ \\
\hline $V_{C4}$ & $< 0.307367$ & $< 0.249849$ & $< 0.210236$ & $< 0.163109$ & $< 0.128740$ \\
\hline $V_{C5}$ & $< 0.004722$ & $< 0.002606$ & $< 0.001446$ & $< 0.000670$ & $< 0.000322$ \\
\hline $V_{C6}$ & $< 0.003889$ & $< 0.002529$ & $< 0.000965$ & $< 0.000461$ & $< 0.000512$ \\
\hline $V_{C7}$ & $< 0.000013$ & $0$ & $0$ & $0$ & $0$ \\
\hline \textbf{Loss from} \boldmath{$C$} & $< 0.704610$ & $< 0.622293$ & $< 0.562035$ & $< 0.489032$ & $< 0.433975$ \\
\hline \textbf{Total Loss} & $< 1.762543$ & $< 1.648328$ & $< 1.562973$ & $< 1.459612$ & $< 1.375821$ \\
\hline \textbf{Upper Bound} & $< 2.762543$ & $< 2.648328$ & $< 2.562973$ & $< 2.459612$ & $< 2.375821$ \\
\hline
\end{tabular}
\\
\textbf{Table 4: Values for \boldmath{$\operatorname{UB}(\theta)$} (Mathematica, Error Added)}
\end{center}

\newpage

\begin{center}
\begin{tabular}{|c|c|c|c|c|}
\hline \boldmath{$\theta = 0.52$} & \textbf{C}\boldmath{$++$} & \textbf{C}\boldmath{$++$}\textbf{~Bounds} & \textbf{Mathematica} & \textbf{Mathematica Error} \\
\hline \textbf{Loss from} \boldmath{$A$} & $--$ & $< 0.241$ & $< 0.240227$ & $0$ \\
\hline $U_{C01}$ & \makecell{$< 0.205494$ \\ $< 0.204864$ \\ $< 0.202968$ \\ $< 0.202764$ \\ $< 0.204471$} & $< 0.21$ & $< 0.179029$ & $\pm 0.018879$ \\
\hline $U_{C02}$ & $--$ & $0$ & $> 0.002844$ & $\pm 0.001237$ \\
\hline $U_{C03}$ & \makecell{$< 0.00867172$ \\ $< 0.00852254$ \\ $< 0.00570209$ \\ $< 0.00614043$ \\ $< 0.00595708$ \\ $< 0.00668194$} & $< 0.015$ & $< 0.011131$ & $\pm 0.009805$ \\
\hline $U_{C04}$ & \makecell{$< 0.0331034$ \\ $< 0.0250261$ \\ $< 0.0281436$ \\ $< 0.0307910$ \\ $< 0.0257140$ \\ $< 0.0255659$ \\ $< 0.0297102$ \\ $< 0.0262677$ \\ $< 0.0313978$ \\ $< 0.0312939$ \\ $< 0.0354621$ \\ $< 0.0341336$ } & $< 0.05$ & $< 0.043774$ & $\pm 0.021259$ \\
\hline $U_{C05}$ & $--$ & $< 0.001$ & $< 0.000087$ & $\pm 0.000015$ \\
\hline $U_{C06}$ & $--$ & $< 0.001$ & $< 0.000066$ & $\pm 0.000065$ \\
\hline $U_{C07}$ & $--$ & $< 0.001$ & $0$ & $0$ \\
\hline $U_{C08}$ & \makecell{$< 0.206760$ \\ $< 0.207967$ \\ $< 0.210592$ \\ $< 0.212729$ \\ $< 0.212458$} & $< 0.22$ & $< 0.182786$ & $\pm 0.018305$ \\
\hline $U_{C09}$ & $--$ & $0$ & $> 0.002623$ & $\pm 0.001930$ \\
\hline $U_{C10}$ & $--$ & $0$ & $> 0.000844$ & $\pm 0.000623$ \\
\hline $U_{C11}$ & \makecell{$< U_{C09}$ \\ $< U_{C10}$ \\ $< 0.000363954$} & \makecell{$0$ \\ ($U_{C09} + U_{C10} - U_{C11} > 0$)} & $< 0.000015$ & $\pm 0.000033$ \\
\hline $U_{C12}$ & \makecell{$< 0.00567703$ \\ $< 0.00520828$ \\ $< 0.00643319$ \\ $< 0.00603185$ \\ $< 0.00644211$ \\ $< 0.00713637$} & $< 0.015$ & $< 0.005207$ & $\pm 0.003603$ \\
\hline $U_{C13}$ & $0$ & $< 0.001$ & $0$ & $0$ \\
\hline \textbf{Loss from} \boldmath{$C$} & $--$ & $< 0.514$ & $< 0.491533$ & $--$ \\
\hline \textbf{Total Loss} & $--$ & $< 0.996$ & $< 0.971987$ & $--$ \\
\hline \textbf{Lower Bound} & $--$ & $> 0.004$ & $> 0.028013$ & $--$ \\
\hline
\end{tabular}
\\
\textbf{Table 5: Values for \boldmath{$\operatorname{LB}(0.52)$} (Comparison)}
\end{center}

\newpage

\begin{center}
\begin{tabular}{|c|c|}
\hline \boldmath{$\operatorname{LB}(\theta)$} & \textbf{Code} \\
\hline \boldmath{$0.520$} & \makecell{\texttt{https://notebookarchive.org/2025-04-3pjf0hd} \\ \texttt{https://notebookarchive.org/2025-04-9p1on6y}} \\
\hline \boldmath{$0.521$} & \makecell{\texttt{https://notebookarchive.org/2025-04-3pju4ml} \\ \texttt{https://notebookarchive.org/2025-04-9p1t3hj}} \\
\hline \boldmath{$0.522$} & \makecell{\texttt{https://notebookarchive.org/2025-04-3pk3cy5} \\ \texttt{https://notebookarchive.org/2025-04-9p1xmq7}} \\
\hline \boldmath{$0.523$} & \makecell{\texttt{https://notebookarchive.org/2025-04-3pmvmzo} \\ \texttt{https://notebookarchive.org/2025-04-9p4hloi}} \\
\hline \boldmath{$0.524$} & \makecell{\texttt{https://notebookarchive.org/2025-04-48d2pum} \\ \texttt{https://notebookarchive.org/2025-04-9p4m3bt} \\ \texttt{https://notebookarchive.org/2025-05-2swdynw}} \\
\hline \boldmath{$\operatorname{UB}(\theta)$} & \textbf{Code} \\
\hline \boldmath{$0.520$} & \makecell{\texttt{https://notebookarchive.org/2025-04-3pjlqxi} \\ \texttt{https://notebookarchive.org/2025-04-9p1qy87}} \\
\hline \boldmath{$0.521$} & \makecell{\texttt{https://notebookarchive.org/2025-04-3pk0dpg} \\ \texttt{https://notebookarchive.org/2025-04-9p1va6j}} \\
\hline \boldmath{$0.522$} & \makecell{\texttt{https://notebookarchive.org/2025-04-3pmmkfv} \\ \texttt{https://notebookarchive.org/2025-04-9p4dl87}} \\
\hline \boldmath{$0.523$} & \makecell{\texttt{https://notebookarchive.org/2025-04-48cy9rw} \\ \texttt{https://notebookarchive.org/2025-04-9p4jsrs}} \\
\hline \boldmath{$0.524$} & \makecell{\texttt{https://notebookarchive.org/2025-04-48d5k5h} \\ \texttt{https://notebookarchive.org/2025-04-9p4o9kp}} \\
\hline
\end{tabular}
\\
\textbf{Table 6: Mathematica code websites}
\end{center}

\newpage





\bibliographystyle{plain}
\bibliography{bib}
\end{document}